\documentclass[11pt]{amsart}
\usepackage[utf8]{inputenc}
\usepackage[english]{babel}
\usepackage{graphicx}
\usepackage{amsmath}
\usepackage{amssymb}
\usepackage{amsfonts}
\usepackage{amsthm}
\usepackage[left=2.1cm,top=1.5cm,right=2.1cm, bottom=2.2cm,letterpaper]{geometry}
\usepackage{mathrsfs}
\usepackage{caption}
\usepackage{subcaption}
\usepackage{latexsym}
\usepackage{xfrac}
\usepackage{enumitem}
\usepackage{float}
\usepackage{hyperref}
\usepackage{tikz}
\usepackage{esint}
 \usepackage{url}
\usepackage{mathtools}

\newtheorem{theorem}{Theorem}[section]
\newtheorem{corollary}[theorem]{Corollary}
\newtheorem{lemma}[theorem]{Lemma}
\newtheorem{definition}[theorem]{Definition}
\newtheorem{proposition}[theorem]{Proposition}

\newtheorem{remark}[theorem]{Remark}
\numberwithin{equation}{section}
\numberwithin{figure}{section}

\newcommand{\N}{\mathbb{N}}
\newcommand{\R}{\mathbb{R}}
\newcommand{\abs}[1]{\left\vert #1\right\vert}
\newcommand{\norm}[2][]{\left\| #2 \right\|_{#1}}
\newcommand{\inner}[3][]{\left\langle #2 , #3\right\rangle_{#1}}
\newcommand{\supp}{\operatorname{supp}}
\newcommand{\nablasym}{{\varepsilon}}
\newcommand{\weakto}{\rightharpoonup}
\newcommand{\weaktostar}{\weakto^*}
\providecommand{\no}[1]{  \lVert  #1  \rVert }
\newcommand{\rn}{{\mathbb R}^n}

\DeclareMathOperator{\Lip}{Lip}

\makeatletter
\renewcommand*\env@matrix[1][*\c@MaxMatrixCols c]{%
	\hskip -\arraycolsep
	\let\@ifnextchar\new@ifnextchar
	\array{#1}}
\makeatother

\newcommand{\Bcal}{{\mathcal{B}}}
\newcommand{\ext}[1]{\mathbb{E}^{{#1}}}
\newcommand{\lam}[1]{\lambda^{{#1}}}

\usepackage[square,sort&compress,comma,numbers]{natbib}

\title{Unrestricted deformations of thin elastic structures interacting with fluids}%

\author{M.~Kampschulte, S.~Schwarzacher, G.~Sperone}
\begin{document}

\begin{abstract}
 In this paper we discuss the motion of a beam in interaction with fluids. We allow the beam to move freely in all coordinate directions. We consider the case of a beam situated in between two different fluids as well as the case where the beam is attached only to one fluid. In both cases the fluid-domain is time changing. The fluid is governed by the incompressible Navier-Stokes equations. The beam is elastic and governed by a hyperbolic partial differential equation. In order to allow for large deformations the elastic potential of the beam is non-quadratic and naturally possesses a non-convex state space. We derive the existence of weak-solutions up to the point of a potential collision. 
\end{abstract}

\maketitle

\section{Introduction}

The field of fluid-structure interaction (FSI) consists of a multitude of vastly different problems, from rigid bodies slowly floating in a container, to the study of fast oscillations of bridges swinging in the wind. Applications are many and the research of the topic is extensive and far reaching. See~\cite{pironneau1994optimal,Quarteroni2000,
kamakoti2004fluid,FSIforBIO,canic2020moving} for some applications and their relation to mathematical modelling. Perhaps one of the most important mathematical benchmarks comes in the form of a one-dimensional elastic beam on top of a two-dimensional fluid reservoir (see Figure \ref{sheet2} for a sketch of that setting) and some of the most fundamental papers in the field have dealt with this particular topic. See~\cite{canic2020moving} and the references there for an overview, see also the recent advances~\cite{GraHil16,CasGraHil20}. The attention spent to this type of problems is by no means surprising, considering that any realistic solid will be at least somewhat deformable and dealing with a thin, lower-dimensional object focuses the interaction to its most essential feature, the interaction-interface. It also relates to fundamental applications of fluid flows through tubes (like blood-flow) where large deformations are known to have a most sincere effect on the dynamics~\cite{FSIforBIO,DEGLT,canic2020moving}.

In fluid-structure interactions more elaborate settings have been investigated. An ever-increasing effort has been granted the study of elastic shells in contact with a three-dimensional fluid. Relevant here is the study on weak solutions. It includes (non-linear) Koiter-shells interacting with incompressible Navier-Stokes equations~\cite{DE,DEGLT,muha2013nonlinear,
LenRuz14,BorSun15,MuhSch20}, their interaction with non-Newtonian fluids~\cite{Len14}, compressible fluids~\cite{BreSch18} and heat-conducting fluids~\cite{BreSch20,Srdan}. 

In all that effort, however the deformation is given as a graph and the interaction is restricted to happen in scalar form in a fixed direction (usually the outer normal to the reference domain) only. The so-called ``tangential interaction'' has mostly been excluded in the study of weak-solutions. In case of the beam on top of a fluid, the beam is %
commonly allowed to move vertically only (see~\cite{GraHil16,GraHil19} and the references there). While this is indeed the direction in which most of the movement is expected, it is nevertheless as artificial as restricting the solid to rigid body deformations in other FSI problems. Concerning the literature on weak solutions, the only available work seems to be~\cite{CanMuh20}. Here, however the large deformation regime is excluded as compression in tangential direction is not penalized substantially enough.

The aim of this paper is to open up the study of thin elastic structures interacting with fluids, allowing the interaction to happen at all times, at all points of the interaction surface and in all directions; including potentially large deformations of the structure (See Theorems \ref{maintheorem} and \ref{maintheorem2}). 

For that innovation we focus on a beam interacting with the two dimensional incompressible Navier-Stokes equations (as is sketched in Figure \ref{sheet2}). %
We believe it to be a good benchmark as a lot of the interesting phenomena already appear in the case of a beam, while at the same time it introduces as little as possible unrelated difficulties. Thus a thorough understanding of this setup involving large deformations and tangential interactions can serve as a crucial guide to the study of more complex situations. Indeed, looking back in history one discovers that many methods developed for the beam transfer to the more general cases, once their details are fully understood.

\begin{figure}[H]
	\centering
	\begin{tikzpicture}
	 \draw (-7,2.5) -- (-2,2.5);
	 \node[below] at (-7,2.5) {$0$};
	 \node[below] at (-2,2.5) {$\ell$};
	 \draw (0,2.5) -- (0,0) -- (5,0) -- (5,2.5);
	 \draw (0,2.5) .. controls +(2,.5) and +(-2,-.5) .. (5,2.5);
	 \draw[->] (-4.5,2.7) .. controls +(2,1) and +(-2, 1) .. node[above] {$\eta(t)$} (2.5,2.7);
	 \node at (2.5,1) {Fluid domain at time $t$};
	\end{tikzpicture}

	\caption{Fluid domain restricted by an elastic beam parametrized by $\eta(t):[0,\ell]\to \R^2$.}\label{sheet2}
\end{figure}
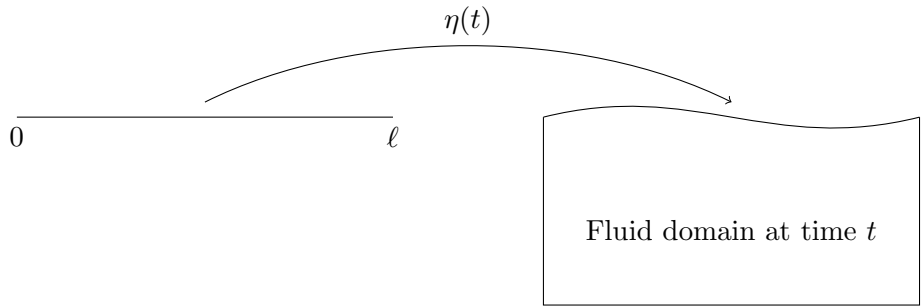
As said, we decided to take a rather specific PDE for the elastic structure.\footnote{See~\eqref{energyfunc} for our particular choice of elastic potential. See also Section \ref{ssec:outlook} where we relate our model to other models.} The main features (each relating to a substantial mathematical challenge)  of the beam equation considered here are:
\begin{enumerate}
\item The beam may deform both horizontally and tangentially. The fluid and the solid velocity are coupled with respect to both coordinate directions.
\item The beam may deform largely. This regime, which has been almost unexplored in the mathematical theory of thin structures so far, includes the penalization of compression (the elastic energy needs to go to infinity with the deformation map becoming degenerate). It includes non-linearities with negative powers.
\item The beam can be purely elastic. We allow for existence of solids that themselves are not viscous. Mathematically this means that the respective PDE for the beam alone is hyperbolic.\footnote{It is noteworthy that purely elastic solids have to be excluded in the existence theory for non measure-valued bulk solids interacting with fluids, as there viscosity is needed to stabilize the evolution of the solid in its interior, away from the fluid.}~\cite{BenKamSch20}.
\end{enumerate}

What is unavoidable to ensure injectivity a-priori, is to include a small additional higher order term. While this can be seen as a type of mathematical relaxation, these so-called second order materials are also studied directly in the mathematical modelling of more complex solid materials.

In this paper we study two different setups. The first is a beam immersed between two potentially different fluids~(Fig. \ref{sheet1}). The second is a fluid interacting with a beam at its boundary~(Fig. \ref{sheet2}). For both settings we show the existence of weak solutions to the problem for arbitrary, well posed initial data; the time-interval of existence can be prolonged provided the boundary of the fluid is not degenerating. Furthermore the setting guarantees a minimal interval of existence (see Lemma \ref{lem:nocol}). On a technical level, the second setting is derived via a limit passage of the first one, by ``removing one of the two fluids''. 

The theory of largely deforming solids involving tangential interactions we develop here is possible due to the overcoming of two major technical {\em challenges} to construct weak solutions. These challenges are (1) to obtain an approximation of the solution and (2) to show the strong convergence of the fluid-velocity by some argument of Aubin-Lions type. 

While (2) is a well known challenge in the framework of weak solutions of fluid-structure interactions (in particular involving hyperbolic structure equations)~\cite{CSS1,CSS2}, 
(1) is a peculiar problem when large deformations of solids are concerned~\cite{BenKamSch20}. Indeed, to date, the only available strategy to produce injective approximants in large deformation models of solids is a variational approach, based on a minimizing movements scheme developed in \cite{BenKamSch20}. It allows us to treat the problem directly from an energetic point of view. This strongly deviates from the more common approaches to FSI-problems, which usually involve a so called ``arbitrary Lagrangian-Eulerian map'', maps that depend strongly on the problem geometry and thus approaches that grow in complexity and difficulty with the number of directions available to and size of deformations. %

We thus emphasize that the reason for the usual restrictions of the normal direction for the interaction is mostly mathematical in nature. What is perhaps the key difficulty of fluid-structure interactions is the mismatch in description between fluid and solid. While there have been some attempts to change this paradigm, fundamentally the only practical way of describing an elastic solid is in a Lagrangian fashion, using a fixed reference configuration. On the other hand a full description of a fluid in such a way is fundamentally impossible for anything other than very short times and ultimately there is no way to avoid at least an Eulerian aspect to them.

As a result, the key problem is the injectivity of the solid deformation. This has already been a traditional concern in solid mechanics~\cite{CiaNec87,
healeyInjectiveWeakSolutions2009}. There, mathematically, injectivity is an {\em optional condition}. (One can for example construct solutions and then later verify that they are injective). In contrast, when dealing with fluid structure interaction, the injectivity is a-priori necessary to formulate the coupling with the fluid. As such it needs to be guaranteed at all points in the proof, even when just constructing approximative solutions.

Now this relates back to the problem at hand, by noting that in the classical spaces available for a-priori estimates, if one allows tangential movement, then there are {\em non-injective deformations in every neighbourhood of any injective configuration} (for instance by doubling back on a tiny interval).

Even though the construction seems only possible due to the recent advances for large deformation solids~\cite{BenKamSch20}, the technical highlight of the present paper is (2), namely the strong convergence of the fluid velocity. Here it was not possible to rely on methods developed in \cite{BenKamSch20} as we consider {\em hyperbolic} beam equations. Hence, due to the non-linear coupling, the dissipative influence of the fluid on the motion of the beam has to be used. This in general is a difficult task~\cite{LenRuz14,muha2013nonlinear,MuhSch20}. We can to some extent rely on ideas developed in these references (in particular~\cite{MuhSch20}), but in order to apply this here we need to invent substantially new analytic tools. We briefly summarize our efforts, which also might be of independent value:
\begin{itemize}
\item {\bf (Almost) solenoidal extensions of vector fields.} We introduce appropriate extension operators of given boundary values in vectors including tangential directions. This is the technical core of the construction. (See Proposition~\ref{prop:arbMeanExt}).
\item {\bf Universal Bogovskij operators.} In order to make functions solenoidal, we use so-called Bogovskij operators. We introduce here Bogovskij operators that act on all domains with certain restrictions (see Theorem~\ref{theobog}). The provided construction here allows for very precise dependencies on the potential degeneration of the fluid-domain.
\item \textbf{Differential pressure across the interface.} Instead of just focusing on divergence free test-functions and ignoring the effects of the pressure during the approximations, at each step we additionally construct the difference in total pressures as a single time-dependent scalar. This allows us to test the beam with test-functions which cannot be extended into divergence-free vector-fields, without at the same time having to deal with the bad compactness-properties of the entire pressure (See \eqref{eq:press1} and \eqref{eq:longTimedelayed-pressure}).
\end{itemize}
All three points are interrelated. In particular, the last two are used in the first one, which seems to be irreplaceable for two known critical challenges in FSI-analysis: The $L^2$-compactness of the fluid velocity and the establishment of a limit PDE (see Section~5).

 The last point seems to be the key missing ingredient in some of the previous literature and as such might have great potential for further applications. Essentially, by breaking with the standard paradigm in the analysis of weak solutions for FSI, of keeping the pressure implicit when constructing weak solutions, we are able to construct them for fluid-structure interactions {\em involving tangential motions of the structure}. The idea here was motivated by the (partial) pressure reconstructions (and decompositions)~\cite{DieWolf10} that were successfully used for the existence analysis for non-Newtonian fluids.

The layout of this paper is as follows: In Section 2 we will give a precise mathematical description of the fluid-structure interaction problem we study and state the main results, Theorems \ref{maintheorem} and \ref{maintheorem2}. Section 3 is the technical core of the paper where we will provide the universal Bogovskij operator and the solenoidal extension. Section 4 will consist of constructing an approximative solution by variational methods. In Section 5 we will proceed with the limit passage. Most of this section is dedicated to an Aubin-Lions type argument showing compactness of the fluid velocities. This finishes the proof of Theorem \ref{maintheorem}. Section 6 will be devoted to the proof of Theorem \ref{maintheorem2}, by ``removing one of the two fluids'' via another limit. Some outlook how to extend the here developed methodology to non-linear shell-like structures concludes the paper in Section~7. 

\section{Concept of solutions and main results} 

\subsection{General description of the problem}
In order to have an easy to follow setup, we will study the FSI-problem on the square domain
\[
\Omega := (0,\ell) \times \left( -\dfrac{\ell}{2}, \dfrac{\ell}{2} \right) \text{ for }\ell > 0. \] 
which will be intersected by a solid in the form of a curve connecting the points $(0,0)$ and $(\ell,0)$.

Here $\Omega$ represents a two-dimensional cavity bisected in its middle section by a one-dimensional viscoelastic beam, corresponding to a simplified one-dimensional analogue of a linear viscoelastic Koiter shell model \cite{K1}. The container is filled with two fluids in motion that deform the beam separating them, allowing for both normal and tangential displacements of the beam, whereas the remaining walls of the cavity remain fixed in time, see Figure \ref{sheet1}. 

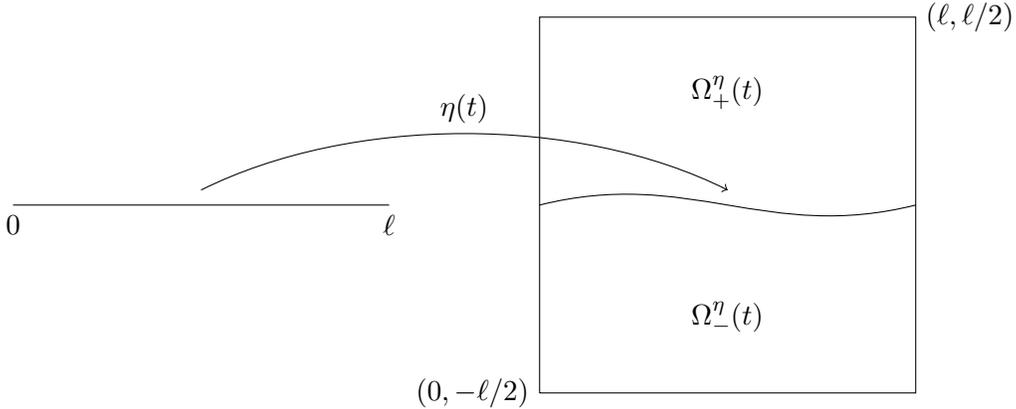
\begin{figure}[H]
	\centering
	\begin{tikzpicture}
	 \draw (-7,2.5) -- (-2,2.5);
	 \node[below] at (-7,2.5) {$0$};
	 \node[below] at (-2,2.5) {$\ell$};
	 \draw (0,0) -- (5,0) -- (5,5) -- (0,5) -- (0,0);
	 \node[left] at (0,0) {$(0,-\ell/2)$};
	 \node[right] at (5,5) {$(\ell,\ell/2)$};
	 \draw (0,2.5) .. controls +(2,.5) and +(-2,-.5) .. (5,2.5);
	 \draw[->] (-4.5,2.7) .. controls +(2,1) and +(-2, 1) .. node[above] {$\eta(t)$} (2.5,2.7);
	 \node at (2.5,1) {$\Omega^{\eta}_-(t)$};
	 \node at (2.5,4) {$\Omega^{\eta}_+(t)$};
	\end{tikzpicture}

	\caption{Two fluid domains $\Omega^\eta_{+}(t)$ and $\Omega^\eta_{-}(t)$ separated by the elastic beam parametrized through $\eta$.}\label{sheet1}
\end{figure}

The fluids and the structure are fully coupled through two coupling conditions: The no-slip boundary condition and the dynamic coupling condition describing the
balance of forces at the structure interface. The flows are driven by initial and right hand side data, which includes
a time-dependent external forcing defined over the fixed domain $\Omega$. %

The time-deformation of the beam will be described by  $\eta : [0,T) \times [0,\ell] \longrightarrow \Omega$, where $[0,\ell]$ serves as a reference configuration. This elastic beam deforms, splitting $\Omega$ into two time-dependent fluid domains which are not known \textit{a priori}. We denote those by $\Omega^{\eta}_{+}(t)$ and $\Omega^{\eta}_{-}(t)$ at each time $t \in [0,T)$.

Accordingly, %
the tangent vector to the deformed structure will be denoted by $\tau(x,t) = \partial_{x} \eta(t,x)$, and the outer
unit normals at the point $\eta(x,t)$ will be denoted by $\nu_{+}(x,t)$ and $\nu_{-}(t,x) = - \nu_{+}(t,x)$. Notice that, for every $t \in [0,T)$ and $x \in (0,\ell)$, $\nu_{+}(t,x)$ is directed to the interior of $\Omega_{-}^{\eta}(t)$ and that $\nu_{-}(t,x)$ is directed to the interior of $\Omega_{+}^{\eta}(t)$. 

The \textbf{fluids} in the two halves of the domain are both assumed to be two-dimensional, homogeneous, viscous, incompressible and Newtonian. They are however allowed to have different densities and viscosities. Their velocity fields $u_{+}$, $u_{-}$ and scalar pressures $p_{+}$, $p_{-}$ satisfy the incompressible Navier-Stokes equations in $\Omega_{+}^{\eta}(t)$ and $\Omega_{-}^{\eta}(t)$, respectively, for $t \in [0,T)$: 
\begin{align} \label{nstokes0}
\rho_{\pm} \left( \partial_{t} u_{\pm} + u_{\pm} \cdot \nabla u_{\pm} \right) - \nabla \cdot \sigma_{\pm} (u_\pm,p_\pm)= \rho_{\pm} f,\ \qquad  \nabla\cdot u_{\pm}=0\qquad\mbox{in} \ \ \Omega_{\pm}^{\eta}(t).
\end{align}
Here and in the future we use the index $\pm$ as a shorthand to consistently denote either $+$ or $-$ as an index, when dealing with otherwise identical equations for the two parts of the domain.

In the above equation $\rho_{\pm} > 0$ denote the (constant) densities of the fluids and $\sigma_{\pm} (u_\pm,p_\pm)$ are the stress tensors of the fluids given by the Newtonian law:
\begin{align*}
\sigma_{\pm}(u_\pm,p_\pm) := 2 \mu_{\pm} \nablasym(u_{\pm}) - p_{\pm} \mathbb{I}_{2} = \mu_{\pm} \left[ \nabla u_{\pm} + (\nabla u_{\pm})^{\top} \right] - p_{\pm} \mathbb{I}_{2}
\end{align*}
with $\mu_{\pm}> 0$ the (constant) kinematic viscosities of the fluids and where $\nablasym(\pm u)$ denotes the symmetrized gradient of $u_{\pm}$.

On the right-hand side of \eqref{nstokes0}, $f : [0,T) \times \Omega \longrightarrow \mathbb{R}^2$ represents an external body force (for example, gravity) acting over each of the fluids. To complete the formulation of the problem, we impose the standard no-slip boundary condition on the rigid walls of $\partial \Omega$:
\begin{align} \label{bc0}
u_{\pm} = 0 \qquad \text{on} \ \ \ [0,T) \times \partial \Omega,
\end{align}
the initial condition
\begin{align} \label{ic0}
u_{\pm}(0,z) = u_0(z) \ \ \ \text{ for all } z \in \Omega_{\pm}(0) \, ,
\end{align}
as well as a coupling with the solid that will be described later.

$\bullet$ The \textbf{elastodynamics of the solid} will be given in terms of displacement with respect to the its parametrization over $[0,\ell]$. Please see the Section \ref{ssec:outlook} for a formal derivation of the model.
 The strong formulation of the equations is given in terms of a differential operator $\mathcal{L}_{\Gamma}$ which can be obtained as the derivative of the elastic energy $\mathcal{E}_{K}(\eta)$ of the beam, as well as additional dissipation and inertial terms. For the model case we are treating here, this energy includes the Koiter-type contributions due to stretching and bending, but also importantly, terms that \textit{penalize} compression in order to avoid degeneracy of the resulting interface. More precisely, we define
\begin{align} \label{energyfunc}
\mathcal{E}_{K}(\eta) := \int_0^{\ell} \left( \dfrac{c_1}{2} \abs{\partial_x \eta_1-1}^2 + \frac{1}{\abs{\partial_x \eta_1}^{2\alpha}} + \dfrac{\lambda_0}{2} \abs{\partial_{xx} \eta_1}^2 + \dfrac{c_2}{2} \abs{\partial_{xx} \eta_2}^2  \right)dx \, ,
\end{align}
for given constants $c_{1}, c_{2}, \lambda_{0} > 0$ and $\alpha > 1$, so that 
\begin{align*}
\mathcal{L}_{\Gamma}(\eta) = \dfrac{\partial^{4}}{\partial x^{4}} \left[\begin{matrix} \lambda_{0} \eta_{1} \\[6pt] c_{2} \eta_{2} \end{matrix}\right] - c_{1} \dfrac{\partial^{2}}{\partial x^{2}} \left[\begin{matrix} \eta_{1} \\[6pt] 0 \end{matrix}\right] - 2 \alpha \dfrac{\partial}{\partial x} \left[\begin{matrix}  \dfrac{\partial_{x}\eta_{1}}{| \partial_{x} \eta_{1} |^{2(\alpha + 1)}} \\  0 \end{matrix}\right].
\end{align*}
The structure elastodynamics problem is then given by

\begin{equation} \label{movingbeam0}
\left\lbrace
\begin{aligned}
& \rho_{s}  \partial_{tt} \eta + \mathcal{L}_{\Gamma}(\eta) = g \qquad & &\text{in} \ \ [0,T) \times (0,\ell), \\[6pt]
& \eta(t,0) = (0,0), \quad \eta(t,\ell) = (\ell,0), \quad \partial_{x} \eta(t,0) = \partial_{x} \eta(t,\ell) = (1,0) &&\text{for all } t \in [0,T) , \\[6pt]
& \eta(0,x) = \eta_{0}(x) \ \ \text{and} \ \ \partial_{t} \eta(0,x) = \eta^*(x) &&\text{for all } x \in [0,\ell],
\end{aligned}
\right.	
\end{equation}
where $\rho_{s} > 0$ denotes the density of the structure and $g$ represents the momentum forces of the fluid acting on $\eta(t)$. Finally, $\eta_{0}$ and $\eta^*$ are the initial structure deformation and the initial structure velocity, respectively.

$\bullet$ The fluid and the structure are linked via \textbf{kinematic and dynamic coupling conditions}. We prescribe the no-slip kinematic coupling condition which means that the fluid and the structure velocities are equal on the elastic boundary: 
\begin{align} \label{coupling1}
u_{+}(t,\eta(t,x)) = u_{-}(t,\eta(t,x)) = \partial_{t} \eta(t,x) \text{ for all } (t,x) \in [0,T) \times (0,\ell).
\end{align}

On the other hand, the dynamic boundary condition states that the total external force on the solid is precisely the force exerted by the fluid:
\begin{align} \label{coupling2}
g(t,x) &= - \abs{\partial_x\eta(t,x)} \left[ \sigma_{+}(u_+,p_+)(t,\eta(t,x)) \, \nu_{+}(t,x) + \sigma_{-}(u_-,p_-)(t,\eta(t,x)) \, \nu_{-}(t,x) \right] \nonumber \\
 &= -\sigma_{+}(u_+,p_+)(t,\eta(t,x)) \, \partial_x \eta^\bot(t,x) + \sigma_{-}(u_-,p_-)(t,\eta(t,x)) \, \partial_x \eta^\bot(t,x), 
\end{align}
for every $(t,x) \in [0,T) \times (0,\ell)$. Relation \eqref{coupling2} states that the structure is driven by the forces resulting from the fluid stress across the interface. Here the term $\abs{\partial_x\eta}$, which multiplies the normal fluid stress $\sigma_\pm(u_\pm,p_\pm) \nu_\pm$, is the line element resulting from the Lagrangian representation of the structure problem. Since $\partial_x \eta$ always points in tangential direction and is of the same magnitude, we use the perpendicular function $\partial_x\eta^\bot$ with $(v_1,v_2)^\bot := (-v_2,v_1)^T$ to simplify notation here.

\subsection{Weak coupled solutions and energy (in)equality} \label{weakcoupledsol}
We use the standard notation of Bochner spaces related to Lebesgue and Sobolev spaces. In order to define the weak solutions of the coupled system \eqref{nstokes0}-\eqref{coupling2}, let us first introduce the appropriate function spaces related to the structure
\begin{align*}
V_{K} &= L^{\infty}(0,T;H_{0}^{2}((0,\ell);\mathbb{R}^{2})) \cap W^{1,\infty}(0,T;L^{2}((0,\ell);\mathbb{R}^{2})), \\
\mathcal{V}_{K} &:=\{ \eta \in H^{2}((0,\ell);\mathbb{R}^{2})) \ | \ (t,x)\mapsto \eta(t,x)-(x,0) \in V_K, \ \ \mathcal{E}_K(\eta(t))<\infty \, \text{ a.e. on } \, t \in [0,T]\} \, ,
\end{align*} 
the spaces associated to the fluids:
\begin{align*}
 V_{\pm}^{\eta}(t) &= \left\{ v \in H^{1}(\Omega_{\pm}^{\eta}(t);\mathbb{R}^2) \ | \ \nabla \cdot v(t,\cdot) = 0 \ \ \text{in} \ \ \Omega_{\pm}^{\eta}(t), \quad v=0 \ \ \text{on} \ \  [0,T) \times \partial\Omega \right\}, \\
 V_{\pm}^\eta &= L^{\infty}(0,T;L^{2}(\Omega_{\pm}^{\eta}(t);\mathbb{R}^2)) \cap L^{2}(0,T;V_{\pm}^{\eta}(t)),
\end{align*}
and the spaces for the solution pair and for the test functions, respectively:
\begin{align*}
V_{S} &=\left\{ (u_{\pm},\eta) \in V_{\pm}^\eta \times \mathcal{V}_{K} \ | \ u_{\pm}(t,\eta(t,x)) = \partial_{t} \eta(t,x) \ \ \ \text{ for all } (t,x) \in [0,T) \times [0,\ell] \right\},\\
V^{\eta}_{T} &= \left\lbrace (q_{\pm}, \xi) \in V_{\pm}^\eta \times V_{K} \ \Bigg\rvert \ 
\begin{aligned}
& q_{\pm}(t,\eta(t,x)) = \xi(t,x) \ \ \ \text{ for all } (t,x) \in [0,T) \times [0,\ell], \\ 
& \partial_{t}q_{\pm} \in L^{2}(0,T;L^{2}(\Omega_{\pm}^{\eta}(t);\mathbb{R}^2)) \}
\end{aligned}
\right\rbrace
\end{align*}
Here $\Omega^\eta_\pm(t)$ always denotes the two connected components of $\Omega \setminus \eta(t,(0,\ell))$.

The space of solutions in the case when only one fluid is considered in $\Omega^\eta_-$ is defined as
\begin{align*}
V_{S}^-&=\left\{ (u_{-},\eta) \in V_{-}^{\eta} \times \mathcal{V}_{K} \ | \ u_{-}(t,\eta(t,x)) = \partial_{t} \eta(t,x) \ \ \ \text{ for all } (t,x) \in [0,T) \times [0,\ell] \right\},
\end{align*}

It should be noted that without additional information, provided for example by the energy, for a given deformation $\eta \in \mathcal{V}_{K}$ the variable domain $\Omega^{\eta}(t)$ may not have a Lipschitz boundary. Nevertheless, the traces used in the definition of $V_{S}$ and $V_{T}$  are still well defined, see for instance \cite{CDEM,Muh14}.

Given $f \in L^{2}([0,T) \times \Omega;\mathbb{R}^2)$, consider a smooth solution $(u_{\pm}, \eta) \in V_{S}$ of the system \eqref{nstokes0}-\eqref{movingbeam0}, with an associated scalar pressure $p_{\pm} \in L^{2}(0,T;L^{2}(\Omega_{\pm}^{\eta}(t);\mathbb{R}))$, and a pair of test functions $(q_{\pm}, \xi) \in V_{T}^\eta$. First multiply the equation of momentum conservation in \eqref{nstokes0}$_{1}$ by $q_{+}$ and integrate in $\Omega_{+}^{\eta}(t)$, for any $t \in [0,T)$. Then, since the elastic beam deforms itself with velocity equal to $u_{+}$, repeated integration by parts and the Reynolds Transport Theorem imply that
\begin{equation} \label{byparts1}
\begin{aligned}
& \int_{\Omega_{+}^{\eta}(t)} \left[ \rho_{+} \left( \partial_{t} u_{+} + (u_{+} \cdot \nabla) u_{+} \right) - \nabla \cdot \sigma_{+} (u_+,p_+) - \rho_{+} f \right] \cdot q_{+} \, dz \\[6pt]
& \hspace{-5mm} = \rho_{+} \left( \dfrac{d}{dt} \int_{\Omega_{+}^{\eta}(t)} u_{+} \cdot q_{+} \, dz - \int_{\Omega_{+}^{\eta}(t)} \left[ u_{+} \cdot \partial_{t} q_{+} + \left(u_{+} \otimes u_{+} \right) \cdot \nabla q_{+} + f \cdot q_{+} \right] dz \right) \\[6pt]
& + \mu_{+} \int_{\Omega_{+}^{\eta}(t)} \nabla u_{+} \cdot \nabla q_{+} \, dz 
+ \int_{\eta(t)} \left[ p_{+} (q_{+} \cdot \nu_{+}) - \mu_{+}(\nabla u_{+}) \nu_{+} \cdot q_{+} \right] dz,
\end{aligned}
\end{equation}
and in the exact same way:
\begin{equation} \label{byparts1-}
\begin{aligned}
&\phantom{{}={}}  \int_{\Omega_{-}^{\eta}(t)} \left[ \rho_{-} \left( \partial_{t} u_{-} + (u_{-} \cdot \nabla) u_{-} \right) - \nabla \cdot \sigma_{-} (u_-,p_-) - \rho_{-} f \right] \cdot q_{-} \, dz \\
&  = \rho_{-} \left( \dfrac{d}{dt} \int_{\Omega_{-}^{\eta}(t)} u_{-} \cdot q_{-} \, dz - \int_{\Omega_{-}^{\eta}(t)} \left[ u_{-} \cdot \partial_{t} q_{-} + \left(u_{-} \otimes u_{-} \right) \cdot \nabla q_{-} + f \cdot q_{-} \right] dz \right) \\
&\phantom{{}={}}  + \mu_{-} \int_{\Omega_{-}^{\eta}(t)} \nabla u_{-} \cdot \nabla q_{-} \, dz 
+ \int_{\eta(t)} \left[ p_{-} (q_{-} \cdot \nu_{-}) - \mu_{-}(\nabla u_{-}) \nu_{-} \cdot q_{-} \right] dz.
\end{aligned}
\end{equation}
Secondly, multiply the dynamic equation of the beam \eqref{movingbeam0}$_{1}$ by $\xi$, and integrate in $[0,\ell]$ to obtain:
\begin{equation} \label{byparts2}
\begin{aligned}
& \phantom{{}={}}   \int_{0}^{\ell} \left[ \rho_{s} \partial_{tt} \eta  %
+ \mathcal{L}_{\Gamma}(\eta) - g \right] \cdot \xi \, dx \\
&  = \rho_{s} \left( \dfrac{d}{dt} \int_{0}^{\ell} \partial_{t} \eta \cdot \xi \, dx - \int_{0}^{\ell} \partial_{t} \eta \cdot \partial_{t} \xi \, dx \right)
+ c_{2} \int_{0}^{\ell} \partial_{xx} \eta_{2} \partial_{xx} \xi_{2} \, dx  \\
& \phantom{{}={}}  + \int_{0}^{\ell} \left( \lambda_{0} \partial_{xx} \eta_{1} \partial_{xx} \xi_{1} + c_{1} \partial_{x} \eta_{1} \partial_{x} \xi_{1} \right) dx 
+ 2\alpha \int_{0}^{\ell} \dfrac{\partial_{x} \eta_{1} \partial_{x} \xi_{1}}{| \partial_{x} \eta_{1} |^{2(\alpha + 1)}} \, dx - \int_{0}^{\ell} g  \cdot \xi \, dx 
\end{aligned}
\end{equation}
The term involving $g \cdot \xi$ in \eqref{byparts2} deserves special attention, because it directly expresses the interaction between the fluids and the beam. Indeed, by the Divergence Theorem we have 
\begin{equation} \label{byparts3}
\begin{aligned}
- \int_{0}^{\ell} g  \cdot \xi \, dx & = - \int_{0}^{\ell}  p_{+}  (\xi \cdot \partial_{x}\eta(x,t)^\bot) \, dx + 2 \mu_{+} \int_{0}^{\ell} \left(\nablasym(u_{+}) \partial_{x}\eta(x,t)^\bot \right) \cdot \xi \, dx \\
& \phantom{{}={}} + \int_{0}^{\ell}  p_{-} \, (\xi \cdot \partial_{x}\eta(x,t)^\bot) \, dx - 2\mu_{-} \int_{0}^{\ell} \left(\nablasym(u_{-}) \partial_{x}\eta(x,t)^\bot \right)  \cdot \xi \, dx \\
& = - \int_{\eta(t)} \left[ p_{+} (q_{+} \cdot \nu_{+}) - \mu_{+}(\nabla u_{+}) \nu_{+} \cdot q_{+} \right] dz + \mu_{+} \int_{\Omega_{+}^{\eta}(t)} (\nabla u_{+})^{\top} \cdot \nabla q_{+} \, dz \\
& \phantom{{}={}} - \int_{\eta(t)} \left[ p_{-} (q_{-} \cdot \nu_{-}) - \mu_{-}(\nabla u_{-}) \nu_{-} \cdot q_{-} \right] dz + \mu_{-} \int_{\Omega_{-}^{\eta}(t)} (\nabla u_{-})^{\top} \cdot \nabla q_{-} \, dz.
\end{aligned}
\end{equation}
Recall that we use a Lagrangian description for the deformation of the beam, but in order to simplify notation, we have omitted the dependence on $(\eta(t,x),t)$ of the integrands over $[0,\ell]$. In identities \eqref{byparts1}-\eqref{byparts2}-\eqref{byparts3} we have repeatedly used the coupling condition \eqref{coupling1}, which is also included in the test space $V^{\eta}_{T}$. Therefore, by inserting \eqref{byparts3} into \eqref{byparts2} and adding this to \eqref{byparts1}-\eqref{byparts1-} we obtain the following weak formulation for the fluid-structure interaction problem:
\begin{definition}
\label{def:weak}
Given $f \in L^{2}([0,T) \times \Omega;\mathbb{R}^2)$ and suitable initial data, we say that $(u_{\pm}, \eta) \in V_{S}$ is a weak solution to  the equations \eqref{nstokes0}-\eqref{movingbeam0}, complemented with boundary conditions \eqref{bc0}, coupling conditions \eqref{coupling1}-\eqref{coupling2} and initial conditions \eqref{ic0}, if it is satisfying the identity
\begin{align} \label{weak0}
& \phantom{{}+{}} \rho_{+} \left( \dfrac{d}{dt} \int_{\Omega_{+}^{\eta}(t)} u_{+} \cdot q_{+} \, dz - \int_{\Omega_{+}^{\eta}(t)} \left[ u_{+} \cdot \partial_{t} q_{+} + \left(u_{+} \otimes u_{+} \right) \cdot \nabla q_{+} \right] dz \right) \nonumber \\
&+  \rho_{-} \left( \dfrac{d}{dt} \int_{\Omega_{-}^{\eta}(t)} u_{-} \cdot q_{-} \, dz - \int_{\Omega_{-}^{\eta}(t)} \left[ u_{-} \cdot \partial_{t} q_{-} + \left(u_{-} \otimes u_{-} \right) \cdot \nabla q_{-} \right] dz \right) \nonumber \\
&+  2\mu_{+} \int_{\Omega_{+}^{\eta}(t)} \nablasym(u_{+}) \cdot \nabla q_{+} \, dz + 2\mu_{-} \int_{\Omega_{-}^{\eta}(t)} \nablasym(u_{-}) \cdot \nabla q_{-} \, dz
\\
&+  \rho_{s} \left( \dfrac{d}{dt} \int_{0}^{\ell} \partial_{t} \eta \cdot \xi \, dx - \int_{0}^{\ell} \partial_{t} \eta \cdot \partial_{t} \xi \, dx \right) + \int_{0}^{\ell} \left( \lambda_{0} \partial_{xx} \eta_{1} \partial_{xx} \xi_{1} + c_{1} \partial_{x} \eta_{1} \partial_{x} \xi_{1} \right) dx \nonumber \\
&+  c_{2} \int_{0}^{\ell} \partial_{xx} \eta_{2} \partial_{xx} \xi_{2} \, dx + 2\alpha \int_{0}^{\ell} \dfrac{\partial_{x} \eta_{1} \partial_{x} \xi_{1}}{| \partial_{x} \eta_{1} |^{2(\alpha + 1)}} \, dx = \rho_{+} \int_{\Omega_{+}^{\eta}(t)} \left( f \cdot q_{+} \right) \, dz + \rho_{-} \int_{\Omega_{-}^{\eta}(t)} \left( f \cdot q_{-} \right) \, dz \nonumber
\end{align}
for almost every $t \in [0,T)$ and for every $(q_{\pm}, \xi) \in V^{\eta}_{T}$ and if the initial conditions \eqref{ic0} and \eqref{movingbeam0}$_{3}$ are attained in the respective weakly continuous sense.
\end{definition}

If we assume enough regularity of a weak solution $(u_{\pm}, \eta) \in V_{S}$ of \eqref{weak0}, the use of no-slip boundary conditions for $u_{\pm}$ in \eqref{bc0} and the condition of clamping for the deformation $\eta$ in \eqref{movingbeam0}$_{2}$ allow us to obtain a-priori estimates. Indeed, in view of the Reynolds transport theorem, the incompressibility condition for $u_{\pm}$ and the coupling condition \eqref{coupling1} we have
\begin{align*}
 \int_{\Omega_{\pm}^{\eta}(s)} \left(\partial_{t} u_{\pm} + (u_{\pm} \cdot \nabla) u_{\pm} \right) \cdot u_{\pm}\, dz &= \dfrac{1}{2} \left( \dfrac{d}{dt} \int_{\Omega_{\pm}^{\eta}(s)} |u_{\pm}|^2 \, dz \right), \\
 \int_{\Omega_{\pm}^{\eta}(s)} \nabla \cdot \sigma_{\pm}(u,p) \cdot u_{\pm} \, dz &= - \int_{\Omega_{\pm}^{\eta}(s)} \mu_\pm| \nabla u_{\pm} |^2 \, dz - \int_{\eta(s)} \left[ p_{\pm} (u_{\pm} \cdot \nu_{\pm}) - \mu_{\pm}(\nabla u_{\pm}) \nu_{\pm} \cdot u_{\pm} \right] dz,
\end{align*}
for every $s \in [0,T)$, so that
\begin{align} \label{apriori1}
& \phantom{{}={}} \int_{\Omega_{\pm}^{\eta}(s)} \left[ \rho_{\pm} \left( \partial_{t} u_{\pm} + (u_{\pm} \cdot \nabla) u_{\pm} \right) - \nabla \cdot \sigma_{\pm} (u,p) - \rho_{\pm} f \right] \cdot u_{\pm} \, dz \nonumber \\
&  = \dfrac{\rho_{\pm}}{2} \dfrac{d}{dt} \left( \int_{\Omega_{\pm}^{\eta}(s)} |u_{\pm}|^2 \, dz \right) + \int_{\Omega_{\pm}^{\eta}(s)} \left[ \mu_{\pm} | \nabla u_{\pm} |^2 - \rho_{\pm} (f \cdot u_{\pm}) \right] \, dz \\
& \phantom{{}={}} + \int_{\eta(s)} \left[ p_{\pm} (u_{\pm} \cdot \nu_{\pm}) - \mu_{\pm}(\nabla u_{\pm}) \nu_{\pm} \cdot u_{\pm} \right] dz, \nonumber
\end{align}
for every $s \in [0,T)$. On the other hand, after multiplying the dynamic equation of the beam \eqref{movingbeam0}$_{1}$ by the velocity $\partial_{t} \eta(s,x)$ (which is identical to $u_\pm(s,\eta(s,x))$ due to the kinematic coupling) and integrating in $[0,\ell]$ we obtain:
\begin{equation} \label{apriori2}
\begin{aligned}
& \phantom{{}={}} \int_{0}^{\ell} \left[ \rho_{s} \partial_{tt}\eta
+ \mathcal{L}_{\Gamma}(\eta) - g \right] \cdot \partial_{t}\eta \, dx \\
&   = \dfrac{d}{dt} \left( \dfrac{\rho_{s}}{2}  \int_{0}^{\ell} |\partial_{t} \eta|^2 \, dx %
+ \mathcal{E}_{K}(\eta(s)) \right) \\
&\phantom{{}={}} + \mu_{+} \int_{\Omega_{+}^{\eta}(s)} \left[ \left( \nabla u_{+} \right)^{\top} \cdot \nabla u_{+} \right] \, dz + \mu_{-} \int_{\Omega_{-}^{\eta}(s)} \left[ \left( \nabla u_{-} \right)^{\top} \cdot \nabla u_{-} \right] \, dz \\
& \phantom{{}={}}- \int_{\eta(s)} \left[ p_{+} (u_{+} \cdot \nu_{+}) - \mu_{+}(\nabla u_{+}) \nu_{+} \cdot u_{+} \right] dz - \int_{\eta(s)} \left[ p_{-} (u_{-} \cdot \nu_{-}) - \mu_{-}(\nabla u_{-}) \nu_{-} \cdot u_{-} \right] dz.
\end{aligned}
\end{equation}
Therefore, we add identities \eqref{apriori1}-\eqref{apriori2} and integrate with respect to $s \in [0,t]$ to obtain the following equality:
\begin{align*}
& \phantom{{}={}}\dfrac{\rho_{+}}{2} \| u_{+}(t) \|^{2}_{L^{2}(\Omega_{+}^{\eta}(t))} + \dfrac{\rho_{-}}{2} \| u_{-}(t) \|^{2}_{L^{2}(\Omega_{-}^{\eta}(t))} + \dfrac{\mu_{+}}{2} \int_{0}^{t} \| \nablasym(u_{+})(s) \|^{2}_{L^{2}(\Omega_{+}^{\eta}(s))} \, ds \\
&\phantom{{}={}} + \dfrac{\mu_{-}}{2} \int_{0}^{t} \| \nablasym(u_{-})(s) \|^{2}_{L^{2}(\Omega_{-}^{\eta}(s))} \, ds + \dfrac{\rho_{s}}{2} \| \partial_{t} \eta(t) \|^{2}_{L^{2}(\Gamma^{\eta}(t))} 
 + \mathcal{E}_{K}(\eta(t)) \\
&  = \dfrac{\rho_{+}}{2} \| u_{0} \|^{2}_{L^{2}(\Omega_{+})} + \dfrac{\rho_{-}}{2} \| u_{0} \|^{2}_{L^{2}(\Omega_{-})} + \dfrac{\rho_{s}}{2} \| \eta^{*} \|^{2}_{L^{2}(\Gamma)}
 + \mathcal{E}_{K}(\eta_0)\\
& \phantom{{}={}}+ \rho_{+} \int_{0}^{t} \int_{\Omega_{+}^{\eta}(s)} (f \cdot u_{+}) \, dz \, ds + \rho_{-} \int_{0}^{t} \int_{\Omega_{-}^{\eta}(s)} (f \cdot u_{-}) \, dz \, ds \text{ for all } t \in [0,T),
\end{align*}
from which we derive the energy estimate
\begin{align}
\label{eq:energ}
\begin{aligned}
& \phantom{{}={}}\dfrac{\rho_{+}}{4} \| u_{+}(t) \|^{2}_{L^{2}(\Omega_{+}^{\eta}(t))} + \dfrac{\rho_{-}}{4} \| u_{-}(t) \|^{2}_{L^{2}(\Omega_{-}^{\eta}(t))} + \dfrac{\mu_{+}}{2} \int_{0}^{t} \| \nablasym(u_{+})(s) \|^{2}_{L^{2}(\Omega_{+}^{\eta}(s))} \, ds \\
&\phantom{{}={}} + \dfrac{\mu_{-}}{2} \int_{0}^{t} \| \nablasym(u_{-})(s) \|^{2}_{L^{2}(\Omega_{-}^{\eta}(s))} \, ds + \dfrac{\rho_{s}}{2} \| \partial_{t} \eta(t) \|^{2}_{L^{2}(\Gamma^{\eta}(t))}
+ \mathcal{E}_{K}(\eta(t)) \\
&  \leq \dfrac{\rho_{+}}{2} \| u_{0} \|^{2}_{L^{2}(\Omega_{+})} + \dfrac{\rho_{-}}{2} \| u_{0} \|^{2}_{L^{2}(\Omega_{-})} + \dfrac{\rho_{s}}{2} \| \eta^{*} \|^{2}_{L^{2}(\Gamma)} 
+ \mathcal{E}_{K}(\eta_0)\\ 
&\phantom{{}={}} + \rho_{+} \int_{0}^{t} \| f(s) \|^{2}_{L^{2}(\Omega_{+}^{\eta}(s))} \, ds + \rho_{-} \int_{0}^{t} \| f(s) \|^{2}_{L^{2}(\Omega_{-}^{\eta}(s))} \, ds \text{ for all } t \in [0,T).
\end{aligned}
\end{align}

\subsection{Main results}
The first main result of this article concerns the setting of Figure \ref{sheet1}: 

\begin{theorem} \label{maintheorem}
Let $\rho^\pm,\mu^\pm,\rho_s>0$. For any initial data $\eta_{0} \in  \mathcal{V}_{K}$, $u_{0} \in L^2(\Omega;\mathbb{R}^2)$  and $\eta^*\in L^2((0,\ell);\mathbb{R}^{2})$
there exists $T > 0$ and a weak solution $(u,\eta) \in V_{S}$ in the sense of Definition \ref{def:weak} in $[0,T]$. Further it satisfies the energy inequality~\eqref{eq:energ}. If $T<\infty$ and there is no collision, $T$ can be further extended.  
\end{theorem}

The second main result of this article concerns the setting of Figure \ref{sheet2}: 
\begin{theorem} \label{maintheorem2}
Let $\rho^-,\mu^-,\rho_s>0$. For any initial data $\eta_{0} \in  \mathcal{V}_{K}$, $u_{0} \in L^2(\Omega_{-}^{\eta_0};\mathbb{R}^2)$  and $\eta^*\in L^2((0,\ell);\mathbb{R}^{2})$
there exists $T > 0$ and a weak solution $(u_{-},\eta) \in V_{S}^{-}$ to \eqref{weak00} (which is the same as \eqref{weak0} with all terms related to $\Omega_+$ removed). Further it satisfies the energy inequality~\eqref{oneSidedIneq}. If $T<\infty$ and there is no collision, $T$ can be further extended. 
\end{theorem}

We will first prove Theorem \ref{maintheorem} in several steps using a time-discretization in the form of a minimizing
movements iteration. Theorem \ref{maintheorem2} is then derived in the last section of the paper by passing to the limit with $\rho_+,\mu_+\to 0$.

But before that, in the next section, we need to construct an universal Bogovkij operator which also turns out to be rather suitable to give a precise pressure reconstruction, or decomposition (compare with \cite{SaaSch21}). This is performed in \eqref{eq:final-pres}. See also Remarks \ref{rem:pres1} and \ref{rem:pres2}.

\section{Analysis for variable in time domains}

This section is the technical heart of the paper. We derive the critical tools in order to prove the main results here. In particular a universal Bogovskij operator and a quasi-solenoidal extension are developed. Before, we provide some properties for functions with finite elastic energy. These determine the Eulerian geometry and translate into uniform Lipschitz properties of the variable-in-time domains. 

\subsection{Estimates for the elastic deformation}
We begin with some further properties of the energy functional $\mathcal{E}_{K}$ as defined in \eqref{energyfunc}. 

Let us introduce the affine space of admissible deformations
\begin{equation*}
\mathcal{H} := \left\{ \eta \in H^{2}((0,\ell);\mathbb{R}^{2})) \ | \ \eta(0) = (0,0), \quad \eta(\ell)=(\ell,0), \quad \partial_{x}\eta(0) = \partial_{x}\eta(\ell) = (1,0)  \right\}\, ,
\end{equation*}
endowed with the standard $H^2$-topology. We then have:
\begin{proposition} \label{propenergy}
Let $\mathcal{E}_{K} : \mathcal{H} \longrightarrow [0,\infty]$ be the energy functional defined in \eqref{energyfunc}. 
\begin{itemize}
\item[(1)] Given any $E_{0} >0$, there exists $\delta_{1} > 0$ such that for all $\eta \in \mathcal{H}$
$$
 \mathcal{E}_{K}(\eta) \leq E_{0} \quad \Longrightarrow \quad  \partial_{x} \eta_1  \geq \delta_{1} \ \ \ \emph{a.e. on} \ \ \ (0,\ell) \, .
$$
As a consequence, if $\eta\in{W^{k,p}(0,\ell)}$ for $p\in [1,\infty]$ and $k\geq 1$, then so is $\eta^{-1}$ with bounds depending on $\delta_1,k,p$ and $\norm[{W^{k,p}(0,\ell)}]{\eta_1})$.
\item[(2)] The energy functional $\mathcal{E}_{K}$ is coercive. More precisely,
$$
\mathcal{E}_{K}(\eta) \geq \dfrac{1}{2} \min\{c_{1}, c_{2}, \lambda_{0}\} \| \eta \|^{2}_{H^2} - \dfrac{\ell c_{1}}{2} \text{ for all } \eta \in \mathcal{H} \, .
$$ 
\item[(3)] The energy functional $\mathcal{E}_{K}$ is lower semi-continuous under weak $\mathcal{H}$-convergence: If $\{\eta_{n} \}_{n \in \mathbb{N}} \subset \mathcal{H}$ converges weakly to $\eta \in \mathcal{H}$, then
$$
\mathcal{E}_{K}(\eta) \leq \liminf_{n \to \infty} \, \mathcal{E}_{K}(\eta_{n}).
$$
\end{itemize}
\end{proposition}

\begin{proof}
\begin{enumerate}[wide=0pt,label=(\arabic{enumi})]
\item Recall that the embedding $H^{2}((0,\ell);\mathbb{R}^2) \subset C^{1, 1/2}([0,\ell];\mathbb{R}^2)$ holds. In particular, given any function $\eta \in \mathcal{H}$, we can apply the mean value theorem to deduce the existence of $x_{*} \in (0,\ell)$ such that $\partial_{x} \eta_{1}(x_{*}) = 1$. Consequently, we can choose $y_{0} \in (0,x_{*})$ sufficiently close to $x_{*}$ such that $\partial_{x} \eta_{1}(\tilde{x}) > 0$, for every $\tilde{x} \in [y_{0},x_{*}]$. Now, since we have using Young's inequality
$$
\frac{1}{(\partial_{x} \eta_{1}(\tilde{x}))^{2 \alpha}} + \dfrac{\lambda_{0}}{2}  \, (\partial_{xx} \eta_{1}(\tilde{x}))^{2} \geq \sqrt{2 \lambda_{0}} \, \frac{\partial_{xx} \eta_{1}(\tilde{x})}{(\partial_{x} \eta_{1}(\tilde{x}))^{\alpha}}  \text{ for all } \tilde{x} \in [y_{0},x_{*}],
$$
we can write
$$
\begin{aligned}
\mathcal{E}_{K}(\eta) & \geq \int_{y_{0}}^{x_{*}} \left[ \dfrac{1}{(\partial_{x} \eta_{1}(\tilde{x}))^{2 \alpha}} + \dfrac{\lambda_{0}}{2} (\partial_{xx} \eta_{1}(\tilde{x}))^{2 } \right] d\tilde{x} \geq \sqrt{2 \lambda_{0}} \int_{y_{0}}^{x_{*}} \dfrac{\partial_{xx} \eta_{1}(\tilde{x})}{(\partial_{x} \eta_{1}(\tilde{x}))^{\alpha}} \, d\tilde{x} \\[3pt]
& = \dfrac{\sqrt{2 \lambda_{0}}}{1 - \alpha} \int_{y_{0}}^{x_{*}} \dfrac{d}{d\tilde{x}} \left[ \dfrac{1}{(\partial_{x} \eta_{1}(\tilde{x}))^{\alpha-1}} \right] d\tilde{x} = \dfrac{\sqrt{2 \lambda_{0}}}{\alpha - 1} \left[ \dfrac{1}{(\partial_{x} \eta_{1}(y_{0}))^{\alpha-1}} - 1 \right],
\end{aligned}
$$
so that
\begin{equation} \label{goodone}
\partial_{x} \eta_{1}(y_{0}) \geq \left[ \dfrac{\alpha-1}{\sqrt{2 \lambda_{0}}} \mathcal{E}_{K}(\eta) + 1 \right]^{\frac{1}{1-\alpha}} \, .
\end{equation}
Now, given any $E_0 >0$, we define
$$
\delta_{1} := \left( \dfrac{\alpha-1}{\sqrt{2 \lambda_{0}}} E_{0} + 1 \right)^{\frac{1}{1-\alpha}} \, ,
$$  
so that, for every $\eta \in \mathcal{H}$ such that $\mathcal{E}_{K}(\eta) \leq E_{0}$, \eqref{goodone} yields the lower bound
\begin{equation} \label{goodone2}
\partial_{x} \eta_{1}(t) \geq \delta_{1}  \quad\text{ for all } t \in I, \quad\text{ for all } I \subset (0,\ell) \quad \text{such that} \quad x_{*} \in I \text{ and } \partial_x \eta_1 > 0 \text{ in }I.
\end{equation}
Since the function $\partial_{x} \eta_{1}$ is continuous in $[0,\ell]$, \eqref{goodone2} guarantees that actually $\partial_{x} \eta_{1} > 0$ in $(0,\ell)$, and so
\[
\partial_{x} \eta_{1}(t) \geq \delta_{1}  \quad\text{ for all } t \in (0,\ell).
\]
This implies further that $\eta_1$ is strictly-monotone, which implies that $\eta_1^{-1}$ is bounded. By the usual chain-rule argument the estimates follow.

\item For any $\eta \in \mathcal{H}$ we clearly have
\begin{align*}
\mathcal{E}_{K}(\eta) & \geq \int_0^{\ell} \left( \dfrac{c_1}{2} \abs{\partial_x \eta_1-1}^2 + \dfrac{\lambda_0}{2} \abs{\partial_{xx} \eta_1}^2 + \dfrac{c_2}{2} \abs{\partial_{xx} \eta_2}^2  \right)dx \\[3pt]
& = \int_0^{\ell} \left( \dfrac{c_1}{2} \abs{\partial_x \eta_1}^2 + \dfrac{\lambda_0}{2} \abs{\partial_{xx} \eta_1}^2 + \dfrac{c_2}{2} \abs{\partial_{xx} \eta_2}^2  \right)dx - \dfrac{\ell c_{1}}{2} \, ,
\end{align*}
from which the statement easily follows.

\item Since $\{\eta_{n} = \left( \eta_{n,1}, \eta_{n,2} \right) \}_{n \in \mathbb{N}} \subset \mathcal{H}$ converges weakly to $\eta \in \mathcal{H}$, there exist $\eta_{1} \in H^{2}((0,\ell);\mathbb{R})$ and $\eta_{2} \in H_{0}^{2}((0,\ell);\mathbb{R})$ such that
$$
(\partial_x \eta_{n,1} - 1) \rightharpoonup (\partial_x \eta_{1} - 1) \, , \qquad \partial_{xx} \eta_{n,1} \rightharpoonup \partial_{xx} \eta_{1} \, , \qquad \partial_{xx} \eta_{n,2} \rightharpoonup \partial_{xx} \eta_{2} \qquad \text{weakly in} \ \ L^{2}((0,\ell);\mathbb{R}),
$$
and therefore
\begin{equation} \label{liminf0}
\left\{
\begin{aligned}
& \| \partial_x \eta_{1} - 1 \|_{L^{2}(0,\ell)} \leq \liminf_{n \to \infty} \| \partial_x \eta_{n,1} - 1 \|_{L^{2}(0,\ell)} \, , \\[3pt]   
& \| \partial_{xx} \eta_{1} \|_{L^{2}(0,\ell)} \leq \liminf_{n \to \infty} \| \partial_{xx} \eta_{n,1} \|_{L^{2}(0,\ell)} \, , \qquad \| \partial_{xx} \eta_{2} \|_{L^{2}(0,\ell)} \leq \liminf_{n \to \infty} \| \partial_{xx} \eta_{n,2} \|_{L^{2}(0,\ell)} \, .
\end{aligned}
\right.
\end{equation}
Finally, consider 
$$
\widetilde{\mathcal{\mathcal{E}_{K}}}(\eta) := \int_{0}^{\ell} \dfrac{1}{|\partial_{x} \eta|^{2\alpha}}  \, dx \text{ for } \eta \in H^{2}((0,\ell);\mathbb{R}^+).
$$
Since $\alpha > 1$, the function $p \mapsto |p|^{-2\alpha}$ is convex in $\mathbb{R}^+$. So since $\partial_x \eta$ cannot take negative values by (1), the functional $\widetilde{\mathcal{\mathcal{E}_{K}}}$ is lower semi-continuous under weak $H^{1}((0,\ell);\mathbb{R})$-convergence. But since $\{\eta_{n,1} \}_{n \in \mathbb{N}}$ also converges weakly to $\eta_{1}$ in $H^{1}((0,\ell);\mathbb{R})$ we deduce that
\begin{equation} \label{liminf1}
\widetilde{\mathcal{\mathcal{E}_{K}}}(\eta) \leq \liminf_{n \to \infty} \, \widetilde{\mathcal{\mathcal{E}_{K}}}(\eta_{n}).
\end{equation}
The conclusion is reached by combining \eqref{liminf0} with \eqref{liminf1}.
\end{enumerate}
\end{proof} 

\subsection{Universal Bogovskij operator} 
Next we need to derive the existence of a Bogovskij-operator for our problem that is in a sense independent of the changing domain. For this we will need the following fundamental theorem. A proof can be found for instance in \cite{hieberbog}.

\begin{proposition}
  \label{prop:reference}
  Let $ Q \subset \mathbb{R}^{n}$ be a cube and  $ b : \rn \to \mathbb{R} $ be any smooth function supported in $ Q $
and having unit integral. Then there exists an operator $\Bcal_Q: C_c^\infty(\R^n) \to C_c^\infty(\R^n;\R^n)$ such that for any $\Omega \subset \rn $ which is star-shaped with respect to $ Q $, the operator $ \Bcal_Q $ maps $ C_0^{\infty}(\Omega) $ into $C_0^{\infty}(\Omega;\rn)$ 
  and for all functions $ g \in C_0^{\infty}(\Omega) $ it holds
\[
\nabla \cdot \Bcal_Q g = g - b \int g \, dx.
\]
In addition,
\[\no{\Bcal_Q}_{W^{s,p}( \Omega )\to W^{s+1,p}(\Omega;\rn)}
\leq  C\] %
for all $ 1<p<\infty $ and $ s \ge 0$
with $C$ only depending on $Q$, $s$, $p$ and $n$. %
\end{proposition}

Note here that the estimate is universal for all star-shaped domains and in particular, the support commutes with the operator, as long as the support of the function stays star-shaped w.r.t.\ the same reference cube $Q$. Moreover, due to the scaling invariance of the operator, we find that $(\mathcal{B}_{Q_r}(g))(x)=(\mathcal{B}_Q (g_r(x))(\frac{x}{r})$, with $g_r(x):=g(rx)$, where $Q_r$ is a cube of side-length $r$.
Hence there is a constant {\em independent of $r$}, such that 
\begin{align*}
\norm[L^p(\Omega)]{\nabla \Bcal_{Q_r} g}\leq c \norm[L^p(\Omega)]{g} \, .
\end{align*}

The theorem below shows that actually an analoguous universal Bogovskij operator can be constructed for a family of uniformly Lipschitz domains. We construct the operator here for our setup, meaning two space dimensions and domains that are sub-graphs. The construction however can be extended to higher dimensions and more general classes of domains without further difficulty (see Corollary \ref{cor:ndBog}). Noteworthy is the fact that the construction does indeed fail once the class of domains is more general than Lipschitz. In this case no universal operator can be expected. See~\cite{SaaSch21} for more details on that complicated matter.

Hence in the following we consider 
 a family of domains given as subgraphs of Lipschitz functions. For $\ell > 0$ and $\eta \in W^{1,\infty}([0,\ell];(0,\infty))$, throughout this section we will denote by
$$
\Omega_{\eta} := \{ z=(x_1,x_2) \in \mathbb{R}^{2} \ | \ 0 < x_1 < \ell \, , \quad 0 < x_2 < \eta(x) \} \, ,
$$
the subgraph of $\eta$.\footnote{Note that for the purposes of this proof, $\eta$ only describes vertical displacement. For the purposes of the rest of this paper, the correct quantity would be $\eta_2 \circ \eta_1^{-1}(x_1)$ but at this point this would only add superfluous notation as this has the same estimates throughout.}

We then have:

\begin{theorem} \label{theobog}
	Let $\ell > 0$, $L > 0$ and $M \geq \gamma > 0$ and take any fixed $b\in C^\infty_0((0,\ell)\times (0,\gamma))$ with unit integral. Then, there is a linear, universal Bogovskij operator $\Bcal: C^\infty_0([0,\ell]\times[0,M]) \to C^\infty_0([0,\ell]\times [0,M];\R^2)$ such that for any $\eta \in W^{1,\infty}([0,\ell];(0,\infty))$ with
	\begin{equation*} %
	\eta(x) \geq \gamma \qquad \text{for a.e. } x \in [0,\ell] \, , \qquad \| \eta \|_{L^{\infty}([0,\ell])} \leq M \, , \qquad \| \eta' \|_{L^{\infty}([0,\ell])} \leq L \, .
	\end{equation*}
	\noindent 
	 the operator $\Bcal$ maps $C^\infty_0(\Omega_\eta)$ to $C^\infty_0(\Omega_\eta;\R^2)$ with
	$
	\nabla\cdot  \Bcal f= f-b\int f\, dx$.
In addition,
\[\no{\Bcal}_{W^{s,p}( \Omega_\eta )\to W^{s+1,p}(\Omega_\eta;\R^2)} \leq C_B^{s,p}\]
for all $ 1<p<\infty $ and $ s \ge 0$
with $C_B^{s,p}$ only depending on $s$, $p$, $\ell, L,\gamma $. In particular in case $s=1$ we find (assuming $\ell L\geq 1$)
	$$
	\begin{aligned}
	C_{B}^{1,p} & \lesssim (1+\ell)\frac{(\ell L)^{\frac{1}{p}}}{\gamma}.
	\end{aligned}
$$
\end{theorem}
\noindent
\begin{proof}
The idea is here to divide the domain into vertical stripes which are thin enough such that on each stripe the domain is star-shaped w.r.t.\ a fixed reference square and Proposition \ref{prop:reference} applies. Afterwards the stripes are connected using a partition of unity and excess/deficiency of mass is corrected on a fixed horizontal stripe away from the boundary.

For that we take a side length $a=\frac{\gamma}{2L}$.
	Without loss of generality we may assume that $L \geq 1$. 
	Accordingly we define for each $[x_0,x_0+a]\subset[0,\ell]$ the set
	$$
	\Omega_{x_0} := \Big\{ z=(x_1,x_2) \in \mathbb{R}^{2} \ \Bigg|\ x_{1} \in (x_0,x_0+a), \,  x_{2} \in (0, \eta(x_1)) \Big\} \, ,
	$$
	Let us now prove that, 
	the domain $\Omega_{x_0}$ is star-shaped with respect to the set $Q_{x_0} := (x_0,x_0+a)\times (0,a)$. %
	It suffices to prove that
	$$
	t(p,q) + (1-t)(x_{1}, \eta(x_{1})) \in \Omega_{x_0} 
	$$
	for all $x_{1} \in (x_0,x_0+a)$, all $(p,q)\in Q_{x_0}$ and all $t\in [0,1]$, which in turn is equivalent to showing
	\begin{equation*} %
	x_{0} \leq tp +  (1-t)x_{1} \leq x_{0} + a \quad \text{and} \quad 0 \leq tq + (1-t) \eta(x_{1}) \leq \eta(tp + (1-t)x_{1}) .
	\end{equation*}
	The first of these is trivially true, as is the lower bound in the second. 
	
	For the respective upper bound, note that equality holds for $t=0$. Treating both sides as a functions of $t$, the left hand side has a derivative
	\begin{align*}
	 \frac{d}{dt} \left(tq + (1-t) \eta(x_1)\right) = q-\eta(x_1) < a-\gamma < -\frac{\gamma}{2} 
	\end{align*}
	while the right-hand side is $L\abs{p-x_{1}}$-Lipschitz as a rescaling of $\eta$. Consequently since $L\abs{p-x_{1}} \leq La =\frac{\gamma}{2}$ the inequality has to hold for all $t>0$ as well.

In order to construct the Bogovskij operator we follow \cite{SaaSch21}. %
	With no loss of generality we may assume that $\int_{\Omega_\eta} f\, dx=0$ and that $f\in C^\infty_0(\Omega_\eta)$.
Next we decompose $ f $ into a sum of functions with mean value zero and support in sets of width $ a$. Let $N:= \lfloor 4L/\gamma\rfloor$
and $ \psi_k \ge 0$ be a positive partition of unity over $[0,\ell]$, such that 
\[
\supp(\psi_k)\subset \Big[\frac{(k-1)\gamma}{4L},\frac{(k+1)\gamma}{4L}\Big],
\] 
with $k\in \{1,\dots,N-1\}$,
\[
\supp(\psi_1)\subset \Big[0,\frac{\gamma}{4L}\Big]\text{ and }\supp(\psi_N)\subset \Big[\frac{(N-1)\gamma}{4L},\frac{N\gamma}{4L}\Big]=.
\] %

Additionally we define non-negative functions $\phi_k \in C_c^\infty(\R^2)$ such that each $\phi_k$ has unit integral and
\[
 \supp(\phi_k) \subset \left[\frac{k\gamma}{4L},\frac{(k+1)\gamma}{4L}\right] \times [0,\gamma]\text{ for }k=\{0,...,N-1\}
\]

Certainly, we may assume that $\abs{\partial^l\psi_k}\leq \frac{c L^l}{\gamma^l}$ and $\abs{\partial^l\psi_k}\leq \frac{c L^l}{\gamma^l}$
for some $c>0$ fixed.
Specifically we may choose 
\[
\phi_k(x_1,x_2)=\frac{2}{\gamma a}\phi\left(\frac{2}{a} \left(x_{1}-k\frac{4L}{\gamma} \right)\right)\phi\Big(\frac{x_{2}}{\gamma}\Big),
\]
where $\int_0^1\phi=1$ and $\norm[{L^p(0,1)}]{\phi}\leq c$ uniformly in $p$.

Next we define linear operators
\[(T_k f) (x_1,x_2) := f(x_1,x_2) \psi_k(x_1) - \sum_{l=1}^k a_l \phi_k(x_1,x_2) + \sum_{l=1}^{k-1} a_l \phi_{k-1}(x_1,x_2) \]
where $a_l := \int_{\Omega_{\eta}} f(x_1,x_2) \psi_{l}(x_1) dx$. Then per definition $T_k f$ has mean zero and $\sum_{k=1}^N T_k f = f$.

Now we have formed the decomposition. Observe that $\supp(T_kf)\subset \Omega_k$, where the domains $\Omega_k$ are of the form of the $\Omega_{x_0}$ studied above. Let us use $Q_k$ to denote the respective cube and accordingly $\Bcal_k$ for the corresponding operator $\Bcal_{Q_k}$ as introduced in Proposition~\ref{prop:reference}. We then define the bounded operator
\[
\Bcal f := \sum_{i=1}^{N} \Bcal_k T_k f.
\]
Now certainly $ \nabla\cdot \Bcal f = f $. Moreover, $\supp(T_k f)\subset \Omega_k$ implies $\supp(\Bcal_k T_k f)\subset \Omega_k$, so in total $\supp \Bcal f \subset \Omega_\eta$ for any $\eta$, as required.

Next we derive the uniform estimates. As the overlap of the $ \Omega_k $ is bounded by a dimensional constant,
we deduce by means of Proposition \ref{prop:reference} for $ \alpha\in \N^{n} $
\[
\|\partial^{\alpha} \Bcal f\|_{L^p(\Omega_\eta)}^{p}
 \le C \sum_{i=1}^{N}\|\partial^{\alpha} \Bcal_{i} T_i f\|_{L^p(\Omega_\eta)}^{p}
 \le C \sum_{i=1}^{N} \sum_{|\alpha'| = |\alpha| - 1}\|\partial^{\alpha'} T_i f\|_{L^p(\Omega_\eta)}^{p} .
\]
Expanding the definition of $ T_i $,
applying the Leibniz rule,
using that $ \psi_i $ and $\phi_i$ have bounded derivatives 
and finally applying Poincar\'e's inequality for trace zero functions
we conclude the claimed bounds for Sobolev spaces of integer order.
The bounds on fractional Sobolev spaces follow by interpolation
although we do not need them here.

We are left estimating the constant in the case $s=0$.
For that we realize that
\begin{align*}
\sum_{i=1}^N\| T_i f\|_{L^p(\Omega_\eta)}^{p}\leq c\sum_{i=1}^N\| f\psi_i-a_i\phi_i\|_{L^p(\Omega_i)}^{p}+c\sum_{i=1}^N\abs{a_i}^p\sum_{j=i}^N \norm[L^p(\Omega_j)]{\phi_j}.
\end{align*}
Since $\int^{\ell}_{0} \psi_i\,dx_1\sim a$, we find
\[
\abs{a_i}^p\leq ca^{p-1}\int_{\Omega_{\eta}}\abs{f}^p\psi_i\, dx\text{ and }\norm[L^p(\Omega_j)]{\phi_j}^p\leq c a^{1-p}\gamma^{1-p}
\]
and so
\[
\|\nabla \Bcal f\|_{L^p(\Omega_\eta)}^{p}\leq c\norm[{L^p(\Omega_\eta)}]{f}^p \left( 1+a^{p-1}\sum_{i=1}^N\norm[L^p(\Omega_j)]{\phi_j}^p \right) \leq  c\norm[{L^p(\Omega_\eta)}]{f}^p (1+N \gamma^{1-p}).
\]
Poincar\'e's inequality implies
\[
\| \mathcal{B} f\|_{L^p(\Omega_\eta)}\lesssim \ell \|\nabla \mathcal{B} f\|_{L^p(\Omega_\eta)}.
\]
Now, by noting that $aN\sim\ell$, we find that $(C_B^{1,p})^p\lesssim (\ell^{p}+1) \gamma^{-p}\ell L$.%
\end{proof}

Since we consider it relevant for future use we show here that the previous theorem can be extended to higher dimensions as well.
\begin{corollary} \label{cor:ndBog}
 Let $\Sigma \subset \R^{n-1}$ be a bounded Lipschitz-domain, $M > \gamma > 0$, $L>0$, $b\in C_0^\infty(\Sigma \times [0,\gamma])$ with unit integral. Then there exists a linear, universal Bogovskij-operator $\Bcal: C_0^\infty(\Sigma \times [0,M]) \to C_0^\infty(\Sigma \times [0,M];\R^{n-1})$ such that for any $L$-Lipschitz function $\eta:\Sigma \to [\gamma,M]$ and $\Omega_\eta := \{(x',x_n) \in \Sigma \times [0,M]: 0< x_n < \eta(x')\}$ the operator $\Bcal$ maps $C_0^\infty(\Omega_\eta)$ to $C_0^\infty(\Omega_\eta;\R^n)$ with $\nabla \cdot \Bcal f = f - b \int f dx$. 
In addition,
\[\no{\Bcal}_{W^{s,p}( \Omega_\eta )\to W^{s+1,p}(\Omega_\eta;\R^2)} \leq C_B^{s,p}\]
for all $ 1<p<\infty $ and $ s \ge 0$
with $C_B^{s,p}$ only depending on $s$, $p$, $\operatorname{diam}(\Sigma), L,\gamma $ and $\Sigma$. %
\end{corollary}

\begin{proof}
 The proof goes along the same lines as before. If we again denote $a:= \tfrac{\gamma}{2L}$, then we can split $\Sigma$ into $O((\operatorname{diam}(\Sigma)/a)^{n-1})$ overlapping sets $(\Sigma_k)_k$ which are each star-shaped with respect to a $n-1$-dimensional cube of side-length $a/c$, where all the constants except $a$ depend only on the boundary of $\Sigma$. We can also assume that $\int_{\Omega_\eta} f \, dx= 0$.
 
 Then by the same considerations as before $\Omega_\eta$ is covered by open sets $\Omega_k := \Sigma_k \times [0,M] \cap \Omega_\eta$ which are each again star-shaped w.r.t.\ an $n$-dimensional cube $Q_k$, moreover, we may assume that $Q_k$ are pairwise disjoint. We can then pick a partition of unity $(\psi_k)_k$ subordinate to $(\Sigma_k)_k$ and smooth functions of unit integral $(\phi_k)_k$ each supported in $Q_k$.
 
Now, we may apply the Proposition \ref{prop:reference} to obtain operators $\Bcal_k$ that map $C_0^\infty(\Omega_k)$ to $C_0^\infty(\Omega_k;\R^n)$ with $\nabla \cdot \Bcal_k f = f$ in case $f$ has mean zero over $\Omega_k$.
 
 Hence if we set $T_k f (x) := \psi_k f(x) - \phi_k \int_{\Omega_k} f(y)\psi_k(y) dy$ then a short calculation reveals that
 \begin{align*}
  \Bcal f := \sum_k \Bcal_k T_k f - \Bcal'\left( \sum_k \phi_k \int_{\Omega_k} f(y)\psi_k(y) dy\right)
 \end{align*}
 is the desired operator, where $\Bcal'$ denotes the Bogovskij-operator on $\Sigma \times [0,\gamma]$. %
\end{proof}

\subsection{Almost solenoidal extensions of vectors}

Finally we use the operator derived in the previous subsection to define an extension-operator tailored to our problem. Important is here that we make an extension of {\em arbitrary} vector-fields into the fluid domain. For that reason the extension cannot be solenoidal in general. In place we give a precise (scalar) control on the divergence of our extension.

\begin{proposition}[Extensions for vector-valued functions] \label{prop:arbMeanExt}
 Let $I \subset [0,T]$ a closed interval, $\gamma > 0$ such that $\sup_{I\times [0,\ell]} \abs{\eta} < \frac{\ell}{2}-\gamma$ and  %
 $\psi^\pm \in C_0^\infty(\Omega^\pm)$ such that $\int_{\Omega^+} \psi^+ = \int_{\Omega^-} \psi^- =1$. Then there exists a linear extension operator 
 \[
 \ext{\eta}:L^2(I;H_0^1([0,\ell];\R^2))\to L^2(I;H^1_0(\Omega;\R^2)),
 \]
 with $\ext{\eta}(b)(t, \eta(t,x)) = b(t,x)$ for all $(t,x) \in I\times [0,\ell]$ and a corresponding linear operator
 \[
 \lam{\eta}:L^2(I;H_0^1([0,\ell];\R^2))\to L^2(I)
 \]
such that  $\nabla \cdot \ext{\eta}(b)(t) = \lam{\eta}(b)(t) (\psi^+ - \psi^-)$. Additionally we have constants $C>0$ depending only on the energy $\sup_{t\in I} \mathcal{E}_K(\eta(t))$ such that for a.e.\ $t\in I$
 \begin{align*}
  \norm[H^1(\Omega)]{\ext{\eta}(b)(t)}& \leq C \norm[H^1(0,\ell)]{b(t)} 
  \\
    \norm[H^2(\Omega)]{\ext{\eta}(b)(t)}& \leq C \norm[H^2(0,\ell)]{b(t)} 
    \\
  \abs{\lam{\eta}(b)} &\leq C  \norm[L^1(0,\ell))]{b}
 \end{align*}
 Furthermore, 
\begin{itemize}
\item for a constant additionally depending on $\norm[L^2{(0,\ell)})]{\partial_t \eta(t)}$ we find
 \begin{align*}
  \norm[L^2(\Omega)]{\partial_t \ext{\eta}(b)(t)} &\leq C \left( \norm[H^2(0,\ell)]{b(t)}+ \norm[L^2(0,\ell)]{\partial_t b(t)}\right) 
  \\
   \abs{\partial_t\lam{\eta}(b)} &\leq C  \left(\norm[H^1(0,\ell)]{b}+\norm[L^1(0,\ell)]{\partial_t b}\right).
 \end{align*}
\item if additionally $\eta(t)\in H^{k_0}([0,\ell];\R^2)$ for some $k_0>1$, then also 
 \begin{align*}
  \norm[H^{k_0}(\Omega)]{\ext{\eta}(b)(t)}& \leq C \norm[H^{k_0}(0,\ell)]{b(t)} 
 \end{align*}
 for constants additionally depending on $\norm[H^{k_0}({[0,\ell]};\R^2)]{\eta}$. 
 \item explicitly\footnote{We include this estimate here, since it will be needed later in the proof.} if $\eta(t)\in H^{3}([0,\ell];\R^2)$, we find
  \begin{align*}
  \norm[H^{3}(\Omega)]{\ext{\eta}(b)(t)}& \leq c(\norm[H^{3}(0,\ell)]{b(t)} + \norm[H^{3}(0,\ell)]{\eta(t)}\norm[H^2(0,\ell)]{b(t)}),
 \end{align*}
 with $c$ depending on $\sup_{t\in I} \mathcal{E}_K(\eta(t))$, only. 
 \item if additionally $\eta(t), \partial_t\eta(t) \in H^{k_0}([0,\ell];\R^2)$ for some $k_0> 1$, then also 
 \begin{align*}
  \norm[H^{k_0-1}(\Omega)]{\partial_t \ext{\eta}(b)(t)} &\leq C \left( \norm[H^{k_0}(0,\ell)]{b(t)}+ \norm[H^{k_0-1}(0,\ell)]{\partial_t b(t)}\right) 
 \end{align*}
 for constants additionally depending on $\norm[H^{k_0}]{\partial_t\eta}$. Note here that all estimates for the time derivative of the extension need an additional spatial derivative of $b$, due to the time-dependency of $\eta$.
 \end{itemize}
\end{proposition}

\begin{proof}
 \emph{Step 1} (naive extension): Let $b\in H_0^1(I;H_0^{2}([0,\ell];\R^2)$. From Proposition \ref{propenergy} we know that the curve $\eta(t): [0,\ell] \to \Omega$ is bi-Lipschitz with bounded constant depending only on the energy. Due to this, it is possible to extend $b(t) \circ \eta(t)^{-1} \in W_0^{2,2}(\eta(t,[0,\ell]);\R^2)$ into a vector field $\tilde{\phi}(t) \in H^1_0(\Omega;\R^2)$ with bounded norm and in such a way that the norm of $\tilde{\phi}(t)$ is bounded by that of $b(t)$ and $\partial_t \tilde{\phi}(t)$ is bounded by $\partial_t b$ and $\partial_t \eta$. To do so, consider first $b \in C_0^\infty(I \times [0,\ell];\R^2)$. Now set
 \begin{align*}
  \tilde{\phi}(t,x) : = b(t,\eta_1(t)^{-1}(x_1)) \beta(x_2)
 \end{align*}
 where $\beta \in C_0^\infty([-\ell/2,\ell/2];[0,1])$ is a fixed, smooth cutoff function with $\beta = 1$ in a neighbourhood of the image of $\eta_2$, (e.g.\ in $[-\tfrac{\ell}{2}+\gamma,\tfrac{\ell}{2}-\gamma]$). 
 
 The construction implies $\tilde{\phi}(t) = 0$ in $\partial \Omega$ and since $\eta_1$ is a diffeomorphism, we have
 \begin{align*}
  \partial_1 \tilde{\phi}(t,x) &= \partial_x \eta_1(t)^{-1} \partial_x b(t,\eta(t)^{-1}(x_1)) \beta(x_2) \\
  \partial_2 \tilde{\phi}(t,x) &= b(t,\eta_1(t)^{-1}(x_1)) \beta'(x_2)
 \end{align*}
 implying $\norm[H^1(\Omega)]{\smash{\tilde{\phi}(t)}} \leq C \norm[H^1(\Omega)]{b(t)}$ and $\norm[H^2(\Omega)]{\smash{\tilde{\phi}(t)}} \leq C \norm[H^2(\Omega)]{b(t)}$. Here we use $(1)$ in Proposition \ref{propenergy}. Further estimates for higher order derivatives in space follow by analogous computations. 
 For the time derivative observe that
 \[
 0=\partial_t\big(\eta_1(t,\eta_1(t)^{-1}(z))\big)=\partial_t\eta_1(t,\eta_1(t)^{-1}(z))+\partial_1\eta_1(t,\eta_1(t)^{-1}(z))\partial_t\eta_1(t)^{-1}(z),
 \]
 which implies by $(1)$ in Proposition \ref{propenergy} that
 \[
 \partial_t\eta_1(t)^{-1}(z)=-\frac{\partial_t\eta_1(t,\eta_1(t)^{-1}(z))}{\partial_1\eta_1(t,\eta_1(t)^{-1}(z))}\text{ and so }\abs{ \partial_t\eta_1(t)^{-1}(z)}\leq \frac{\abs{\partial_t\eta_1(t,\eta_1(t)^{-1}(z))}}{\delta_0}
 \]
 and the respective natural bounds for the time derivative of $\eta_1^{-1}$. In particular, as
 \begin{align*}
  \partial_t \tilde{\phi} &= \partial_t b(t,\eta_1(t)^{-1}(x_1)) \beta(x_2) + \partial_t \eta_1(t)^{-1} \partial_x b(t,\eta(t)^{-1}(x_1))\beta(x_2) %
 \end{align*}
implies for a constant depending on $\delta_0$ by Sobolev embedding that
\begin{align*}
\norm[L^2(0,\ell)]{\partial_t\phi}&\leq c\norm[L^2(0,\ell)]{\partial_t\eta}\norm[L^\infty(0,\ell)]{\partial_x b}+c\norm[L^2(0,\ell)]{\partial_t b}
 \\
 &\leq  c\norm[L^2(0,\ell)]{\partial_t\eta}\norm[H^2(0,\ell)]{ b}+c\norm[L^2(0,\ell)]{\partial_t b}.
 \end{align*}
 Further
 \begin{align*}
   \partial_t \partial_1 \tilde{\phi} &= \partial_x \eta_1(t)^{-1}(x_1) \partial_t \partial_x b(t,\eta_1(t)^{-1}(x_1)) \beta(x_2) + \partial_t \partial_x \eta_1(t)^{-1} \partial_x b(t,\eta(t)^{-1}(x_1))\beta(x_2) \\
   & \quad +\partial_t \eta_1(t)^{-1} \partial_x \eta_1(t)^{-1} \partial_{xx} b(t,\eta(t)^{-1}(x_1))\beta(x_2) \\
  \partial_t \partial_2 \tilde{\phi} &= \partial_t b(t,\eta_1(t)^{-1}(x_1))\beta'(x_2) + \partial_t \eta_1(t)^{-1} \partial_x b(t,\eta(t)^{-1}(x_1))\beta'(x_2).
 \end{align*}
 which implies respective estimates on the $H^1(\Omega;\R^2)$-norm of $\partial_t\tilde{\phi}$ provided $\partial_t\eta$ is differentiable. Higher order derivatives follow similarly.%

\emph{Step 2} (correction of divergence): We define $\lam{\eta}(b)(t) := \int_{\Omega^+(t)} \nabla \cdot \tilde{\phi}(t) dx$. As $\int_{\Omega} \nabla \cdot \tilde{\phi}(t) dx = \int_{\partial \Omega} \tilde{\phi}(t) \cdot n\, dx = 0$, we then automatically get that also $\lam{\eta}(b)(t) = - \int_{\Omega^-} \nabla \cdot \tilde{\phi}(t)$ and similarly
\begin{align}
\label{eq:solenoid}
 \lam{\eta}(b)(t) =  \int_{\Omega^+(t)} \nabla \cdot \tilde{\phi}(t) dx= \int_{\eta(t,[0,\ell])} \tilde{\phi}(t) \cdot n dx = \int_0^\ell \partial_x \eta(t) \wedge b(t) dx=-\int_0^\ell \eta(t) \wedge \partial_x b(t) dx.
\end{align}
This description implies the required estimates for $ \lam{\eta}(b)(t)$ rather directly.

Now we use the time-independent Bogovskij operator from Theorem~\ref{theobog}
on $\Omega^+$ and $\Omega^-$ (denoted by $\mathcal{B}^\pm$) to construct
\begin{align*}
\ext{\eta}(b) := \tilde{\phi}(t) - \mathcal{B}^+(\nabla \cdot \tilde{\phi}(t) |_{\Omega^+} - \lambda(t) \psi^+) - \mathcal{B}^-(\nabla \cdot \tilde{\phi}(t) |_{\Omega^-} + \lambda(t) \psi^+).
\end{align*}
Per construction this has the correct divergence and coupling condition. Additionally the estimates on the Bogovskij-operator mean that, up to another constant factor, $\ext{\eta}(b)$ obeys the same space-regularity estimates as $\tilde{\phi}$ and because of its domain-independence, the time derivatives commute.

The higher order-estimates are obtained in precisely the same fashion.
\end{proof}

\section{An intermediate, time-delayed model}

The following section is in analogue to the methodology introduced in~\cite{BenKamSch20} for the a related problem on bulk solids.
It begins by proving existence of weak solutions for a so-called {\em time-delayed problem} associated with \eqref{weak0}, where inertial effects in the Navier-Stokes equations are not modeled by time derivatives but instead by requiring the velocity $u_{\pm}(t, \cdot)$ to be close to the velocity $u_{\pm}(t-h, \cdot)$ at some previous time. Further, since we only consider time intervals of length $h$ for now, this previous velocity can be thought of as given data. 	
In this section, we will denote by $\Phi_{\pm} : [0,T) \times \Omega_{\pm}(0) \longrightarrow \Omega$ the \textit{flow map} such that $\Phi_{\pm}(0, \cdot) : \Omega_{\pm}(0) \longrightarrow \Omega_{\pm}(0)$ is the identity function, and it maps the initial fluid domain $\Omega_{\pm}(0)$ into its form $\Omega_{\pm}(t)$ at a later time $t \in [0,T)$, while following the trajectory of the flow, that is
\begin{equation*}
\partial_{t} \Phi_{\pm}(t,x) = u_{\pm}(t,\Phi_{\pm}(t,x)) \text{ for all } (t,x) \in [0,T) \times \Omega_{\pm}(t_{0}).
\end{equation*} 
Existence of such a map is not guaranteed in general. In fact we will spend quite some work constructing it for the time-delayed problem and even then it will be the one object for which we cannot obtain convergence to the limit $h \to 0$. 
For the flow-map to be well defined we introduce the regularization parameter $\delta_0>0$, which eventually will be set equal to the {\em acceleration scale} $h$. A second level of approximation is the $\epsilon_0$ level, that introduces some dissipation for the solid, which is later needed to derive the energy equality and will be removed before passing with $h\to 0$.
\begin{definition} \label{timedelay0}
Given $w \in L^{2}([0,h] \times \Omega; \mathbb{R}^2)$ and initial data $\eta_0 \in H^3((0,\ell);\R^2)$ with  $\mathcal{E}_K(\eta_0) < \infty$ and $w \circ \eta_0 \in L^2([0,h] \times (0,\ell);\mathbb{R}^2)$, $f \in L^{2}(\Omega; \mathbb{R}^2)$ and a resulting $\Omega_\pm$,  the tuple $(u_{\pm}, \eta) \in V_{S}\cap L^2([0,h],H^3(\Omega;\R^2))\times H^1([0,h],H^3((0,\ell)))$ is called a \textbf{weak solution to the time-delayed equation} if for almost all $t\in[0,h]$
\begin{align} \label{equasistaticSol} 
 \int_{0}^{\ell} \left( \mathcal{L}_{\Gamma}(\eta) + \rho_s \tfrac{\partial_t \eta - w \circ \eta}{h} \right) \cdot \xi \, dx &+ \varepsilon_{0} \int_{0}^{\ell} \partial_{t}\partial_{x}^3  \eta \cdot \partial_{x}^3 \xi \, dx +\sqrt{\delta_0} \int_{0}^{\ell} \partial_{x}^3  \eta \cdot \partial_{x}^3  \xi \, dx + \mu_{+} \int_{\Omega_{+}^{\eta}(t)} \nablasym(u_{+}) \cdot \nabla q \, dz \nonumber \\
 + \mu_{-} \int_{\Omega_{-}^{\eta}(t)} \nablasym(u_{-}) \cdot \nabla q \, dz &+\delta_{0} \int_{\Omega_+^\eta(t)}  \nabla(\Delta u_+) \cdot \nabla (\Delta q) \, dz + \delta_{0} \int_{\Omega_-^\eta(t)}  \nabla(\Delta u_-) \cdot \nabla (\Delta q) \, dz   \nonumber \\ &+ \rho_{+} \int_{\Omega_{+}} \left( \tfrac{u_{+}\circ \Phi_{+} - w}{h} \cdot q \circ \Phi_{+} \right) dz  
 +  \rho_{-} \int_{\Omega_{-}} \left( \tfrac{u_{-}\circ \Phi_{-} - w}{h} \cdot q \circ \Phi_{-} \right) dz \nonumber \\
 &= \rho_{+} \int_{\Omega_{+}^{\eta}(t)} \left(f \cdot q \right) \, dz + \rho_{-} \int_{\Omega_{-}^{\eta}(t)} \left( f \cdot q \right) \, dz,
\end{align}
holds with the kinetic coupling condition
\begin{equation*} %
u_{+}(t, \eta (t,x)) = u_{-}(t, \eta (t,x)) = \partial_t \eta(x,t) \quad \text{for a.e.} \ (t,x) \in [0,h] \times [0,\ell],
\end{equation*} for all coupled test functions 
$
(\xi,q)\in L^2([0,h],H^3\cap H^2_0(\Omega;\R^2))\times L^{2}([0,h],H^3\cap H^2_0((0,\ell)))$ 
such that 
$$
\nabla \cdot q = 0 \quad \text{in} \quad \Omega, \qquad q(t,\eta (t,x)) = \xi(t,x) \quad\text{ for all } (t,x) \in [0,h] \times [0,\ell].
$$
Here $\Phi_{\pm} :[0,h] \times \Omega^{\eta}_{\pm} \longrightarrow \Omega$ solves $\partial_t \Phi_{\pm} = u_{\pm} \circ \Phi_{\pm}$ and $\Phi_{\pm}(0, \cdot) = \operatorname{id}|_{\Omega_{\pm}}$. We also note that here and in the future, the term $\mathcal{L}_\Gamma(\eta) \cdot \xi$ will be understood in terms of the corresponding $H^{-2}\times H^{2}$ dual pairing. Furthermore $\eta$ is assumed to attain appropriate initial condition.
\end{definition}

\begin{remark}
The terms $w$ and $w \circ \eta$ appearing in \eqref{equasistaticSol} will later be substituted by transported versions of the combined fluid and solid velocity from the previous time-interval (the one corresponding to $[-h,0]$).  This will turn all terms that include $w$ into difference quotients, which will then be shown to converge to the corresponding (material) derivatives.
\end{remark}

The purpose of this sub-section is to prove the following existence theorem:

\begin{theorem} \label{maintheoremdelay}
Given $h > 0$, $\varepsilon_0\geq 0$, $\delta_0>0$, $w \in L^{2}([0,h] \times (0,\ell); \mathbb{R}^2)$, $\eta_0 \in H^3((0,\ell);\R^2)$ with  $\mathcal{E}_K(\eta_0) < \infty$  and $w \circ \eta_0 \in L^2([0,h] \times \Omega;\mathbb{R}^2)$ and $f \in L^{2}(\Omega; \mathbb{R}^2)$, there exists a weak solution of the time delayed equation \eqref{equasistaticSol} in the sense of Definition \ref{timedelay0}. Furthermore, $\Phi_{\pm}(t, \cdot)$ are volume-preserving homomorphisms, for every $t \in [0,h]$ and the solution satisfies the energy-inequality of Lemma \ref{timeDelayedEnergyInequality}.
\end{theorem}
The proof of the theorem is performed in several steps. First we consider $\varepsilon_0>0$. For that regime we first show the existence of the weak solution. In Subsection \ref{ssec:energ} we show that it satisfies an energy inequality {\em independent of $h,\delta_0$ and $\varepsilon_0$.} Finally in Subsection \ref{ssec:eps} we use this to treat the case $\varepsilon_0=0$ via another limit passage.

\subsection{Proof of Theorem \ref{maintheoremdelay}, Step 1: Constructing an iterative approximation}

In order to prove Theorem \eqref{maintheoremdelay}, for a given integer $N > 1$ we fix a step size $\tau = h / N$. For an integer $k \in \{0,\cdots,N-1\}$ assume that we are given:
\begin{itemize} [leftmargin=*]
\item{} a deformation $\eta^{k} \in \mathcal{H}$ that decomposes the domain $\Omega$ as the union of its lower and upper parts $\Omega_{\pm}^{k}$;
\item{} flow maps in the form of diffeomorphisms $\Phi^{k}_{\pm} : \Omega \longrightarrow \Omega_{\pm}^{k}$. 
\end{itemize} 
\noindent
With this we can define the \textit{global quantities} over the fixed domain $\Omega$, at the $k$-level, through the expressions
$$
\begin{aligned}
& \rho^{k} = \chi_{\Omega_{+}^{k}} \rho_{+} + \chi_{\Omega_{-}^{k}} \rho_{-} \, , \qquad \mu^{k} = \chi_{\Omega_{+}^{k}} \mu_{+} + \chi_{\Omega_{-}^{k}} \mu_{-} \, , \qquad u^{k} = \chi_{\Omega_{+}^{k}} u^{k}_{+} + \chi_{\Omega_{-}^{k}} u^{k}_{-} \, , \\[6pt]
&  %
 \Phi^{k} =  \chi_{\Omega_{+}^{k}} \Phi^{k}_{+} + \chi_{\Omega_{-}^{k}} \Phi^{k}_{-} \, ,
\end{aligned}
$$
where $\chi_{\Omega_{\pm}^{k}}$ denotes the characteristic function of $\Omega_{\pm}^{k}$. Furthermore, we introduce the time-average $w^{k} : \Omega \longrightarrow \mathbb{R}^2$ by
\begin{equation} \label{wtaverage}
w^{k} = \dfrac{1}{\tau} \int_{\tau k}^{\tau (k+1)} w(t,\cdot) \, dt \qquad \text{in} \ \ \Omega.
\end{equation}
Given $\delta_{0} > 0$, the next incremental time-step on this level of the velocity quotients (depending on $\tau$) is then defined as a minimizer $(\eta^{k+1},u^{k+1})$ of the following coupled functional
\begin{equation} \label{coupledfunc}
\begin{aligned}
& G^{k}(\eta, u) = \mathcal{E}_{K}(\eta) + \dfrac{\varepsilon_0\tau}{2} \int_{0}^{\ell} \left| \partial_{x}^3 \left( \tfrac{\eta - \eta^{k}}{\tau} \right)  \right|^{2} dx + \dfrac{\tau \rho_s }{2 h} \int_{0}^{\ell} \left| \tfrac{\eta - \eta^{k}}{\tau} -  w^k \circ \eta_{0} \right|^{2} dx  +\frac{\sqrt{\delta_0}}{2}\int_{0}^{\ell} | \partial_{x}^3  \eta|^2\, dx 
\\[6pt]
& \hspace{4mm} + \dfrac{\tau}{2 h} \int_{\Omega} \rho_0 \left| u \circ \Phi^{k} - w^{k} \right|^{2} dz  +  \dfrac{\tau}{2} \int_{\Omega} \left( \delta_{0} \left| \nabla (\Delta u) \right|^{2} + \mu^{k} \left| \nablasym(u) \right|^{2} \right) dz - \tau \int_{\Omega} \rho^{k} \left( f \cdot u \right) dz 
\end{aligned}
\end{equation}
over all pairs $(\eta,u)\in \mathcal{H}\cap H^3((0,\ell))\times H^3\cap H^1_0(\Omega)$,
such that 
$$
\nabla \cdot u = 0 \quad \text{in} \quad \Omega, \qquad \dfrac{\eta - \eta^{k}}{\tau} = u \circ \eta^{k} \quad \text{in} \quad [0,\ell].
$$
This is then used to update $\Phi_{k+1} := (\operatorname{id} + \tau u^{k+1}) \circ \Phi_k$ and the process is repeated until $(k+1) \tau > h$. %
For this we now prove the following:
\begin{proposition} \label{minfunctional1}
The functional \eqref{coupledfunc} has a minimizing couple
$
(\eta^{k+1},u^{k+1}) \in \mathcal{H}\cap H^3((0,\ell))\times H^1_0\cap H^{3}(\Omega;\mathbb{R}^{2})
$
fulfilling the conditions
\begin{equation*} %
\nabla \cdot u^{k+1} = 0 \quad \text{in} \quad \Omega, \qquad \dfrac{\eta^{k+1} - \eta^{k}}{\tau} = u^{k+1} \circ \eta^{k} \quad \text{in} \quad [0,\ell].
\end{equation*}
Furthermore, the minimizers obey the identity 
\begin{equation} \label{equasistaticSolmin} 
\begin{aligned}
& \int_{0}^{\ell} \mathcal{L}_{\Gamma}(\eta^{k+1}) \cdot \xi \, dx + \varepsilon_0 \int_{0}^{\ell} \partial_{x}^3 \left( \tfrac{\eta^{k+1} - \eta^{k}}{\tau} \right) \cdot \partial_{x}^3 \xi \, dx + \dfrac{\rho_s }{h} \int_{0}^{\ell} \left( \tfrac{\eta^{k+1} - \eta^{k}}{\tau} -  w^k \circ \eta_{0} \right) \cdot \xi \, dx \\[6pt] 
& +\sqrt{\delta_0}\int_{0}^{\ell}  \partial_{x}^3  \eta^{k+1}\partial_x^3\xi\, dx 
 + \dfrac{1}{h} \int_{\Omega} \rho_{0} \left(u^{k+1} \circ \Phi^{k} - w^{k} \right) \cdot \left( q \circ \Phi^{k} \right) dz  + \delta_{0} \int_{\Omega}  \nabla(\Delta u^{k+1}) \cdot \nabla (\Delta q) \, dz \\[6pt]
& + \int_{\Omega} \mu^{k} \nablasym u^{k+1} \cdot \nablasym(q) \, dz = \int_{\Omega} \rho^{k} \left( f \cdot q \right) dz,
\end{aligned}
\end{equation}for all (time-independent) coupled test functions
$
(\xi,q )\in \mathcal{H}\cap H^3(0,\ell)\times  H_{0}^{1}\cap H^3(\Omega;\mathbb{R}^{2})
$
such that 
$
\nabla \cdot q = 0\text{ in } \Omega$ and $q \circ \eta^{k} = \xi \text{ in }[0,\ell].
$
\end{proposition}

\begin{proof}
We investigate the existence of a minimizer for the functional $G^{k}$, defined in \eqref{coupledfunc}, using the direct method. Let us introduce the space of \textit{admissible} functions for $G^{k}$ by
\begin{equation} \label{admisset} 
\mathcal{A}^{k} = \left\lbrace (\eta, u) \in \mathcal{H}\cap H^3((0,\ell)) \times H^1_0\cap H^{3}(\Omega;\mathbb{R}^{2}) \ | \ \nabla \cdot u = 0 \ \ \text{in} \ \ \Omega, \quad \eta - \eta^{k} = \tau ( u \circ \eta^{k} ) \ \ \text{in} \quad [0,\ell] \right\rbrace.
\end{equation}
\noindent
Firstly, notice that for every $(\eta, u) \in \mathcal{A}^{k}$ we have
\begin{align} \label{lowerbound1} 
G^{k}(\eta, u) & \geq \dfrac{\tau \delta_{0}}{2} \int_{\Omega} \left| \nabla (\Delta u) \right|^{2} dz + \dfrac{\tau}{2} \int_{\Omega} \mu^{k} \left| \nablasym(u) \right|^{2} dz - \tau \int_{\Omega} \rho^{k} \left( f \cdot u \right) dz + \frac{\sqrt{\delta_0}}{2}\int_0^\ell \abs{\partial_{x}^3 \eta}^2 dx  \nonumber \\[6pt]
& \geq \dfrac{\tau \delta_{0}}{2} \int_{\Omega} \left| \nabla (\Delta u) \right|^{2} dz + \dfrac{\tau \mu_{\min}}{2} \int_{\Omega} \left| \nablasym(u) \right|^{2} dz - \tau \rho_{\max} \int_{\Omega} \left| f \cdot u \right| dz+ \frac{\sqrt{\delta_0}}{2}\int_0^\ell \abs{\partial_{x}^3 \eta}^2 dx \\[6pt]
& \geq \dfrac{\tau \delta_{0}}{2} \int_{\Omega} \left| \nabla (\Delta u) \right|^{2} dz + \tau \int_{\Omega} \left[ \left( \dfrac{\mu_{\min}}{4} - \varepsilon \rho_{\max} C_{P}  \right) \left| \nabla u \right|^{2} - \dfrac{\rho_{\max}}{4 \varepsilon}  \left| f \right|^{2} \right] dz+ \frac{\sqrt{\delta_0}}{2}\int_0^\ell \abs{\partial_{x}^3 \eta}^2 dx, \nonumber
\end{align}
for a sufficiently small $\varepsilon > 0$, where $\mu_{\min} := \min\{\mu_{+}, \mu_{-}\}$ and $\rho_{\max} := \max\{\rho_{+}, \rho_{-}\}$; the third inequality in \eqref{lowerbound1} is obtained after integrating by parts, applying Young and Poincaré inequalities, with $C_{P} := 2 L^{2} / \pi^{2}$ being the Poincaré constant of the square $\Omega$. Inequality \eqref{lowerbound1} clearly shows that the functional $G^{k}$ has a finite infimum, that is
$$
\inf_{(\eta, u) \in \mathcal{A}^{k}} G^{k}(\eta, u) > - \infty,
$$
so that a minimizing sequence $\{ (\eta_{n}, u_{n}) \}_{n \in \mathbb{N}} \subset \mathcal{A}^{k}$ exists and verifies
$$
\inf_{(\eta, u) \in \mathcal{A}^{k}} G^{k}(\eta, u) = \lim_{n \to \infty} G^{k}(\eta_{n}, u_{n}).
$$
By the coercivity of $\mathcal{E}_{K}$, see Proposition \ref{propenergy}, and the regularization term, we know that the sequence $\{ \eta_{n} \}_{n \in \mathbb{N}}$ is bounded in $\mathcal{H}\cap H^3$. In view of the compact embeddings $H^{2}((0,\ell);\mathbb{R}) \subset C^{1,1/2^{-}}([0,\ell];\mathbb{R})$
we can therefore extract a sub-sequence (not relabeled) such that
\begin{equation} \label{convergence1} 
\eta_{n} \rightharpoonup \eta \quad \text{weakly in} \quad \mathcal{H} \cap H^3([0,\ell];\R^2), \qquad \eta_{n} \longrightarrow \eta \quad \text{strongly in} \quad C^{1,1/2^{-}}([0,\ell];\mathbb{R}).
\end{equation}
Since $\mathcal{H}$ is clearly weakly closed, we deduce that $\eta^{k+1} := \eta$ is an element of $\mathcal{H}$. Regarding the velocity field $u$, let us go back to \eqref{lowerbound1}, from where we deduce the following coercivity property:
\begin{equation} \label{lowerbound2} 
G^{k}(\eta, u) \geq \dfrac{\tau \delta_{0}}{2} \| u \|_{H_{0}^{3}(\Omega)}^{2} +  \dfrac{\tau \mu_{\min}}{8} \| \nabla u \|_{L^{2}(\Omega)}^{2} - \dfrac{4 \tau \rho_{\max}^{2} L^{2}}{\mu_{\min} \pi^{2}} \| f \|_{L^{2}(\Omega)}^{2} \text{ for all } (\eta, u) \in \mathcal{A}^{k}.
\end{equation}
The coercivity condition \eqref{lowerbound2} ensures that the sequence $\{ u_{n} \}_{n \in \mathbb{N}}$ is bounded in $H^{3}_{0}(\Omega;\mathbb{R}^2)$, and therefore, we can extract a sub-sequence (not relabeled) such that
\begin{equation} \label{convergence2} 
u_{n} \rightharpoonup u \quad \text{weakly in} \quad H_{0}^{3}(\Omega;\mathbb{R}^2), \qquad u_{n} \longrightarrow u \quad \text{strongly in} \quad H^1(\Omega;\mathbb{R}^2).
\end{equation}

Furthermore, since $\nabla  \cdot u_{n} = 0$ in $\Omega$ for every $n \in \mathbb{N}$, by weak convergence we also deduce that $\nabla  \cdot u^{k+1} = 0$ in $\Omega$. In order to conclude that $\left( \eta^{k+1}, u^{k+1} \right) \in \mathcal{A}^{k}$ we need to check that the coupling condition is met; to see this, notice that convergence in \eqref{convergence2} implies also convergence of the trace of $u_n$. Since $\{ (\eta_{n}, u_{n}) \}_{n \in \mathbb{N}} \subset \mathcal{A}^{k}$, we have
\begin{equation} \label{convergence3} 
\eta_{n} - \eta^{k} = \tau ( u_{n} \circ \eta^{k} ) \quad \text{in} \quad [0,\ell], \quad\text{ for all } n \in \mathbb{N}.
\end{equation}
which also implies the same in the limit. Thus we have $\left( \eta^{k+1}, u^{k+1} \right) \in \mathcal{A}^{k}$. Finally in order to conclude 
$$
\min_{(\eta, u) \in \mathcal{A}^{k}} G^{k}(\eta, u) = G^{k} \left( \eta^{k+1}, u^{k+1} \right),
$$
we need to state the weak lower semi-continuity of the functional $G^{k}$, or equivalently, the weak lower semi-continuity of each of the terms appearing in \eqref{coupledfunc}. This property has been proved for the energy $\mathcal{E}_{K}$ in Proposition \ref{propenergy} and can be directly proved for the remaining terms in \eqref{coupledfunc} as all of them are either linear or semi-norms in spaces where we have weak convergence. In conclusion, $\left( \eta^{k+1}, u^{k+1} \right) \in \mathcal{A}^{k}$ is indeed a minimizer of $G^{k}$.\par
Let us now derive the Euler-Lagrange equation associated to the minimizer $\left( \eta^{k+1}, u^{k+1} \right)$.  If we consider a test pair $(\xi, q) \in \mathcal{H} \times H_{0}^{3}(\Omega;\mathbb{R}^{2})$ such that $\nabla \cdot q = 0$ in $\Omega$ and $q \circ \eta^{k} = \xi$ in $[0,\ell]$, we have
$$
\left( \eta^{k+1} + \varepsilon \xi, u^{k+1} + \dfrac{\varepsilon}{\tau} q \right) \in \mathcal{A}^{k}, \quad\text{ for all } \varepsilon > 0.
$$
Therefore, we are allowed to take the first variation with respect to $\left( \xi, \frac{q}{\tau} \right)$, resulting in the weak Euler-Lagrange equation \eqref{equasistaticSolmin}.
\end{proof}

Now continuing our construction we define
\begin{equation} \label{newphi}
\Phi^{k+1} := (\operatorname{id}+\tau u^{k+1}) \circ \Phi^k = \Phi^{k} + \tau \left( u^{k+1} \circ \Phi^{k} \right), \quad \Phi^{k+1} : \Omega \longrightarrow \Omega^{k+1}.
\end{equation}

From Proposition \ref{minfunctional1} we can immediately derive the following result:
\begin{lemma} \label{apriorih}
Let $\left( \eta^{k+1}, u^{k+1} \right) \in \mathcal{H}\cap H^3((0,\ell)) \times H_{0}^{1}\cap H^3(\Omega;\mathbb{R}^{2})$ be a minimizer of the functional \eqref{coupledfunc} and define $\Phi^{k+1} : \Omega \longrightarrow \Omega$ as in \eqref{newphi}. The following a priori estimate holds:
{\small
\begin{equation} \label{apriorieu1} 
\begin{aligned}
& \mathcal{E}_{K}(\eta^{k+1})+ \frac{\sqrt{\delta_0}}{2}\int_{0}^{\ell} | \partial_{x}^3  \eta^{k+1}|^2\, dx + \dfrac{\varepsilon_0}{2 \tau} \int_{0}^{\ell} \left| \partial_{x}^3 \left( \eta^{k+1} - \eta^{k} \right)  \right|^{2} dx + \dfrac{\tau \rho_s }{2 h} \int_{0}^{\ell} \left| \dfrac{\eta^{k+1} - \eta^{k}}{\tau} -  w^k \circ \eta_{0} \right|^{2} dx   \\[6pt]
& + \dfrac{\tau}{2 h} \int_{\Omega} \rho_{0} \left| u^{k+1} \circ \Phi^{k} - w^{k} \right|^{2} dz  +  \dfrac{\tau}{2} \int_{\Omega} \left[ \delta_{0} \left| \nabla \left( \Delta u^{k+1} \right) \right|^{2} + \mu^{k} \left| \nablasym u^{k+1} \right|^{2} \right] dz - \tau \int_{\Omega} \rho^{k} \left( f \cdot u^{k+1} \right) dz \\[6pt]
& \hspace{-3mm} \leq \mathcal{E}_{K}(\eta^{k}) + \frac{\sqrt{\delta_0}}{2}\int_{0}^{\ell} | \partial_{x}^3  \eta^{k}|^2\, dx+ \dfrac{\tau \rho_s }{2 h} \int_{0}^{\ell} \left| w^k \circ \eta_{0} \right|^{2} dx + \dfrac{\tau}{2 h} \int_{\Omega} \rho_{0} \left| w^{k} \right|^{2} dz   \, ,
\end{aligned}
\end{equation}}as well as
{\small
\begin{equation} \label{apriorieu2} 
\begin{aligned}
& \mathcal{E}_{K}(\eta^{N}) + \frac{\sqrt{\delta_0}}{2}\int_{0}^{\ell} | \partial_{x}^3  \eta^{N}|^2\, dx + \sum_{k=0}^{N-1} \Bigg[ \dfrac{\varepsilon_0}{2 \tau} \int_{0}^{\ell} \left| \partial_{x}^3 \left( \eta^{k+1} - \eta^{k} \right)  \right|^{2} dx + \dfrac{\tau \rho_s }{2 h} \int_{0}^{\ell} \left| \dfrac{\eta^{k+1} - \eta^{k}}{\tau} -  w^{k} \circ \eta_{0} \right|^{2} dx  \\[6pt]
&  + \dfrac{\tau}{2 h} \int_{\Omega} \rho_{0} \left| u^{k+1} \circ \Phi^{k} - w^{k} \right|^{2} dz  +  \dfrac{\tau}{2} \int_{\Omega} \left[ \delta_{0} \left| \nabla \left( \Delta u^{k+1} \right) \right|^{2} + \mu^{k} \left| \nablasym u^{k+1}  \right|^{2} \right] dz - \tau \int_{\Omega} \rho^{k} \left( f \cdot u^{k+1} \right) dz \Bigg] \\[6pt]
& \hspace{-3mm} \leq \mathcal{E}_{K}(\eta^{0}) + \frac{\sqrt{\delta_0}}{2}\int_{0}^{\ell} | \partial_{x}^3  \eta^{0}|^2\, dx + \dfrac{\rho_s }{2 h} \int_{0}^{h} \int_{0}^{\ell} \left| w \circ \eta_{0} \right|^{2} dx\,dt + \dfrac{\rho_{+} + \rho_{-}}{2 h} \int_{0}^{h} \int_{\Omega} \left| w \right|^{2} dz \, dt \, .\
\end{aligned}
\end{equation}}
\end{lemma}
\begin{proof}
Inequality \eqref{apriorieu1} is obtained after comparing the value of $G_k$ at the minimizer $\left( \eta^{k+1}, u^{k+1} \right)$ with that at the admissible pair $\left( \eta^{k}, 0 \right) \in \mathcal{A}^{k}$, see \eqref{admisset}. In order to derive \eqref{apriorieu2} we add each of the inequalities \eqref{apriorieu1} for $k \in \{0,\cdots,N-1\}$, thus yielding
{\small
\begin{equation} \label{apriorieu3} 
\begin{aligned}
& \mathcal{E}_{K}(\eta^{N})+ \frac{\sqrt{\delta_0}}{2}\int_{0}^{\ell} | \partial_{x}^3  \eta^{N}|^2\, dx + \sum_{k=0}^{N-1} \Bigg[ \dfrac{\varepsilon_0}{2 \tau} \int_{0}^{\ell} \left| \partial_{x}^3 \left( \eta^{k+1} - \eta^{k} \right)  \right|^{2} dx + \dfrac{\tau \rho_s }{2 h} \int_{0}^{\ell} \left| \dfrac{\eta^{k+1} - \eta^{k}}{\tau} -  w^{k} \circ \eta_{0} \right|^{2} dx \\[6pt]
& + \dfrac{\tau}{2 h} \int_{\Omega} \rho_0 \left| u^{k+1} \circ \Phi^{k} - w^{k} \right|^{2} dz  +  \dfrac{\tau}{2} \int_{\Omega} \left[ \delta_{0} \left| \nabla \left( \Delta u^{k+1} \right) \right|^{2} + \mu^{k} \left| \nablasym u^{k+1} \right|^{2} \right] dz - \tau \int_{\Omega} \rho^{k} \left( f \cdot u^{k+1} \right) dz \Bigg] \\[6pt]
& \hspace{-2mm} \leq \mathcal{E}_{K}(\eta^{0}) + \frac{\sqrt{\delta_0}}{2}\int_{0}^{\ell} | \partial_{x}^3  \eta^{0}|^2\, dx + \sum_{k=0}^{N-1} \Bigg[ \dfrac{\tau \rho_s }{2 h} \int_{0}^{\ell} \left| w^k \circ \eta_{0} \right|^{2} dx + \dfrac{\tau}{2 h} \int_{\Omega} \rho^{k} \left| w^{k} \right|^{2} dz \Bigg] .
\end{aligned}
\end{equation}} Now, recalling the definition of $w^{k}$ in \eqref{wtaverage} and using Jensen's inequality, we notice that
$$
\begin{aligned}
& \sum_{k=0}^{N-1} \Bigg[ \dfrac{\tau \rho_s }{2 h} \int_{0}^{\ell} \left| w^k \circ \eta_{0} \right|^{2} dx + \dfrac{\tau}{2 h} \int_{\Omega} \rho^{k} \left| w^{k} \right|^{2} dz \Bigg] \\[6pt]
& \hspace{-5mm} = \sum_{k=0}^{N-1} \Bigg[ \dfrac{\tau \rho_s }{2 h} \int_{0}^{\ell} \left| \dfrac{1}{\tau} \int_{\tau k}^{\tau (k+1)} w \left( t, \eta_{0}(x) \right) dt \right|^{2} dx + \dfrac{\tau}{2 h} \int_{\Omega} \rho^{k} \left| \dfrac{1}{\tau} \int_{\tau k}^{\tau (k+1)} w \left( t, z \right) dt \right|^{2} dz \Bigg] \\[6pt]
& \hspace{-5mm} \leq \sum_{k=0}^{N-1} \Bigg[ \dfrac{\rho_s }{2 h} \int_{0}^{\ell} \int_{\tau k}^{\tau (k+1)} \left| w \left( t, \eta_{0}(x) \right) \right|^{2} dt \, dx + \dfrac{1}{2 h} \int_{\Omega} \rho^{k} \left( \int_{\tau k}^{\tau (k+1)} \left| w \left( t, z \right) \right|^{2} dt  \right) dz \Bigg] \\[6pt]
& \hspace{-5mm} \leq \dfrac{\rho_s }{2 h} \int_{0}^{\ell} \int_{0}^{h} \left| w \left( t, \eta_{0}(x) \right) \right|^{2} dt \, dx + \dfrac{\rho_{+} + \rho_{-}}{2 h} \int_{\Omega} \int_{0}^{h} \left| w \left( t, z \right) \right|^{2} dt \, dz \,
\end{aligned}
$$
which is then inserted into \eqref{apriorieu3} to yield \eqref{apriorieu2}. 
\end{proof}

\subsection{Proof of Theorem \ref{maintheoremdelay}, Step 2: Constructing interpolations}
Now we unfix $\tau$ and write the functions of the previous subsection as $\eta_{\tau}^{k}$, $u_{\tau}^{k}$ and $\Phi_{\tau}^{k}$ to prevent confusion. Using this and denoting $I_{k}(\tau) = [k \tau, (k+1) \tau)$, we define 
$$
\Omega^{(\tau)}(t) = \Omega^{k} \quad \text{for} \ \ t \in I_{k}(\tau),
$$
and the following time-dependent functions:
\begin{equation} \label{unfixedtau}
\left\{
\begin{aligned}
\eta^{(\tau)}(t,x) &= \eta_{\tau}^{k}(x) \qquad &\text{ for all } (t,x) \in I_{k}(\tau) \times [0,\ell],\\[3pt]
\overline{\eta}^{(\tau)}(t,x) &= \eta_{\tau}^{k+1}(x) \qquad &\text{ for all } (t,x) \in I_{k}(\tau) \times [0,\ell],\\[3pt]
\widetilde{\eta}^{(\tau)}(t,x) &= \frac{\tau (k+1)-t}{\tau} \eta_{\tau}^{k}(x) + \frac{t-\tau k}{\tau} \eta_{\tau}^{k+1}(x) \qquad &\text{ for all } (t,x) \in I_{k}(\tau) \times [0,\ell],\\[3pt]
u^{(\tau)}(t,z) &= u_{\tau}^{k}(z) \qquad &\text{ for all } (t,z) \in I_{k}(\tau) \times \Omega^{(\tau)}(t),\\[3pt]
\Phi^{(\tau)}(t,z) &= \Phi^{k-1}_{\tau}(z) \qquad &\text{ for all } (t,z) \in I_{k}(\tau) \times \Omega^{(\tau)}(t),\\[3pt]
\widetilde{\Phi}^{(\tau)}(t,z) &=\frac{\tau (k+1)-t}{\tau} \Phi^{k-1}_{\tau}(z) + \frac{t-\tau k}{\tau} \Phi^{k}_{\tau}(z) \qquad &\text{ for all } (t,z) \in I_{k}(\tau) \times \Omega^{(\tau)}(t).
\end{aligned}
\right.
\end{equation}
Note that this immediately implies
\begin{align*}
 u^{(\tau)} \left( t,\eta^{(\tau)}(t,x) \right) &= \partial_t \tilde{\eta}^{(\tau)}(t,x) \qquad &\text{ for all } (t,x) \in I \times [0,\ell].
\end{align*}
Using the estimates of Lemma \ref{apriorih}, we derive the following bounds on the functions defined in \eqref{unfixedtau}.

\begin{lemma}[Uniform bounds in $\tau$] \label{uniformbounds}
	The following sequences are uniformly bounded in $\tau$:
\begin{equation*} %
\left\{
\begin{aligned}
	\mathcal{E}_{K}(\eta^{(\tau)}(t,\cdot)) & \in L^\infty([0,h];\mathbb{R}),\\[3pt]
	\eta^{(\tau)}, \overline{\eta}^{(\tau)}, \widetilde{\eta}^{(\tau)} &\in L^\infty([0,h]; \mathcal{H}{\cap H^3((0,\ell))}),\\[3pt]
	\partial_t \widetilde{\eta}^{(\tau)} & \in L^2([0,h]; \mathcal{H}\cap H^3((0,\ell))),\\[3pt]
	u^{(\tau)} &\in L^2([0,h];H_{0}^{1}\cap H^3(\Omega;\R^2) \quad \text{and} \quad u^{(\tau)} \in L^2([0,h];C^{1,1^{-}}(\overline{\Omega};\R^2), \quad \\[3pt]
	u^{(\tau)} \circ \Phi^{(\tau)} &\in L^2([0,h]\times \Omega; \R^2).
\end{aligned}
\right.
\end{equation*}
Furthermore, by definition we have that $\partial_t \widetilde{\Phi}^{(\tau)} = u^{(\tau)} \circ \Phi^{(\tau)}$ in $\Omega$.
\end{lemma}
\noindent
\begin{proof} 
First, arguing as in \eqref{lowerbound1}-\eqref{lowerbound2},  from inequality \eqref{apriorieu1} we can infer the following:
\begin{align} \label{linfenergy00}
&\phantom{{}={}}\mathcal{E}_{K}(\eta^{j}) + \frac{\sqrt{\delta_0}}{2}\int_{0}^{\ell} | \partial_{x}^3  \eta^{j}|^2\, dx - \tau \dfrac{4 \rho_{\max}^{2} L^{2}}{\mu_{\min} \pi^{2}} \| f \|_{L^{2}(\Omega)}^{2} \nonumber \\
&\leq   \mathcal{E}_{K}(\eta^{j-1}) + \frac{\sqrt{\delta_0}}{2}\int_{0}^{\ell} | \partial_{x}^3  \eta^{j-1}|^2\, dx + \dfrac{\tau \rho_s }{2 h} \int_{0}^{\ell} \left| w^{j-1} \circ \eta_{0} \right|^{2} dx + \dfrac{\tau}{2 h} \int_{\Omega} \rho^{j-1} \left| w^{j-1} \right|^{2} dz,
\end{align}	
for every $j \in \{1, \cdots, N\}$, where $\mu_{\min} := \min\{\mu_{+}, \mu_{-}\}$ and $\rho_{\max} := \max\{\rho_{+}, \rho_{-}\}$. By taking the sum from $j=1$ to $j=k$ in both sides of inequality \eqref{linfenergy00}, we obtain:  
{\small
\begin{equation} \label{linfenergy0}
\begin{aligned}
& \mathcal{E}_{K}(\eta^{k}) + \frac{\sqrt{\delta_0}}{2}\int_{0}^{\ell} | \partial_{x}^3  \eta^{k}|^2\, dx \leq  \mathcal{E}_{K}(\eta^{0}) + \frac{\sqrt{\delta_0}}{2}\int_{0}^{\ell} | \partial_{x}^3  \eta^{0}|^2\, dx \\&+ \sum_{j=1}^{k} \Bigg[ \dfrac{\tau \rho_s }{2 h} \int_{0}^{\ell} \left| w^{j-1} \circ \eta_{0} \right|^{2} dx + \dfrac{\tau}{2 h} \int_{\Omega} \rho^{j-1} \left| w^{j-1} \right|^{2} dz + \tau \dfrac{4 \rho_{\max}^{2} L^{2}}{\mu_{\min} \pi^{2}} \| f \|_{L^{2}(\Omega)}^{2} \Bigg] \\[6pt]
& \leq  \mathcal{E}_{K}(\eta^{0}) + \frac{\sqrt{\delta_0}}{2}\int_{0}^{\ell} | \partial_{x}^3  \eta^{0}|^2\, dx + \dfrac{\rho_s }{2 h} \int_{0}^{\ell} \int_{0}^{h} \left| w \left( t, \eta_{0}(x) \right) \right|^{2} dt \, dx \\&+ \dfrac{\rho_{+} + \rho_{-}}{2 h} \int_{\Omega} \int_{0}^{h} \left| w \left( t, z \right) \right|^{2} dt \, dz + \dfrac{4 h \rho_{\max}^{2} L^{2}}{\mu_{\min} \pi^{2}} \| f \|_{L^{2}(\Omega)}^{2} \, ,
\end{aligned}
\end{equation}}for every $k \in \{1, \cdots, N\}$. Since the right-hand side of \eqref{linfenergy0} is a (finite) constant that is independent of $t$ and $\tau$ (denoted henceforth by $C_{0}$), we obtain an uniform bound on $\mathcal{E}_{K}(\eta^{k}_\tau)$ and $\smash{\norm[H_0^3]{\eta^k_\tau}}$ and thus an $L^\infty$-bound on $\mathcal{E}_{K}(\eta^{(\tau)}(t,.))$ and $\smash{\norm[H_0^3]{\eta^{(\tau)}(t,.)}}$ over the time-interval $[0,h]$. By the coercivity of the energy functional $\mathcal{E}_{K}$, see Proposition \ref{propenergy}, inequality \eqref{linfenergy0} also yields a uniform bound on the $\mathcal{H}$-norm of $\eta^{k}$, thus resulting in a $L^\infty([0,h]; \mathcal{H})$-bound for $\eta^{(\tau)}$. Regarding $\widetilde{\eta}^{(\tau)}$, we notice that:
$$
\| \widetilde{\eta}^{(\tau)}(t,\cdot) \|_{\mathcal{H}} \leq \frac{\tau (k+1)-t}{\tau} \| \eta_{\tau}^{k} \|_{\mathcal{H}} + \frac{t-\tau k}{\tau} \| \eta_{\tau}^{k+1} \|_{\mathcal{H}} \leq C_{0} \text{ for all } t \in [0,h],
$$
from which we also derive a uniform $L^\infty([0,h]; \mathcal{H})$-bound on $\eta^\tau$. A similar calculation holds for the $H^3$-norm. On the other hand, in view of \eqref{apriorieu2} and arguing as in \eqref{linfenergy00}, we have
$$
\begin{aligned}
\int_{0}^{h} \| \partial_{x}^3 \, \partial_{t} \widetilde{\eta}^{(\tau)}(t,\cdot) \|^{2}_{L^{2}(0,\ell)} \, dt = \dfrac{1}{\tau} \sum_{k=0}^{N-1} \int_{0}^{\ell} \left| \partial_{x}^3 \left( \eta_{\tau}^{k+1}(x) - \eta_{\tau}^{k}(x) \right) \right|^{2} \, dx \leq \dfrac{2 C_{0}}{\varepsilon_{0}} \, .
\end{aligned}
$$
Using Poincaré's inequality this then extends into an uniform $L^2([0,h]; \mathcal{H})$-bound on $\partial_t \tilde{\eta}^{(\tau)}$. The uniform bound on the velocity gradient is obtained in a similar way, since
$$
\begin{aligned}
\int_0^{h} \| \nabla (\Delta u^{(\tau)}) \|^{2}_{L^{2}(\Omega)} dt = \tau \sum_{k=0}^{N-1} \| \nabla (\Delta u_{\tau}^{k+1}) \|^{2}_{L^{2}(\Omega)} \leq \dfrac{2 C_{0}}{\delta_{0}} \, ,
\end{aligned}
$$
which, in view of the (compact) embedding $H_{0}^{3}(\Omega) \subset C^{1,1^{-}}(\overline{\Omega})$ (here $C^{1,1^-}$ is used to denote $C^{1,1-\varepsilon}$ for any $\varepsilon > 0$), also yields a uniform for $u^{(\tau)}$ in the space $L^2([0,h];C^{1,1^{-}}(\overline{\Omega};\R^2)$. Finally, we consider the function $u^{(\tau)} \circ \Phi^{(\tau)}$, which can be handled in the following way:
\begin{align*}
& \int_0^{h} \| u^{(\tau)} \circ \Phi^{(\tau)} \|^{2}_{L^{2}(\Omega)} dt = \sum_{k=0}^{N-1} \tau \| u_{\tau}^{k+1} \circ \Phi_{\tau}^{k} \|^{2}_{L^{2}(\Omega)} \leq 2 \tau \sum_{k=0}^{N-1} \left (\|u_{\tau}^{k+1} \circ \Phi_{\tau}^{k} - w^k\|^{2}_{L^{2}(\Omega)} + \| w^k \|^{2}_{L^{2}(\Omega)} \right) \\[6pt]
& \leq \dfrac{4 h C}{\min\{\rho_{+}, \rho_{-}\}} + 2 \tau \sum_{k=0}^{N-1} \| w^k \|^{2}_{L^{2}(\Omega)} \leq \dfrac{4 h C}{\min\{\rho_{+}, \rho_{-}\}} + 2 \int_{\Omega} \int_{0}^{h} \left| w \left( t, z \right) \right|^{2} dt \, dz \, 
\qedhere
\end{align*}
\end{proof}
We now arrive at the main difficulty in implementing the scheme, establishing the properties of and bounds on $\Phi^{(\tau)}$, because $\Phi^{(\tau)}$ is defined via concatenation of an unbounded (for $\tau \to 0$) number of functions and is thus highly nonlinear. As any linearizing would break the coupling properties needed, we will instead rely on the $\delta$-terms proving a high enough regularity for the constituting functions. Very similar estimates have been shown in~\cite{BenKamSch20}. We do repeat them here to emphasize the methodological possibilities in the shell-regime of fluid-structure interactions.

\begin{proposition} \label{prop:regularityOfV}
There exists a constant $\mathcal{K} > 0$, depending on $w$, $\delta_{0}$, $h$, $\mathcal{E}_{K}(\eta_0)$ and $f$, such that
\begin{equation} \label{uniformkappa}
\sum_{j=0}^{N-1} \tau \left\| u_{\tau}^{j+1} \right\|^{2}_{C^{1,\lambda}(\overline{\Omega})} \leq \mathcal{K} \, , \quad \text{for any} \ \ \tau > 0 \quad \text{and} \quad \lambda \in (0,1).
\end{equation}	
For any $\tau > 0$ the function $\Phi^{(\tau)}$ is uniformly Lipschitz in the spatial variable; more precisely, for every $k \in \{1,\dots, N \}$ we have that $\Phi_{\tau}^{k} \in C^{1}(\overline{\Omega})$ is such that
\begin{equation} \label{lip0}
| \Phi_{\tau}^{k}(z_1) - \Phi_{\tau}^{k}(z_2) | \leq \exp \left( \sqrt{\mathcal{K} h} \right) |z_{1} - z_{2} | \text{ for all } z_{1}, z_{2} \in \overline{\Omega}.
\end{equation}
On the other hand, for any $\tau \in (0, 1/4 \mathcal{K})$ and any integer $k \geq 0$ such that $\tau k < h$, the function $\Phi_{\tau}^{k}: \Omega \longrightarrow \Omega$ is a diffeomorphism with 
\begin{equation} \label{uniformkappa2}
e^{-4 \tau \mathcal{K}} \leq \det (\nabla \Phi_{\tau}^{k})(z) \leq e^{2 \tau \mathcal{K}} \, \text{ for all } z \in \overline{\Omega},
\end{equation}
so that, in particular, $\lim\limits_{\tau \to 0} \det (\nabla \Phi_{\tau}^{k})(z) = 1$, for every $z \in \overline{\Omega}$. 
\end{proposition}

\noindent
\begin{proof} The existence of a constant $\mathcal{K} > 0$ (depending on $w$, $h$, $\mathcal{E}_{K}(\eta_0)$ and $f$) satisfying \eqref{uniformkappa} follows directly from the uniform bounds obtained in Lemma \ref{uniformbounds}, in view of the (compact) embedding $H^{1}(\Omega) \subset C^{1,\lambda}(\overline{\Omega})$. Now, $\Phi_0 : \Omega \longrightarrow \Omega$ is the identity function and for every $k \in \{0,\dots, N-1 \}$ we have, by definition
\begin{equation} \label{nextflowmap}
\Phi_{\tau}^{k+1}  = (\operatorname{id} + \tau u_\tau^{k+1} )\circ \Phi_{\tau}^k = \Phi_{\tau}^{k} + \tau (u_{\tau}^{k+1} \circ \Phi_{\tau}^{k}) \, .
\end{equation}
Therefore $\Phi_{\tau}^{k+1} \in C^{1}(\overline{\Omega})$ and, by recursion,
\begin{equation} \label{lip1}
\text{Lip}(\Phi_{\tau}^{k+1}) \leq \prod_{j=0}^{k} \left[ 1 + \tau \text{Lip}(u_{\tau}^{j+1}) \right].
\end{equation}
By the inequality between arithmetric and geometric mean and by properties of the exponential function, we can further bound the right-hand side of \eqref{lip1} in the following way:
\begin{equation} \label{lip2}
\begin{aligned}
\text{Lip}(\Phi_{\tau}^{k+1}) & \leq \left[ \dfrac{1}{k+1} \sum_{j=0}^{k} \left(  1 + \tau \text{Lip}(u_{\tau}^{j+1}) \right) \right]^{k+1} \leq \left[ 1+ \dfrac{1}{k+1} \sum_{j=0}^{N-1} \tau \text{Lip}(u_{\tau}^{j+1}) \right]^{k+1} \\[3pt]
& \leq \left[ 1+ \dfrac{1}{k+1} \sqrt{\sum_{j=0}^{N-1} \tau} \, \sqrt{\sum_{j=0}^{N-1} \tau \text{Lip}(u_{\tau}^{j+1})^2} \, \right]^{k+1} \leq \left[ 1+ \dfrac{\sqrt{\mathcal{K} h}}{k+1}  \, \right]^{k+1} \leq \exp \left( \sqrt{\mathcal{K} h} \right) \, ,
\end{aligned}
\end{equation}
where as a reminder $N = [h / \tau]$. We thus prove \eqref{lip0}.

Now, fix any $\tau \in (0, 1/4 \mathcal{K})$. We proceed inductively: Since $\Phi_0 : \Omega \longrightarrow \Omega$ is the identity function, it is a diffeomorphism with $\det (\nabla \Phi_{0}) \equiv 1$. For some integer $k \in \{1,\dots, N_{0}-1 \}$, assume that $\Phi_k:\Omega \longrightarrow \Omega^{k}$ is a diffeomorphism such that
$$
e^{-4 \tau \mathcal{K}} \leq \det (\nabla \Phi_{\tau}^{k})(z) \leq e^{2 \tau \mathcal{K}} \, \text{ for all } z \in \overline{\Omega}.
$$
We then apply standard properties of the determinant to infer that
\begin{equation} \label{detflowmap1}
\det(\nabla \Phi_{\tau}^{k+1}) =  \prod_{j=0}^{k} \det \left[ \mathrm{I} + \tau \nabla u_\tau^{j+1} (\Phi_\tau^j) \right] = \prod_{j=0}^{k} \left[ 1 + \tau \mathrm{I} : \nabla u_\tau^{j+1} (\Phi_\tau^j)+ \tau^2 \det \nabla u_\tau^{j+1} (\Phi_\tau^j) \right]   %
\end{equation}
where the middle term can be rewritten as $\tau \nabla \cdot u_\tau^{j+1} = 0$ since we are dealing with a divergence free flow. For the last term we note that %
\begin{equation} \label{detflowmap2} 
\left| \det(\nabla u_{\tau}^{j+1})(\Phi_{\tau}^{j})(z) \right| \leq 2 \left\| \nabla u_{\tau}^{j+1}(\Phi_{\tau}^{j}) \right\|^{2}_{L^{\infty}(\Omega)} \leq 2 \left\| \nabla u_{\tau}^{j+1} \right\|^{2}_{L^{\infty}(\Omega)} \, .
\end{equation}
From \eqref{detflowmap1} and \eqref{detflowmap2} we derive pointwise upper and lower bounds for $\det(\nabla \Phi_{\tau}^{k+1})$; indeed, we proceed as in \eqref{lip2} to obtain the following:
\begin{equation} \label{detflowmap3}
\begin{aligned}
\det(\nabla \Phi_{\tau}^{k+1}) & \leq \prod_{j=0}^{k} \left[ 1 + 2 \tau^{2} \left\| \nabla u_{\tau}^{j+1} \right\|^{2}_{L^{\infty}(\Omega)} \right] \leq \left[ \dfrac{1}{k+1} \sum_{j=0}^{k} \left( 1 + 2 \tau^{2} \left\| \nabla u_{\tau}^{j+1} \right\|^{2}_{L^{\infty}(\Omega)} \right) \right]^{k+1} \\[3pt]
& \leq \left[ 1 + \dfrac{2 \tau}{k+1} \sum_{j=0}^{N_{0}-1}  \tau \left\| \nabla u_{\tau}^{j+1} \right\|^{2}_{L^{\infty}(\Omega)} \right]^{k+1} \leq \left[ 1 + \dfrac{2 \tau \mathcal{K}}{k+1} \right]^{k+1} \leq e^{2 \tau \mathcal{K}} \, .
\end{aligned}
\end{equation}
On the other hand, since $4 \tau \mathcal{K} < 1$, for every $j \in \{0,\dots, N_{0} -1 \}$ we have
$$
\tau^{2} \left\| \nabla u_{\tau}^{j+1} \right\|^{2}_{L^{\infty}(\Omega)} \leq \tau^{2} \sum_{j=0}^{N_{0}-1} \left\| \nabla u_{\tau}^{j+1} \right\|^{2}_{L^{\infty}(\Omega)} \leq \tau \mathcal{K} < \dfrac{1}{4} \, ,
$$
thus implying that
$$
\dfrac{1}{1 - 2 \tau^{2} \| \nabla u_{\tau}^{j+1} \|^{2}_{L^{\infty}(\Omega)}} \leq 1 +4 \tau^{2} \left\| \nabla u_{\tau}^{j+1} \right\|^{2}_{L^{\infty}(\Omega)} \text{ for all } j \in \{0,\dots, N_{0}-1 \}.
$$
The last inequality enables us to derive the following (pointwise) lower bound for $\det(\nabla \Phi_{\tau}^{k+1})$:
\begin{equation} \label{detflowmap4}
\begin{aligned}
\det(\nabla \Phi_{\tau}^{k+1})^{-1} & = \prod_{j=0}^{k} \left[ 1 + \tau^{2} \det(\nabla u_{\tau}^{j+1})(\Phi_{\tau}^{j}) \right]^{-1} \leq \prod_{j=0}^{k} \left[ 1 - 2 \tau^{2} \left\| \nabla u_{\tau}^{j+1} \right\|^{2}_{L^{\infty}(\Omega)} \right]^{-1} \\[3pt]
& \leq \prod_{j=0}^{k} \left[ 1 + 4 \tau^{2} \left\| \nabla u_{\tau}^{j+1} \right\|^{2}_{L^{\infty}(\Omega)} \right] \leq \left[ \dfrac{1}{k+1} \sum_{j=0}^{k} \left( 1 + 4 \tau^{2} \left\| \nabla u_{\tau}^{j+1} \right\|^{2}_{L^{\infty}(\Omega)} \right) \right]^{k+1} \\[3pt]
& \leq \left[ 1 + \dfrac{4 \tau^{2}}{k+1} \sum_{j=0}^{N_{0}-1}  \tau^{2} \left\| \nabla u_{\tau}^{j+1} \right\|^{2}_{L^{\infty}(\Omega)} \right]^{k+1} \leq \left[ 1 + \dfrac{4 \tau \mathcal{K}}{k+1} \right]^{k+1} \leq e^{4 \tau \mathcal{K}} \, ,
\end{aligned}
\end{equation}
from which we can conclude the other half of \eqref{uniformkappa2}. In particular, $\det(\nabla \Phi_{\tau}^{k+1})(z) > 0 $ for every $z \in \overline{\Omega}$. Since all $\Phi_\tau^k$ are equal to the identity on $\partial \Omega$, together with a degree argument, his proves that $\Phi_{\tau}^{k+1}$ is a diffeomorphism.
\end{proof}

Note that this already implies that $\Omega^{(\tau)}(t)$ is diffeomorphic to the initial domain and thus that there is no collision in this approximation.

\subsection{Proof of Theorem \ref{maintheoremdelay}, Step 3: Establishing the limit equation}

In view of Lemma \ref{uniformbounds} and the Banach–Alaoglu theorem, we may find functions $\eta \in H^{1}([0,h]; \mathcal{H})$, $u \in L^2([0,h];H_{0}^{3}(\Omega;\R^2)$, $\Phi \in C([0,h];C^{0,1}(\overline{\Omega};\R^2)$ such that the following limits hold as $\tau \to 0$:
\begin{align*} 
\eta^{(\tau)},\overline{\eta}^{(\tau)},\widetilde{\eta}^{(\tau)} &\rightharpoonup^* \eta & \text{ in } \ & L^\infty([0,h]; \mathcal{H} \cap H^3((0,\ell))),\\[3pt]
\partial_t \widetilde{\eta}^{(\tau)} & \rightharpoonup \partial_t \eta& \text{ in } \ & L^2([0,h]; \mathcal{H} \cap H^3((0,\ell))),\\[3pt]
u^{(\tau)} & \rightharpoonup u &\text{ in } \ & L^2([0,h];H_{0}^{1}\cap H^3(\Omega;\R^2),\\[3pt]
\Phi^{(\tau)} & \to \Phi & \text{ in } \ & C([0,h];C^{0,1}(\overline{\Omega};\R^2),
\end{align*}
for sub-sequences that are not being relabeled. Furthermore, by interpolation (see for instance \cite[Proposition 2.20]{BenKamSch20}) we can establish that $\eta \in C([0,h]; C^{1}([0,\ell];\mathbb{R}^{2}))$ and that it verifies the following convergences as $\tau \to 0$: %
\begin{equation} \label{uniformeta}
\begin{aligned} 
\eta^{(\tau)} & \to \eta & \text{ in } \ & L^\infty([0,h]; C^{1}([0,\ell];\mathbb{R}^{2})),\\[3pt]
\widetilde{\eta}^{(\tau)} & \to \eta & \text{ in } \ & C([0,h]; C^{1}([0,\ell];\mathbb{R}^{2})).
\end{aligned}
\end{equation}

Next we define $\Omega^{\eta}(t)$ as the time-dependent fluid domain given by the limit deformation $\eta(t)$ at time $t \in [0,h]$. In particular, due to Proposition \ref{prop:regularityOfV} we know that $\Phi$ is Lipschitz with constant $\exp(\sqrt{Lh})$ and that $\det \nabla \Phi \equiv 1$ almost everywhere in $\Omega$. We also remark that $\Phi(t,.):\Omega \longrightarrow \Omega$ is again a volume preserving diffeomorphism by the same degree-argument as before. As a direct consequence of this and Proposition \ref{propenergy}, the image of the solid $\eta(t,(0,\ell))$ and $\partial \Omega$ cannot intersect.

Now, in view of Proposition \eqref{minfunctional1} we have the identity \begin{align} \label{equasistaticSolmin2} 
& \int_{0}^{\ell} \mathcal{L}_{\Gamma}(\eta_{\tau}^{k+1}) \cdot \xi \, dx + \dfrac{\rho_s }{h} \int_{0}^{\ell} \left( \dfrac{\eta_{\tau}^{k+1} - \eta_{\tau}^{k}}{\tau} -  w_{\tau}^k \circ \eta_{0} \right) \cdot \xi \, dx + \varepsilon_0 \int_{0}^{\ell} \partial_{x}^3 \left( \dfrac{ \eta_{\tau}^{k+1} - \eta_{\tau}^{k}}{\tau} \right) \cdot \partial_{x}^3 \xi \, dx \\ 
  +&\sqrt{\delta_0}\int_{0}^{\ell}  \partial_{x}^3  \eta^{k+1}\partial_x^3\xi\, dx  + \dfrac{1}{h} \int_{\Omega} \rho_{\tau}^{k} \left(u_{\tau}^{k+1} \circ \Phi_{\tau}^{k} - w_{\tau}^{k} \right) \cdot \left( q \circ \Phi_{\tau}^{k} \right) dz  + \delta_{0} \int_{\Omega}  \nabla(\Delta u_{\tau}^{k+1}) \cdot \nabla (\Delta q) \, dz \nonumber \\
 +& \int_{\Omega} \mu_{\tau}^{k} \, \nablasym u_{\tau}^{k+1}  \cdot \nablasym(q) \, dz = \int_{\Omega} \rho_{\tau}^{k} \left( f \cdot q \right) dz, \nonumber
\end{align}
for every fixed $k \in \{0,\dots,N-1\}$, for all (time-independent) coupled test functions $
(\xi,q )\in \mathcal{H}\cap H^3(0,\ell)\times  H_{0}^{1}\cap H^3(\Omega;\mathbb{R}^{2})
$
such that 
$
\nabla \cdot q = 0\text{ in } \Omega$ and $q \circ \eta^{k}_\tau = \xi \text{ in }[0,\ell].
$

Therefore, by multiplying by $\tau$ and adding (over $k$) all the identities in \eqref{equasistaticSolmin2} we deduce that
\begin{equation} \label{equasistaticSolmin3} 
\begin{aligned}
& \int_{0}^{h} \int_{0}^{\ell} \mathcal{L}_{\Gamma}(\overline{\eta}^{(\tau)}) \cdot \xi \, dx \, dt +\sqrt{\delta_0}\int_{0}^{\ell}  \partial_{x}^3  \overline{\eta}\partial_x^3\xi\, dx + \dfrac{\rho_s }{h} \int_{0}^{h} \int_{0}^{\ell} \left( \partial_{t} \widetilde{\eta}^{(\tau)} -  w^{(\tau)} \circ \eta_{0} \right) \cdot \xi \, dx \, dt \\
&+ \varepsilon_0 \int_{0}^{h} \int_{0}^{\ell} \partial_{x}^3 \partial_{t} \widetilde{\eta}^{(\tau)} \cdot \partial_{x}^3 \xi \, dx \, dt  
  + \dfrac{1}{h} \int_{0}^{h} \int_{\Omega} \rho^{(\tau)} \left(u^{(\tau)} \circ \Phi^{(\tau)} - w^{(\tau)} \right) \cdot \left( q \circ \Phi^{(\tau)} \right) dz \, dt  \\&+ \delta_{0} \int_{0}^{h} \int_{\Omega}  \nabla(\Delta u^{(\tau)}) \cdot \nabla (\Delta q) \, dz \, dt  + \int_{0}^{h} \int_{\Omega} \mu^{(\tau)} \, \nablasym (u^{(\tau)})  \cdot \nablasym(q) \, dz \, dt = \int_{0}^{h} \int_{\Omega} \rho^{(\tau)} \left( f \cdot q \right) dz \, dt \, ,
\end{aligned}
\end{equation}
for all time-dependent coupled test functions $
(\xi,q )\in L^2([0,h],\mathcal{H}\cap H^3(0,\ell))\times   L^2([0,h],H_{0}^{1}\cap H^3(\Omega;\mathbb{R}^{2}))
$
such that 
$
\nabla \cdot q = 0\text{ in }[0,h]\times \Omega$ and $q \circ \eta^{(\tau)} = \xi \text{ in }[0,h]\times [0,\ell].
$

Finally we can conclude that
\begin{align*}
\partial_t \Phi = \lim_{\tau \to 0} \partial_t \widetilde{\Phi}^{(\tau)} = \lim_{\tau \to 0} u^{(\tau)} \circ \Phi^{(\tau)} = u \circ \Phi \quad \text{almost everywhere in} \ \Omega.
\end{align*}
Then, $\Phi$ has the properties required for a solution. Moreover, from the definitions in \eqref{unfixedtau} we have
$$
\partial_t \widetilde{\eta}^{(\tau)}(t,x) = \frac{\eta_{\tau}^{k+1}(x)-\eta_{\tau}^{k}(x)}{\tau} = u^{(\tau)} \left( t,\eta^{(\tau)}(t,x) \right) \text{ for all } (t,x) \in I_{k}(\tau) \times [0,\ell],
$$
so that, for every $\phi \in C_{0}^{\infty}([0,h] \times [0,\ell]; \mathbb{R}^2)$, we have
\begin{equation} \label{couplingtest}
\begin{aligned}
& \int_{0}^{h} \int_{0}^{\ell} \partial_t \eta \cdot \phi \, dx \, dt = \lim_{\tau \to 0} \int_{0}^{h} \int_{0}^{\ell} \partial_t \widetilde{\eta}^{(\tau)} \cdot \phi \, dx \, dt = \lim_{\tau \to 0} \int_{0}^{h} \int_{0}^{\ell} \left( u^{(\tau)} \circ \eta^{(\tau)} \right) \cdot \phi \, dx \, dt \\[3pt]
& = \lim_{\tau \to 0} \left[ \int_{0}^{h} \int_{0}^{\ell} \left( u^{(\tau)} \circ \eta \right) \cdot \phi \, dx \, dt + \int_{0}^{h} \int_{0}^{\ell} \left( u^{(\tau)} \circ \eta^{(\tau)} - u^{(\tau)} \circ \eta \right) \cdot \phi \, dx \, dt \right] .
\end{aligned}
\end{equation}
On one hand, by weak compactness we have 
\begin{equation} \label{trace}
\lim_{\tau \to 0} \int_{0}^{h} \int_{0}^{\ell} \left( u^{(\tau)} \circ \eta \right) \cdot \phi \, dx \, dt = \int_{0}^{h} \int_{0}^{\ell} \left( u \circ \eta \right) \cdot \phi \, dx \, dt \, .
\end{equation}
On the other hand, for $s \in [0,1]$ we define
$$
\pi_{s}(t,x) = s \eta^{(\tau)}(t,x) + (1-s) \eta(t,x) \text{ for all } (t,x) \in I_{k}(\tau) \times [0,\ell],
$$
so that for every $(t,x) \in I_{k}(\tau) \times [0,\ell]$ we have
\begin{equation} \label{diffeta}
\begin{aligned}
& \left| u^{(\tau)} ( t , \eta^{(\tau)}(t,x) ) - u^{(\tau)} \left( t , \eta(t,x) \right) \right|^{2} = \left| \int_{0}^{1} \dfrac{d}{d s} u^{(\tau)} ( t , \pi_{s}(t,x) ) \, ds \right|^{2} \\[3pt]
& \hspace{-5mm} \leq \int_{0}^{1} \left| \nabla u^{(\tau)} ( t , \pi_{s}(t,x) ) \cdot (\eta^{(\tau)}(t,x) - \eta(t,x)) \right|^{2} ds \\[3pt]
& \hspace{-5mm} \leq \left( \int_{0}^{1} \left| \nabla u^{(\tau)} ( t , \pi_{s}(t,x) ) \right|^{2} ds \right) \sup_{(t,x) \in [0,h] \times [0,\ell]} \left| \eta^{(\tau)}(t,x) - \eta(t,x) \right|^{2} \, .
\end{aligned}
\end{equation}
Now, as $\eta^{(\tau)}$ and $\eta$ are both injective maps with lower bound on their spatial derivative and uniformly
close gradients, the linear interpolation $\pi_{s}$ has to obey the same property. Therefore, by integrating in \eqref{diffeta} we get by the trace theorem
\begin{equation} \label{diffeta2}
\begin{aligned}
& \int_{0}^{h} \int_{0}^{\ell} \left| u^{(\tau)} ( t , \eta^{(\tau)}(t,x) ) - u^{(\tau)} \left( t , \eta(t,x) \right) \right|^{2} \, dx \, dt \\[3pt]
& \hspace{-5mm} \leq \left( \int_{0}^{h} \int_{0}^{\ell} \int_{0}^{1} \left| \nabla u^{(\tau)} ( t , \pi_{s}(t,x) ) \right|^{2} ds \, dx \, dt \right) \sup_{(t,x) \in [0,h] \times [0,\ell]} \left| \eta^{(\tau)}(t,x) - \eta(t,x) \right|^{2} \\[3pt]
& \hspace{-5mm} \leq C \left( \int_{0}^{h}  \norm[H^2(\Omega)]{\smash{u^{(\tau)}}}^{2} \right) \sup_{(t,x) \in [0,h] \times [0,\ell]} \left| \eta^{(\tau)}(t,x) - \eta(t,x) \right|^{2}
\end{aligned}
\end{equation}
But since we have uniform $L^2([0,h];H_{0}^{3}(\overline{\Omega};\R^2))$-bounds on $u^{(\tau)}$ and the uniform convergence for $\eta^{(\tau)}$ towards $\eta$ given in \eqref{uniformeta}, the right-hand side of \eqref{diffeta2} vanishes as $\tau \to 0$. This, combined with \eqref{trace}, implies that \eqref{couplingtest} yields
$$
\int_{0}^{h} \int_{0}^{\ell} \partial_t \eta \cdot \phi \, dx \, dt = \int_{0}^{h} \int_{0}^{\ell} \left( u \circ \eta \right) \cdot \phi \, dx \, dt \text{ for all } \phi \in C_{0}^{\infty}([0,h] \times [0,\ell]; \mathbb{R}^2),
$$
and thus $\partial_{t} \eta = u \circ \eta$ almost everywhere on $[0,\ell]$. 

Finally we consider convergence of the actual equation. We need to show that \eqref{equasistaticSol} holds. We do so in two parts:

\subsubsection*{The equation in the interior of the fluid}
First consider $q \in C_0^1\cap C^\infty([0,h]\times \Omega;\R^2)$ such that $\nabla \cdot q = 0$ in $\Omega$ and $\supp q \subset \Omega^+ \cup \Omega^-$. Then since $\eta^{\tau} \to \eta$ uniformly, we have $q \circ \eta^{(\tau)}(t,x) = 0$ for all $\tau$ small enough. Thus the pair $(0,q)$ is admissible for \eqref{equasistaticSolmin3} for all $\tau$ small enough and by sending $\tau \to 0$ we get \eqref{equasistaticSol} for all such $q$. By a density argument, we can then extend this to all $q \in C([0,h];H^3\cap H^1_0(\Omega;\R^2)$ with $\nabla \cdot q = 0$ and $q\circ \eta = 0$.

\subsubsection*{The equation over the interface--introduction of the bulk-pressure.}
Here we wish to pass to the limit with test-functions that are situated on the structure. For that take $\xi \in C_0^1\cap C^\infty([0,h] \times [0,\ell];\R^2)$ such that $\int_0^\ell \xi(t) \wedge \partial_x \eta(t) dx = 0$ for all $t \in [0,\ell]$. The latter condition does precisely mean that it is the trace of a solenoidal function~\eqref{eq:solenoid}.
Since, in general $\int_0^\ell \xi(t) \wedge \partial_x \eta^\tau(t) dx \neq 0$ the respective error has to be compensated with the respective scalar differential pressure that we will introduce here and which we denote by $P^\tau:[0,T]\to \R$.

For that we choose ${\xi}_0(x):= x(\ell-x): C_0^\infty([0,\ell];\R^2)$, such that $\int_0^\ell \partial_x \eta^{(\tau)}(t) \wedge {\xi}_0 dx =\lambda_0^\tau(t) \neq 0$. Next we use the extension Proposition \ref{prop:arbMeanExt} to find $q_0\in L^\infty(I;H^1_0(\Omega;\R^2)$ such that $q_0(t) \circ \eta^{(\tau)}(t) = {\xi}_0$ and
\begin{align*}
 \nabla \cdot q_0(t) = \lambda^\tau (t) (\psi^+-\psi^-)
\end{align*}
where $\psi^\pm$ are time-independent ``bump-functions'' with $\supp \psi^\pm \subset \Omega_{\pm}(t)$ for all $t \in [0,T]$ and $\tau$ small enough and with $\int_\Omega \psi^\pm dx = 1$.

We define $ P^{(\tau)}(t)$ as a rescaling of the resulting error in the weak equation via
\begin{equation}\label{eq:reducedPressureDiscrete}
\begin{aligned} 
 P^{(\tau)}(t)\lambda_0^\tau(t) :=
&  \int_{0}^{\ell} \mathcal{L}_{\Gamma}(\overline{\eta}^{(\tau)}) \cdot \xi_0 \, dx  +\sqrt{\delta_0}\int_{0}^{\ell}  \partial_{x}^3  \overline{\eta}^{(\tau)}\partial_x^3\xi_0\, dx  + \dfrac{\rho_s }{h}  \int_{0}^{\ell} \left( \partial_{t} \widetilde{\eta}^{(\tau)} -  w^{(\tau)} \circ \eta_{0} \right) \cdot \xi_0 \, dx  \\&+ \varepsilon_0  \int_{0}^{\ell} \partial_{x}^3 \partial_{t} \widetilde{\eta}^{(\tau)} \cdot \partial_{x}^3 \xi_0 \, dx 
  + \dfrac{1}{h}  \int_{\Omega} \rho^{(\tau)} \left(u^{(\tau)} \circ \Phi^{(\tau)} - w^{(\tau)} \right) \cdot \left( q_0 \circ \Phi^{(\tau)} \right) dz   \\
  &+ \delta_{0}  \int_{\Omega}  \nabla(\Delta u^{(\tau)}) \cdot \nabla (\Delta q_0) \, dz   +  \int_{\Omega} \mu^{(\tau)} \, \nablasym  u^{(\tau)}  \cdot \nablasym(q_0) \, dz  -  \int_{\Omega} \rho^{(\tau)} \left( f \cdot q_0 \right) dz. 
\end{aligned}
\end{equation}
 The construction and the a-priori estimates above imply that in $ P^{(\tau)}\in L^2(I)$ uniformly in $\tau$.

We now turn back to a general $\xi \in C_0^\infty([0,h] \times [0,\ell];\R^2)$ which is suitable for the limit equation i.e.\ $\int_0^\ell \xi(t) \wedge \partial_x \eta(t) dx = 0$ for all $t \in [0,\ell]$.
With this, for $\tau >0$, we can use Proposition \ref{prop:arbMeanExt} to find that the pair $(\xi,\ext{\eta^{(\tau)}}(\xi))$
fulfils \eqref{equasistaticSolmin3} plus a pressure term acting on
\[
\lambda^\tau(t)=\frac{1}{\lambda^\tau_0(t)}\int_0^\ell \xi(t) \wedge \partial_x \eta^\tau(t) dx \to 0\text{ with }\tau \to 0.
\]
Indeed we find that by the $H^3$ estimates of $\eta^{(\tau)}$ that also $\nabla \Delta \ext{\eta^{(\tau)}}(\xi))\in L^2$. Hence
\begin{align}
\label{eq:press1}
\begin{aligned} 
& \int_{0}^{h} \int_{0}^{\ell} \mathcal{L}_{\Gamma}(\overline{\eta}^{(\tau)}) \cdot \xi \, dx \, dt +\sqrt{\delta_0}\int_{0}^{\ell}  \partial_{x}^3  \overline{\eta}^{(\tau)}\partial_x^3\xi\, dx + \dfrac{\rho_s }{h} \int_{0}^{h} \int_{0}^{\ell} \left( \partial_{t} \widetilde{\eta}^{(\tau)} -  w^{(\tau)} \circ \eta_{0} \right) \cdot \xi \, dx \, dt \\
&+ \varepsilon_0 \int_{0}^{h} \int_{0}^{\ell} \partial_{x}^3 \partial_{t} \widetilde{\eta}^{(\tau)} \cdot \partial_{x}^3 \xi \, dx \, dt   + \dfrac{1}{h} \int_{0}^{h} \int_{\Omega} \rho^{(\tau)} \left(u^{(\tau)} \circ \Phi^{(\tau)} - w^{(\tau)} \right) \cdot \left( \ext{\eta^{(\tau)}}(\xi) \circ \Phi^{(\tau)} \right) dz \, dt  \\
&+ \delta_{0} \int_{0}^{h} \int_{\Omega}  \nabla(\Delta u^{(\tau)}) \cdot \nabla (\Delta \ext{\eta^{(\tau)}}(\xi)) \, dz \, dt  + \int_{0}^{h} \int_{\Omega} \mu^{(\tau)} \,\nablasym u^{(\tau)} \cdot \nablasym(\ext{\eta^{(\tau)}}(\xi)) \, dz \, dt \\&= \int_{0}^{h} \int_{\Omega} \rho^{(\tau)} \left( f \cdot \ext{\eta^{(\tau)}}(\xi) \right) dz \, dt \,  + \int_0^{h} \lambda^{\tau}(t) P^{(\tau)}(t)\, dt.
\end{aligned}
\end{align}
Since all the terms in the equation converge, $\lambda^{(\tau)} \to 0$ and $P^{(\tau)}$ is bounded, we get \eqref{equasistaticSol} for the pair $(\xi,\ext{\eta}(\xi))$. Since every pair $(\xi,q)$ that is an admissible test function for \eqref{equasistaticSol} can be decomposed into
\[
(\xi,q)=(\xi,\ext{\eta}(\xi))+(0, q-\ext{\eta}(\xi)),
\]
combining the two limit equations, we establish a weak solution in the sense of \eqref{equasistaticSol}.

\subsection{An intermediate energy estimate} 
\label{ssec:energ}
To derive the proper energy estimate is precisely the point where we need the extra regularizer related to $\varepsilon_0$.
\begin{lemma}[Time delayed energy inequality] \label{timeDelayedEnergyInequality}
 Let $(\eta,u)$ be a solution to the time delayed equation \eqref{equasistaticSol}. Then this solution fulfills the following energy inequality for almost all $t_0 \in [0,h]$:
 \begin{align*}
  &\phantom{{}={}}\mathcal{E}_{K}(\eta(t_0))
  +\frac{\sqrt{\delta_0}}{2}\int_0^\ell\abs{\partial_x^3\eta(t_0)}^2\, dx 
   + \varepsilon_0 \int_0^{t_0} \int_0^{\ell} \abs{\partial_t \partial_{x}^3 \eta}^2 dx + \frac{1}{h}\int_{0}^{t_0}\int_0^{\ell} \rho_s \frac{\abs{\partial_t \eta (t)}^2}{2} dx dt\\
  &+\int_0^{t_0}  \int_{\Omega} \mu(t) \abs{\nablasym (u)}^2dxdt + \delta_0 \int_0^{t_0} \int_\Omega \abs{\nabla \Delta u}^2 dx + \frac{1}{h}\int_{0}^{t_0} \int_{\Omega}\rho(t) \frac{\abs{u(t)}^2}{2}dx dt \\
  &\leq \mathcal{E}_K(\eta(0)) + \frac{1}{h} \int_0^{t_0} \int_0^\ell \rho_s \frac{\abs{w \circ \eta_0}^2}{2} dx dt +  \frac{1}{h}\int_{0}^{t_0}\int_{\Omega}\rho_0 \frac{\abs{w}^2}{2}dx  + \int_0^{t_0}\int_{\Omega} \rho f \cdot u \, dx 
 \end{align*}
\end{lemma}

\begin{proof}
Consider the weak equation \eqref{equasistaticSol}. Due to the regularizing terms, we can test this with the admissible pair $(\partial_t \eta,u)$ on the interval $[0,t_0]$. This yields
\begin{align*}
  &\phantom{{}={}}\mathcal{E}_{K}(\eta(t_0))+\frac{\sqrt{\delta_0}}{2}\int_0^\ell\abs{\partial_x^3\eta(t_0)}^2\, dx  + \varepsilon_0 \int_0^{t_0} \int_0^{\ell} \abs{\partial_t \partial_{x}^3 \eta}^2 dx + \frac{1}{h}\int_{0}^{t_0} \int_0^{\ell} \rho_s (\partial_t \eta -w\circ \eta_0)\cdot \partial_t \eta \, dx dt \\
  &+\int_0^{t_0}  \int_{\Omega} \mu(t) \abs{\nablasym (u)}^2dxdt + \delta_0 \int_0^{t_0} \int_\Omega \abs{\nabla \Delta u}^2 dx + \frac{1}{h}\int_{0}^{t_0} \int_{\Omega}\rho_0 (u\circ \Phi - w) \cdot u \circ \Phi \,dx dt \\
  &= \mathcal{E}_K(\eta(0))  + \int_0^{t_0}\int_{\Omega} \rho f \cdot u \, dx 
\end{align*}
Now using Young's inequality on $w\circ \eta_0 \cdot \partial_t \eta$ we have for almost every $t\in [0,t_0]$
\begin{align*}
 \int_0^{\ell} \rho_s (\partial_t \eta -w\circ \eta_0)\cdot \partial_t \eta \, dx &\geq \int_0^{\ell} \rho_s \abs{\partial_t \eta}^2  - \dfrac{\rho_s}{2} \abs{\partial_t \eta}^2 - \dfrac{\rho_s}{2} \abs{w \circ \eta_0}^2 \, dx \\
 &= \int_0^\ell \frac{\rho_s}{2} \abs{\partial_t \eta}^2 dx -  \int_0^\ell \frac{\rho_s}{2} \abs{w \circ \eta_0}^2 dx
\end{align*}
and similarly for the fluid
\begin{align*}
\int_{\Omega}\rho_0 (u\circ \Phi - w) \cdot u \circ \Phi \,dx \geq  \int_{\Omega}\frac{\rho_0}{2} \abs{u\circ \Phi}^2 - \frac{\rho_0}{2} \abs{w}^2\,dx dt = \int_{\Omega}\frac{\rho}{2} \abs{u}^2 \,dx - \int_{\Omega} \frac{\rho_0}{2} \abs{w}^2\,dx
\end{align*}
where we used that $\Phi$ is a density preserving diffeomorphism. Combining these finishes the proof.
\end{proof}

\begin{remark}
 Note that this energy inequality is a strict improvement over what could be obtained going to the limit in the estimates of Lemma \ref{apriorih} as all the dissipative terms carry another factor of 2.
\end{remark}

\subsection{Passing with \texorpdfstring{$\varepsilon_0\to 0$}{epsilon to zero}.}
\label{ssec:eps}

We can now also treat the case $\varepsilon_0=0$. Note that even then we still have some additional regularity of $\partial_t \eta$ by the fact that it is equal to $u \circ \eta$ which itself is the trace of $u$ under a diffeomorphism. 

 We can use Theorem \ref{maintheoremdelay} to obtain a sequence of solutions $\eta^{(\varepsilon_k)}$, $u^{(\varepsilon_k)}$ to the weak equation for $\varepsilon_0 := \varepsilon_k \to 0$. All these solutions fulfill the energy inequality and the other respective estimates, so they are uniformly bounded in the same spaces as before, apart from the reduced regularity of $\partial_t \eta$. We can then pick a converging subsequence and go to the limit. By lower-semicontinuity this limit still fulfills the energy inequality (the $\varepsilon_0$ term has positive sign and can just be dropped). We also have convergence in all the terms of the equation by the same arguments as before, except for the one involving $\varepsilon_0$, where
 \begin{align*}
\abs{  \varepsilon_0 \int_{0}^{h} \int_{0}^{\ell} \partial_{x}^3 \partial_{t} \widetilde{\eta}^{(\tau)} \cdot \partial_{x}^3 \xi \, dx \, dt } \leq \varepsilon_0 \norm[L^2((0,\ell))]{ \partial_t \partial_{x}^3\eta }  \norm[L^2((0,\ell))]{\partial_{x}^3\xi} \to 0
 \end{align*}
 since $\sqrt{\varepsilon_0} \norm[L^2((0,\ell))]{ \partial_t \partial_{xx}\eta }$ is uniformly bounded by the energy inequality.

\section{Proof of Theorem \ref{maintheorem}}

\subsection{Construction of approximative solutions}

We now proceed to string together short, time-delayed solutions into a time-delayed solution on a longer interval $[0,T]$. For this fix $h >0$. 

For any time-delayed solution $(\eta,u, \Phi)$ on $[0,h]$ produced by Theorem \ref{maintheoremdelay}, the energy inequality of Lemma \ref{timeDelayedEnergyInequality} implies that $\mathcal{E}_{K}(\eta(h)) < \infty$ and $(t,x) \mapsto u(t,\Phi(t,x)) \in L^2([0,h]\times \Omega;\R^2)$. But then the resulting quantities $\eta(h)$ and $(t,x) \mapsto u(t) \circ \Phi(t) \circ \Phi^{-1}(h)$ are valid initial and right hand side data for Theorem \ref{maintheoremdelay} and we can use them to construct a solution on another interval of length $h$, which we will shift to be $[h,2h]$. We iterate this argument until  %
we reach the chosen time $T$. We will see later that there is always a minimal time $T>0$ independent of $h$ that allows us to keep uniform distance from collision.

We will write the total solution as 
\[\eta^{(h)}: [0,T)\times (0,\ell) \to \Omega, u^{(h)}: [0,T)\times \Omega \to \R^2.\]
and use a similar notation $\Omega^{(h)}_\pm(t)$ for the two time-dependent parts of the fluid domain and $\rho^{(h)}(t,z)$, $\mu^{(h)}(t,z)$ to denote $\rho_\pm$ depending on whether $z \in \Omega^{(h)}_\pm(t)$.
Additionally in each step we can use the invertibility of the flow map to construct a family of general flow maps 
between time $t$ and $t+h$ and then concatenate them into a full flow map $\Phi^{(h)}(t)$ for $t \in [0,T]$ such that%
\begin{align*}
 \Phi^{(h)}(0,y) &= y, &&\text{ for all } %
 y \in \Omega\\
 \partial_t \Phi^{(h)}(t,y) &= u^{(h)}(t,\Phi^{(h)}(t)(y)) &&\text{ for all } t\in (0,T), y \in \Omega.
\end{align*}

Since the flow map is exclusively used to switch between coordinates in time intervals of size $h$ we further introduce
\[
\Phi_{s}^{(h)}(t):\Omega^{(h)}_\pm(t)\to  \Omega^{(h)}_\pm(t+s),\quad \Phi_{s}^{(h)}(t)(y)=\Phi(t+s)\circ(\Phi(t))^{-1}(y),
\] 
usually with $s\in [-h,h]$. By definition, we find that
\[
\partial_s\Phi_{s}^{(h)}(t,y)=u^{(h)}(t+s, \Phi_{s}^{(h)}(t,y))
\]

Further all maps $\Phi_s(t)$ are volume preserving maps which map the respective upper and lower fluid domains at time $t$ into their respective counterpart at time $t+s$, which gives them the structure of a semi-group. As a direct consequence $\Phi_s$ is also density preserving, i.e. 
\begin{align*}
 \rho^{(h)}(t+s) \circ \Phi_s^{(h)}(t) = \rho^{(h)}(t).
\end{align*}

As a result we get for any small enough $h>0$ a pair $(\eta^{(h)},u^{(h)})$ of weak solutions to
\begin{align} \label{eq:longTimedelayed}
 \int_0^T \int_{0}^{\ell}& \left( \mathcal{L}_{\Gamma}(\eta^{(h)}) + \rho_s \frac{\partial_t \eta^{(h)}(t) - \partial_t \eta^{(h)}(t-h)}{h} \right) \cdot \xi \, dx + \sqrt{\delta_0}\int_{0}^{\ell}  \partial_{x}^3  \eta^{(h)}\partial_x^3\xi_0\, dx \nonumber \\
&+  \int_{\Omega} \mu \nablasym (u^{(h)}) \cdot \nabla q \, dz + \int_\Omega \delta_0 \nabla \Delta u^{(h)} \cdot \nabla \Delta  q\, dz \\ \nonumber
& + \int_{\Omega} \left( \frac{(\rho^{(h)} u^{(h)})(t) - (\rho^{(h)} u^{(h)})(t-h) \circ \Phi_{-h}^{(h)}(t) }{h} \cdot q  \right) dz  -  \int_{\Omega} \left(f \cdot q \right) \, dz\, dt =0, 
\end{align}
for all pairs of test functions $(\xi,q) \in L^2(0,T;\mathcal{H} \cap H^3((0,\ell)))  \times L^2(0,T;H^3(\Omega))$ such that $\xi = q \circ \eta^{(h)}$ on $[0,T]\times [0,\ell]$ and $\nabla \cdot q =0$ in $\Omega$.

Now we can iterate and reformulate the energy inequality of Lemma \ref{timeDelayedEnergyInequality} for our approximate functions and add an estimate.

\begin{corollary}[Energy estimate] \label{longTimedelayedEnergyEstimate}
 Let $(\eta^{(h)},u^{(h)})$ as before. Then this solution fulfills the following energy inequality for almost all $t_0 \in [0,T]$:
 $$
 \begin{aligned}
  &\phantom{{}={}}\mathcal{E}_{K}(\eta^{(h)}(t_0))
  + \frac{\sqrt{\delta_0}}{2}\int_0^\ell\abs{\partial_x^3\eta(t_0)}^2\, dx 
   + \fint_{t_0-h}^{t_0}\int_0^{\ell} \rho_s \frac{\abs{\partial_t \eta^{(h)} (t)}^2}{2} dx dt\\
  &+\int_0^{t_0}  \int_{\Omega} \mu^{(h)}(t) \abs{\nablasym(u^{(h)})}^2dxdt + \delta_0 \int_0^{t_0} \int_\Omega \abs{\nabla \Delta u^{(h)}}^2 dx + \fint_{t_0-h}^{t_0} \int_{\Omega}\rho^{(h)}(t) \frac{\abs{u^{(h)}(t)}^2}{2}dx dt \\
  &\leq \mathcal{E}_K(\eta_0) +  \int_0^\ell \rho_s \frac{\abs{v_0 \circ \eta_0}^2}{2} dx +  \int_{\Omega}\rho_0 \frac{\abs{v_0}^2}{2}dx  + \int_0^{t_0}\int_{\Omega} \rho f \cdot u^{(h)} \, dx dt
 \end{aligned}
 $$
 In particular we gain uniform in $h$ bounds in the order $C(1+T^2)$ on
 \begin{align*}
  &\sup_{t\in [0,T]} \mathcal{E}_{K}(\eta^{(h)}(t)),\quad \norm[(H^2+\delta_0^\frac{1}{4}H^3)((0,\ell))]{\eta^{(h)}(t)},
  \quad %
\int_0^T\norm[(H^1+\sqrt{\delta_0}H^3)(\Omega)]{u^{(h)}(t)}^2\, dt  
  \\
  &\sup_{t_0\in [0,T]} \fint_{t_0-h}^{t_0}\int_0^{\ell} \rho_s \frac{\abs{\partial_t \eta^{(h)} (t)}^2}{2} dx dt \ \ \text{ and } \ \  \sup_{t_0\in [0,T]} \fint_{t_0-h}^{t_0} \int_{\Omega}\rho^{(h)}(t) \frac{\abs{u^{(h)}(t)}^2}{2}dx dt.
 \end{align*}
\end{corollary}

\begin{proof}
 The energy inequality is simply a telescoped version of Lemma \ref{timeDelayedEnergyInequality}. From this the uniform bounds follow by noting that the only $h$-dependent term on the right hand side is the force term. For this we have using Young's inequality,
 \begin{align*}
  \int_0^{t_0}\int_{\Omega} \rho f \cdot u^{(h)} \, dx \leq \frac{\delta}{2} \int_0^{t_0}\int_{\Omega} \rho \abs{u^{(h)}}^2 \, dx + \frac{1}{2\delta}   \int_0^{t_0}\int_{\Omega} \max(\rho_\pm) \abs{f}^2 \, dx ,
 \end{align*}
 where for small enough $\delta$ the first term can be absorbed on the left hand side and the second only contains initial data. The estimates of the norms of $\eta^{(h)}$ and $u^{(h)}$ follow by the energy estimate and Korn's inequality.
\end{proof}

\begin{lemma}[Minimal time to collision]
\label{lem:nocol} 
 For given initial data, there is a time $T_0 >0$ such that for all $h$ small enough
 \begin{align*}
  \sup_{(t,x) \in [0,T_0] \times [0,\ell]} \abs{\eta_2^{(h)}(t,x)} \leq \frac{\ell}{2}-\frac{1}{2}\left(\frac{\ell}{2}-\sup_{x\in [0,\ell]} \abs{(\eta_0)_2(x)}\right                                                                                                                                                   ).
 \end{align*}
 In particular there will be no collision of the solid with $\partial \Omega$ even after sending $h\to 0$.
\end{lemma}

\begin{proof}
 Using the uniform bounds and Jensens inequality we know that 
 \begin{align*} 
\norm[H^1((0,\ell))]{\eta^{(h)}-\eta_0}^2 \leq C  T \int_0^{T} \norm[(0,\ell)]{\partial_t \eta^{(h)}}^2 dt \leq T C(1+T^2)
 \end{align*}
 Since additionally $\eta^{(h)}(t)$ is uniformly bounded in $C^1$, this implies that $\norm[\infty]{\eta^{(h)}(t)-\eta_0} \leq C(T) \to 0$ for $T \to 0$, which proves the lemma.
\end{proof}
The last lemma implies that we can assume that $T = T_0$ is fixed for our initial data, the solution can then afterwards be extended up to a potential point of collision.

\subsection{Aubin-Lions type argument}
\label{sec:aubin}

In this section we will pass to the limit with $h \to 0$. We wish to remove the $\delta_0$ regularizers simultaneously. For that we choose
\begin{align}
\label{eq:delta0}
\delta_0:=h^{\alpha_0}\text{ with }\alpha_0\in (0,1)\text{ fixed.}
\end{align}
All constants from now on will be independent of $\delta_0$.

The main effort is the strong convergence of the fluid velocity. This goes via a so called Aubin-Lions type argument derived in this sub-section. In the next sub-section we then pass to the limit with the weak equation.

Using the energy estimate and the usual compactness results we can take a subsequence and find limits $\eta \in C_w([0,T];H^2((0,\ell);\R^2)) \cap H^1([0,T];H^1((0,\ell);\R^2))$ and $u \in L^2([0,T];W_{div}^{1,2}(\Omega;\R^2))$ such that
\begin{align*}
 \eta^{(h)} &\rightharpoonup \eta &&\text{ in } C_w([0,T];H^2((0,\ell);\R^2)) \\
 \partial_t \eta^{(h)} &\rightharpoonup \partial_t \eta &&\text{ in } L^2([0,T];H^1((0,\ell);\R^2)) \\
 u^{(h)} &\rightharpoonup u &&\text{ in } L^2([0,T];W_{div}^{1,2}(\Omega;\R^2)) \\
\end{align*}

As on the $\tau$-level, we find by interpolation (see for instance the similar proof \cite[Proposition 2.20]{BenKamSch20}) strong convergences of a subsequence of $\eta^{(h)}$ as $h\to 0$: 
\begin{equation}
\label{uniformetah}
\begin{aligned} 
\eta^{(h)} & \to \eta & \text{ in } \ & C([0,T]\times [0,\ell];\mathbb{R}^{2}).
\end{aligned}
\end{equation}

What we do need is the convergence in weak form of 
\[
\frac{(\rho^{(h)}  u^{(h)} ) (t)-(\rho^{(h)}  u^{(h)}) (t-h,\Phi_{-h}^{(h)}(t))}{h}.
\]
First observe that due to \eqref{uniformetah} we find that $\rho^{(h)} $ converges strongly in any Lebesgue space.
Hence certainly it is sufficient to prove strong convergence of $u^{(h)}$. Due to the coupling of the velocities, classic Aubin-Lions arguments are not available.
We follow the convergence strategy of \cite{MuhSch20}.
 As it was the case there, also here it is now essential to consider the convergence of $u^{(h)}$ over the full domain including its ``trace'' $\partial_t\eta^{(h)}$ that satisfies its own PDE. For that we rely on the following generalisation of the Aubin-Lions lemma that was derived in~\cite[Theorem 5.1 and Remark 5.2]{MuhSch20}.
\begin{theorem}
\label{thm:auba}
Let $X,Z$ be two Banach spaces, such that $X^*\subset Z^*$.
Assume that $f_n:(0,T)\to X$ and $g_n: (0,T)\to X^*$, such that $g_n\in L^\infty(0,T;Z^*)$ uniformly. Moreover assume the following: 
\begin{enumerate}
\item The {\em weak convergence}: for some $s\in [1,\infty]$ we have that $f_n\weaktostar f$ in $L^s(X)$ and $g_n\weaktostar g$ in $L^{s'}(X^*)$.
\item The {\em approximability-condition} is satisfied: For every $\kappa\in (0,1]$ there exists a $f_{n,\kappa}\in L^s(0,T;X)\cap L^1(0,T;Z)$,
such that for every $\epsilon\in (0,1)$ there exists a $\kappa_\epsilon\in(0,1)$ (depending only on $\epsilon$) such that %
\[
\norm[{L^s(0,T;X)}]{f_n-f_{n,\kappa}}\leq \epsilon\text{ for all } \kappa\in (0,\kappa_\epsilon]%
\]
and for every $\kappa\in (0,1]$ there is a $C(\kappa)$ such that
\[
\norm[{L^1(0,T;Z)}]{f_{n,\kappa}}\, dt\leq C(\kappa).
\]
Moreover, we assume that for every $\kappa$ there is a function $f_\kappa$, and a subsequence such that $f_{n,\kappa}\weaktostar f_\kappa$ in $L^s(0,T;X)$.

\item The {\em equi-continuity} of $g_n$. We require that there exists an $\alpha\in (0,1]$ a functions $A_n$ with $A_n\in L^1(0,T)$ uniformly, such that for every $\kappa>0$ that there exist a $C(\kappa)>0$ and an $n_\kappa\in \N$ such that for $\tau>0$ and a.e.\ $t\in [0,T-\tau]$
\[
\sup_{n\geq n_{\kappa}} \Big|\fint_{0}^\tau\inner[{X',X}]{g_n(t)-g_n(t+s)}{f_{n,\kappa}(t)}\, s\Big| \leq C(\kappa)\tau^\alpha(A_n(t)+1).
\]

\item The {\em compactness assumption} is satisfied: $X^*\hookrightarrow \hookrightarrow Z^*$. More precisely, every uniformly bounded sequence in $X^*$ has a strongly converging subsequence in $Z^*$. 
\end{enumerate}
Then there is a subsequence, such that
\[
\int_0^T\inner[{X'\times X}]{f_n}{g_n}\, dt\to \int_0^T\inner[{X' \times X}]{f}{g}\, dt.
\]
\end{theorem}

We begin with the following lemma, that allows us to compare between certain mean-values. It is closely connected to~\cite[Lemma 4.14]{BenKamSch20}.

\begin{lemma}
\label{lem:closeness}
Let $q:[0,T] \times \Omega \to \R^2$ integrable and define
 \[
 \tilde{q}^{(h)}(t,y):=\fint_{-h}^0 q(t+s,{\Phi}^{(h)}_{s}(t)(y))\, ds,
 \]
 then for all $p\in [1,2)$, $p_1\in [1,\infty]$ and all $h>0$, we find that if $\nabla q\in L^{p_1}(0,T;L^p(\Omega))$, then
\begin{align*}
 \norm[L^\infty(0,T;L^p(\Omega))]{\fint_{-h}^0q(t+s,y)\,ds-\tilde{q}(t)}
 \leq C h^\frac{p_1-1}{p_1}
 \norm[{L^{p_1}(0,T;L^\frac{2p}{2-p}(\Omega))}]{\nabla q},
 \end{align*}
 where $C$ depends on the energy estimates only.
 Further we find for $q\in L^\frac{2p}{p-2}(\Omega)$ that
 \begin{align}
\label{eq:malte1}
\norm[L^p(\Omega^\pm(t))]{q-q \circ \Phi_{-h}^{(h)}(t)}\leq C h\norm[L^{\frac{2p}{2-p}}(\Omega)]{\nabla q}\text{ and } \norm[L^2(\Omega^\pm(t))]{q-q \circ \Phi_{-h}^{(h)}(t)}\leq C h\Lip(q).
\end{align}
\end{lemma}
 \begin{proof}

We begin with the following calculation
  \begin{align*}
\begin{aligned}
\abs{q(t+s,{\Phi}^{(h)}_{s}(t)(y))-q(t+s,y)}
&=\abs{\int_0^s\partial_\theta q(t+s,{\Phi}^{(h)}_{\theta}(t)(y))d\theta}
\\
&\leq \int_0^s\abs{\nabla q(t+s,{\Phi}^{(h)}_{\theta}(t)(y))}\abs{u(t+\theta,{\Phi}^{(h)}_{\theta}(t)(y))} d\theta.
\end{aligned}
 \end{align*}
Using the fact that $\det(\nabla \Phi_\theta^{(h)})=1$ we find by change of variables and the energy bounds that, for $p\in [1,2)$,
  \begin{align*}
\begin{aligned}
\norm[L^p(\Omega)]{q(t+s,{\Phi}^{(h)}_{s}(t,y))-q(t+s,y)}
&\leq \int_0^s\norm[L^{\frac{2p}{2-p}}(\Omega)]{\nabla q(t+s)}\norm[\Omega]{u(t+\theta,{\Phi}^{(h)}_{\theta}(t))}\, d\theta
\\
& \leq \sqrt{s}{\sqrt{h}}C\norm[L^{\frac{2p}{2-p}}(\Omega)]{\nabla q(t+s)}.
\end{aligned}
 \end{align*}
 this implies \eqref{eq:malte1}, for $q$ independent of $t$ and $s=h$. Further we find
 \begin{align*}
 \norm[L^\infty(0,T;L^p(\Omega))]{\fint_{-h}^0q(t+s,y)\,ds-\tilde{q}(t)}\leq \int^0_{-h} \norm[L^{\frac{2p}{2-p}}(\Omega)]{\nabla q(t+s)}\, ds
 \leq Ch^\frac{p_1-1}{p_1}
\norm[{L^{p_1}(0,T;L^\frac{2p}{2-p}(\Omega))}]{\nabla q}.
 \end{align*}
The Lipschitz case follows doing the respective estimates for $p=2$.
\end{proof}

In the following we use averaged momenta as functions that do possess some kind of weak time-derivative. For this we define
\begin{align*}
m^{(h)}(t):=\fint_{-h}^0 \rho^{(h)}(t+s)u^{(h)}(t+s)\, ds\text{ and }(\partial_t\eta)^{(h)}(t):=\fint_{-h}^0 \partial_t\eta^{(h)}(t+s)\, ds.
\end{align*}
We first show that $m^{(h)}(t)$ possesses some weak time-derivative in the interior. This argument is valid for both $\Omega_\pm(t)$ separately.
\begin{lemma}
\label{lem:time1}
 There exist a constant $C> 0$ depending only on $T$ and the initial data, such that for $k>2$, we find that $\partial_t m^{(h)}\in L^2(0,T; H^{-k}_{div}(\Omega_\pm^{(h)}))$. Moreover,\footnote{As is apparent following the proof, it is possible to relax the conditions from $L^2(H^2)$ on the right hand side to $L^2(L^p)$ for any $p>2$.}%
 \begin{align*}
 \sup_{q \in C^\infty_{0,div}((0,T)\times \Omega_\pm^{(h)}(t))} \inner[\Omega_\pm^{(h)}(t)]{\partial_t  m^{(h)}(t) }{q(t)}
 \leq C(\sqrt{\delta_0} \no{q}_{L^2(0,T;H^3(\Omega_\pm^{(h)}))}+\no{q}_{L^2(0,T;H^2(\Omega_\pm^{(h)}))})
 \end{align*}
\end{lemma}

\noindent
\begin{proof}
Let $q \in C^\infty_0((0,T)\times \Omega)$ with $\nabla \cdot q = 0$.
Let us first split the integrand into two 
along the flow map.
$$
\begin{aligned}
\inner[\Omega_\pm^{(h)}(t)]{\partial_t  m^{(h)}(t) }{q(t)}&=\inner[\Omega_\pm^{(h)}(t)]{ \frac{(\rho^{(h)}  u^{(h)})(t)-(\rho^{(h)}  u^{(h)})(t-h)}{h}}{q(t)} \\
&=\inner[\Omega_\pm^{(h)}(t)]{ \frac{(\rho^{(h)}  u^{(h)})(t)-(\rho^{(h)}  u^{(h)})(t-h) \circ \Phi^{(h)}_{-h}(t)}{h}}{q(t)} \\
&\quad
- \inner[\Omega_\pm^{(h)}(t)]{ \frac{(\rho^{(h)}  u^{(h)})(t-h)-(\rho^{(h)}  u^{(h)})(t-h) \circ \Phi^{(h)}_{-h}(t)}{h}}{q(t)}=: J_1(t) - J_2(t)
\end{aligned}
$$
We may estimate $J_1(t)$ by the weak formulation \eqref{eq:longTimedelayed}
$$
\begin{aligned}
&\phantom{{}={}}\abs{J_1(t)}  
=
\left |  \int_{\Omega_\pm^{(h)}(t)} \mu \nablasym(u^{(h)}(t)) \cdot \nabla q(t) \, dz + \int_{\Omega_\pm^{(h)}} \delta_0 \nabla \Delta u^{(h)} \cdot \nabla \Delta  q(t)\, dz -  \int_{\Omega_\pm^{(h)}(t)} \left(f (t)\cdot q(t) \right) \, dz \right |
\\
&\leq
\norm[\Omega_\pm^{(h)}(t)]{\nablasym u^{(h)}(t)} \norm[\Omega_\pm^{(h)}(t)]{\nablasym q(t)}+ \delta_0 \norm[H^3(\Omega_\pm^{(h)})]{u^{(h)}(t)} \norm[H^3(\Omega_\pm^{(h)})]{q(t)} +\norm[\Omega_\pm^{(h)}(t)]{f(t)} \norm[\Omega_\pm^{(h)}(t)]{q(t)}
 \end{aligned}
 $$
This gives for all $q \in L^2({[0,T]};H^{3}_0(\Omega_\pm^{(h)}))$, which are solenoidal, that
$$
\begin{aligned}
\int_0^T \abs{J_1(t)} dt \leq c(\delta_0 \no{q}_{L^2(0,T;H^3(\Omega_\pm^{(h)}))}+\no{q}_{L^2(0,T;H^1(\Omega_\pm^{(h)}))})
\end{aligned}
$$

For $J_2(t)$ we first note that by the density preserving nature of $\Phi$ we can obtain by a change of variables that
$$
\begin{aligned}
\inner[\Omega_\pm^{(h)}(t)]{(\rho^{(h)}  u^{(h)})(t-h) \circ \Phi_{-h}^{(h)}(t)}{q(t)} = \inner[\Omega_\pm^{(h)}(t-h)]{(\rho^{(h)} u^{(h)})(t-h)}{q(t) \circ \Phi_{h}^{(h)}(t-h)}.
\end{aligned}
$$
At this point we use Lemma~\ref{lem:closeness}
and find (using that all quantities are extended constant in time for $t<0$)
$$
\begin{aligned}
\int_0^T \abs{J_2(t)} dt &= \int_0^T \left| \inner[\Omega_\pm^{(h)}(t-h)]{(\rho^{(h)}  u^{(h)})(t-h)}{\frac{q(t)-q(t) \circ \Phi_h^{(h)}(t-h)}{h}} \right| dt 
\\
&\leq C\int_0^T\norm[L^4(\Omega_\pm^{(h)})]{(\rho^{(h)} u^{(h)}(t))} \norm[L^\frac{4}{3}(\Omega)]{\frac{q(t)-q(t) \circ \Phi_h^{(h)}(t-h)}{h}}\, dt
\\
&\leq C\int_0^T\norm[L^4(\Omega_\pm^{(h)})]{(\rho^{(h)} u^{(h)}(t))} \norm[L^{4}(\Omega)]{\nabla q(t)}\, dt
\end{aligned}
$$
Since by Sobolev embedding $ L^2(0,T;H^1(\Omega_\pm^{(h)})) \subset L^2(0,T;L^s(\Omega_\pm^{(h)}))$ for all $s<\infty$, we find
\[
\int_0^T \abs{J_2(t)} dt \leq C \norm[L^2(0,T;H^2(\Omega_\pm^{(h)}))]{q},
 \]
which finishes the proof.
\end{proof}

The next ingredient is to again define the differential pressure resulting from a spoiled compatibility for solenoidality.
For that we first derive the weak formulation acting on extensions constructed in Proposition~\ref{prop:arbMeanExt}. 
This we do by introducing an average pressure on $\psi^\pm$ (as defined in Proposition~\ref{prop:arbMeanExt}) for any $h>0$ by taking a fixed $q_0 \in C_0^\infty(\Omega;\R^2)$ such that $\nabla \cdot q_0 = \psi^+-\psi^-$ and defining
\begin{align*}
{P}^{(h)}(t) &:= \int_0^T \int_{0}^{\ell} \left( \mathcal{L}_{\Gamma}(\eta^{(h)}) + \rho_s \frac{\partial_t \eta^{(h)}(t) - \partial_t \eta^{(h)}(t-h)}{h} \right) \cdot \phi_0 \, dx + 
\sqrt{\delta_0} \int_{0}^{\ell}\partial_{x}^3  \eta^{(h)} \cdot \partial_{x}^3 \phi_0 \, dx
\\\nonumber &+  \int_{\Omega} \mu \nablasym(u^{(h)}) \cdot \nabla q_0 \, dz + \int_\Omega \delta_0 \nabla \Delta u^{(h)} \cdot \nabla \Delta  q_0\, dz
  \\&+ \int_{\Omega} \left( \frac{(\rho^{(h)} u^{(h)})(t) - (\rho^{(h)} u^{(h)})(t-h) \circ \Phi_{-h}^{(h)}(t) }{h} \cdot q_0  \right) dz  
  -  \int_{\Omega} \left(f \cdot q_0 \right) \, dz dt 
\end{align*}
where $\phi_0 := q\circ \eta^{(h)}$.
Using the estimates available on $u^{(h)}$ and $\eta^{(h)}$, as well as Lemma \ref{lem:time1} we deduce that $P^{(h)}(t)$ is uniformly bounded in $L^2(0,T)$.

Further we find for $\xi\in L^2([0,T];(H^1_0 \cap H^3)((0,\ell)))$ that 
\begin{align} \label{eq:longTimedelayed-pressure}
& \int_0^T \int_{0}^{\ell} \left( \mathcal{L}_{\Gamma}(\eta^{(h)}) + \rho_s \frac{\partial_t \eta^{(h)}(t) - \partial_t \eta^{(h)}(t-h)}{h} \right) \cdot \xi \, dx + 
\sqrt{\delta_0} \int_{0}^{\ell}\partial_{x}^3  \eta^{(h)} \cdot \partial_{x}^3 \xi \, dx
 \\\nonumber 
&+  \int_{\Omega} \mu \nablasym(u^{(h)}) \cdot \nabla \ext{\eta^{(h)}}(\xi) \, dz + \int_{\Omega} \left( \frac{(\rho^{(h)} u^{(h)})(t)  - (\rho^{(h)} u^{(h)})(t-h) \circ \Phi_{-h}^{(h)}(t) }{h} \cdot \ext{\eta^{(h)}}(\xi)  \right) dz  \\ \nonumber &+ \int_\Omega \delta_0 \nabla \Delta u^{(h)} \cdot \nabla \Delta  \ext{\eta^{(h)}}(\xi)\, dz -  \int_{\Omega} \left(f \cdot \ext{\eta^{(h)}}(\xi)\right) \, dz dt =\int_0^T\lam{\eta^{(h)}}(\xi)P^{(h)}\, dt. 
\end{align}
As can be seen later in the compactness proof it turns out to be essential to decouple the boundary values from the interior motion of the fluid. This was already observed in~\cite{LenRuz14}. %

Next we provide the respective estimate for the interface:
\begin{lemma}
\label{lem:time2}
 There exist a constant $C> 0$ depending only on $T$ and the initial data, such that for $\xi\in L^2(0,T;H^k_0(0,\ell))$
 \begin{align*}
  \int_0^T \inner[(0,\ell)]{\partial_t(\eta)^{(h)}}{\xi}+\inner[\Omega]{\partial_t m^{(h)}}{\ext{\eta^{(h)(t)}}(\xi)} dt \leq C\left(\norm[{L^2(0,T;H^2_0(0,\ell))}]{\xi}+\delta^\frac{1}{4}\norm[L^2(0,T;H^3(0,\ell))]{\xi}\right).
 \end{align*}
\end{lemma}

\noindent
\begin{proof}
Since $(\xi, \ext{\eta^{(h)}}(\xi))$ is a valid pair of test-functions for \eqref{eq:longTimedelayed-pressure}, we find %
\begin{align*} 
&\int_0^T \inner[(0,\ell)]{\partial_t(\eta)^{(h)}}{\xi}+\inner[\Omega]{\partial_t m^{(h)}}{\ext{\eta^{(h)}}(\xi)}  \\&- \int_\Omega \frac{1}{h} \left( (\rho^{(h)} u^{(h)})(t-h) - (\rho^{(h)} u^{(h)})(t-h) \circ \Phi_{-h}(t) \right) \cdot \ext{\eta^{(h)}}(\xi) dx dt
\\
&=
\int_0^T\rho_s  \int_{0}^{\ell}\left(\tfrac{\partial_t \eta^{(h)}(t) - \partial_t \eta^{(h)}(t-h)}{h} \right) \cdot \xi \, dx +\int_{\Omega} \left( \tfrac{(\rho^{(h)} u^{(h)})(t) - (\rho^{(h)} u^{(h)})(t-h) \circ \Phi_{-h}(t) }{h} \cdot \ext{\eta^{(h)}}(\xi)  \right) dz \, dt
\\
&
= \int_0^T-\inner[(0,\ell)]{\mathcal{L}_{\Gamma}(\eta^{(h)})}{\xi}
-\sqrt{\delta_0}\int_{0}^{\ell} \partial_{x}^3  \eta^{(h)} \cdot \partial_{x}^3 \xi \, dx
 - \int_{\Omega} \mu \nablasym(u^{(h)}) \cdot \nabla \ext{\eta^{(h)}}(\xi) \, dz 
\\  
&\quad  +\delta_0\int_\Omega  \nabla \Delta u^{(h)} \cdot \nabla \Delta  \ext{\eta^{(h)}}(\xi)\, dz+  \int_{\Omega} \left(f \cdot \ext{\eta^{(h)}}(\xi)\right) \, dz dt +\lam{\eta^{(h)}}(\xi)P^{(h)}\, dt. 
\end{align*}
Now the properties of the extension and the estimates on $\lam{\eta^{(h)}}$ derived in Proposition~\ref{prop:arbMeanExt}, the assumptions on $\mathcal{L}_{\Gamma}$, the $L^2$ in time bound of $P^{(h)}$ and the estimates of Lemma \ref{lem:closeness} and \ref{lem:time1} imply the result.
\end{proof}

By the above lemmata we find that both quantities do possess some kind of weak time derivative. 
 We have now finished all preparations to provide the main result of this section.
\begin{proposition}[Convergence of the velocity]
We find that for a subsequence 
\[
\int_0^T\inner[(0,\ell)]{(\partial_t\eta)^{(h)}}{\partial_t\eta^{(h)}}\, dt + \int_0^T\inner[\Omega]{m^{(h)}}{u^{(h)}}\, dt
 \to \int_0^T\norm[(0,\ell)]{\partial_t\eta}^2+\int_0^T\norm[\Omega]{\sqrt{\rho}u}^2\, dt.
\]
Further, the same subsequence may be assumed to satisfy
$(\partial_t\eta)^{(h)}\to \partial_t \eta$, strongly in $L^2([0,T]\times [0,\ell])$ and $m^{(h)}\to \rho u$  
 strongly in $L^2([0,T]\times \Omega)$.
\end{proposition}
\begin{proof}
We follow the strategy developed in \cite[Section 5]{MuhSch20}, see also \cite{LenRuz14}. 

For the first convergence we split
\begin{align*}
&\int_0^T\inner[(0,\ell)]{(\partial_t\eta)^{(h)}}{\partial_t\eta^{(h)}}\, dt + \int_0^T\inner[\Omega]{m^{(h)}}{u^{(h)}}\, dt \\
&=\underbrace{\int_0^T \inner[(0,\ell)]{(\partial_t\eta)^{(h)}}{ \partial_t\eta^{(h)}}+\inner[\Omega]{m^{(h)}}{\ext{\eta^{(h)}}(\partial_t\eta^{(h)})}\, dt}_{=:\mathrm{I}^{(h)}} + \underbrace{\int_0^T\inner[\Omega]{m^{(h)}}{u^{(h)}-\ext{\eta^{(h)}}(\partial_t\eta^{(h)})}\,dt}_{ =:\mathrm{II}^{(h)}}.
\end{align*}

\subsubsection*{Convergence of $\mathrm{I}^{(h)}$}
For $I^{(h)}$ we choose $g_h=( (\partial_t\eta)^{(h)},m^{(h)})$ and $f_h=( \partial_t\eta^{(h)},\ext{\eta^{(h)}}(\partial_t\eta^{(h)}))$. The spaces are $X:=L^2(0,\ell)\times H^{-s}(\Omega)$ and consequently $X^*=L^2(0,\ell)\times H^{s}(\Omega)$. We also define $Z:=H^{s_0}(0,\ell)\times H^{s_0}(\Omega)$, where we choose $0<s<s_0<\frac14$, while for integrability in time we restrict ourselves to $L^2$. Next we introduce $(\partial_t\eta^h)_\kappa$ to be a standard smooth approximation of $\partial_t\eta^h$ induced by a linear mollification operator that satisfies respective estimates and boundary values. Accordingly we introduce
\[
f_{h,\kappa}:=((\partial_t\eta^h)_\kappa,\ext{\eta^{(h)}}(\partial_t\eta^h)_\kappa)).
\]
Hence let us check the assumptions of Theorem~\ref{thm:auba}. Note that by the trace theorem $\partial_t\eta^{(h)}$ is uniformly bounded in $L^2(0,T;H^\frac{1}{2}(0,\ell))$. Hence (1) and (2) in Theorem~\ref{thm:auba} follow by weak compactness, approximation results and the linearity of the mollifier as well as the extension. The property (4) is a direct consequence of the choices of the spaces made by the compactness properties of Sobolev spaces. Hence we focus on property (3) in Theorem~\ref{thm:auba}, which is (as common in settings involving variable geometry) the main issue; as it relates to continuity properties in time. For that we estimate for $0<t<\sigma<t+\sigma_0$ 
\begin{align*}
&\abs{\inner{g_h(t)-g_h(\sigma)}{f_{h,\kappa}(t)}}
=\abs{\inner{\int_t^\sigma \partial_t g_h(s)\, ds}{f_{h,\kappa}(t)}}
\\
&\quad \leq \left | \int_t^\sigma\inner[(0,\ell)]{\partial_t(\partial_t\eta)^{(h)}(s)}{ (\partial_t\eta^{(h)})_\kappa(t)}\, ds+\int_t^\sigma\inner[\Omega]{\partial_t m^{(h)}(s)}{\ext{\eta^{(h)}(s)}((\partial_t\eta^{(h)})_\kappa(t))}\, ds\right|
\\
&\quad + \left |\int_t^\sigma \inner[\Omega]{\partial_t m^{(h)}(s)}{\ext{\eta^{(h)}(s)}((\partial_t\eta^{(h)})_\kappa(t))-\ext{\eta^{(h)}(t)}((\partial_t\eta^{(h)})_\kappa(t))}\, ds\right|=\mathcal{A}_1+\mathcal{A}_2
\end{align*}
Crucially above was to change the geometry of the extension in order to use the weak coupled time-derivative. Indeed on $\mathcal{A}_1$ we may now use Lemma~\ref{lem:time2} and the properties of the mollifier as well as the properties of the extension operator defined in Proposition \ref{prop:arbMeanExt}. This is the key point where the explicit bound on the $H^3$ norm of the extension operator is relevant.
Indeed, it implies that
\begin{align*}
\mathcal{A}_1\leq C\bigg(\int_t^\sigma \norm[\mathrlap{H^3(0,\ell)}]{(\partial_t\eta^{(h)})_\kappa(t)}^2\quad ds\bigg)^{\mathrlap{\frac12}}
\leq C(\kappa)\bigg(\int_t^\sigma \norm[\mathrlap{L^2(0,\ell)}]{\partial_t\eta^{(h)}(t)}^2\,\quad ds\bigg)^{\mathrlap{\frac12}}
= C(\kappa)\abs{t-\sigma}^\frac12\norm[L^2(0,\ell)]{\partial_t\eta^{(h)}(t)}.
\end{align*}
For that estimate of $\mathcal{A}_2$ we find using Proposition~\ref{prop:arbMeanExt} and the property of the mollifier to get
\begin{align*}
\norm[L^2(\Omega)]{\int_t^\sigma\partial_\alpha\ext{\eta^{(h)}(\alpha)}((\partial_t\eta^{(h)})_\kappa(t))\, d\alpha}
\leq C\abs{\sigma-\tau}\norm[H^1(0,\ell)]{(\partial_t\eta^{(h)})_\kappa(t)}\leq C(\kappa)\abs{\sigma-\tau}.
\end{align*}
Now by partial integration, we find using the uniform $L^\infty$ bounds in time on $m^{(h)}$ to get
\begin{align*}
\mathcal{A}_2&= \left |\int_t^\sigma \inner[\Omega]{\partial_t m^{(h)}(s)}{\int_t^s\partial_\alpha\ext{\eta^{(h)}(\alpha)}((\partial_t\eta^{(h)})_\kappa(t))\, d\alpha}\, ds\right|
\\
&\leq \left |\inner[\Omega]{m^{(h)}(\sigma)}{\int_t^\sigma \partial_\alpha\ext{\eta^{(h)}(\alpha)}((\partial_t\eta^{(h)})_\kappa(t))\, d\alpha}\right|+\left |\int_t^\sigma \inner[\Omega]{m^{(h)}(s)}{ \partial_s\ext{\eta^{(h)}(s)}((\partial_t\eta^{(h)})_\kappa(t))\, ds}\right|
\\
&\leq  C(\kappa)\abs{\sigma-\tau}.
\end{align*}
This implies the convergence of $\mathrm{I}^{(h)}$.

\subsubsection*{Convergence of $\mathrm{II}^{(h)}$}
For $\mathrm{II}^{(h)}$ we apply Theorem~\ref{thm:auba} once more. Here we may rely on the interior properties of $u^{(h)}$. Hence it makes sense to consider the convergence in both domains separately.
 We choose 
 \[g_\pm^{(h)}=m^{(h)}\chi_{\Omega_\pm^{(h)}} \text{ and }f_\pm^{(h)}=( u^{(h)}-\ext{\eta^{(h)}}(\partial_t\eta^{(h)}))\chi_{\Omega_\pm^{(h)}}.\]
 The relevant spaces now are $X=H^{-s}(\Omega)$,  $X^*=H^{s}(\Omega)$ as well as $Z=L^2(\Omega)$, where we may choose $0<s<\frac14$. For integrability in time we again choose $L^2$. As before (1) and  (4) in Theorem~\ref{thm:auba} follow by weak compactness and the compactness relations between Sobolev spaces.
The main difference with respect to the convergence of $I^{(h)}$ is the different mollification approach. Here we have a sequence of functions that is zero on the interface between $\Omega^+$ and $\Omega^-$. Hence (in a weak enough space) it is possible to make the function zero in a neighborhood in a solenoidal manner. Indeed as has been shown in \cite[Lemma A.13]{LenRuz14} (compare also to \cite[Lemma 6.3]{MuhSch20} that there is a mollifying sequence $f_\pm^{(h,\kappa)}$ that is smooth (with respective $\kappa$ bounds). Further it satisfies $\supp(f^{(h,\kappa)}_\pm)\subset \subset \Omega_\pm^{(h)}$ for all $h<h_\kappa$ and
\begin{align*}
\norm[H^{-s}(\Omega)]{f^{(h)}_\pm-f^{(h,\kappa)}_\pm}\leq c\epsilon \norm[L^2(\Omega)]{f_\pm^{(h)}},
\end{align*}
for all $\kappa<\sigma_\epsilon$ and $h<h_\kappa$ fixed. Hence (2) in Theorem~\ref{thm:auba} is satisfied. With this preparations the equicontinuity property (3) in Theorem~\ref{thm:auba} follows by the properties of the weak time-derivative in the interior Lemma~\ref{lem:time1} in each subdomain $\Omega_\pm^{(h)}$ separately.

\subsubsection*{Strong covergence of the time-averaged momenta}
Due to the strict convexity of $L^2$, it suffices to show the convergence of the norm. Further, since $\rho^{(h)}\to \rho$ strongly in any $L^p$-space and is uniformly bounded from above and below, the convergence of 
\begin{align*}
\int_0^T \norm[L^2(0,\ell)]{(\partial_t\eta)^{(h)}}^2 +\norm[L^2(\Omega)]{\frac{m^{(h)}}{\sqrt{\rho^{(h)}}}}^2\,dt\to \int_0^T \norm[(0,\ell)]{\partial_t\eta}^2 +\norm[L^2(\Omega)]\rho {u}^2\,dt
\end{align*}
 implies the wanted convergence. We split the above into
\begin{align*}
\int_0^T \norm[L^2(0,\ell)]{(\partial_t\eta)^{(h)}}^2 +\norm[L^2(\Omega)]{\frac{m^{(h)}}{\sqrt{\rho^{(h)}}}}^2\,dt
&=
\int_0^T \inner[(0,\ell)]{(\partial_t\eta)^{(h)}}{ (\partial_t\eta)^{(h)}}+\inner[\Omega]{m^{(h)}}{\ext{\eta^{(h)}}((\partial_t\eta)^{(h)})}\, dt
\\
 & \hspace{-1cm} + \int_0^T\inner[\Omega]{m^{(h)}}{\fint_{t}^{t+h} u^{(h)}(s)\, ds-\ext{\eta^{(h)}}((\partial_t\eta)^{(h)})}\,dt=: \mathrm{III}^{(h)} + \mathrm{IV}^{(h)}.
\end{align*}
On both we will apply Theorem~\ref{thm:auba} again. The term $\mathrm{III}^{(h)}$ converges by the very same arguments as the convergence of $\mathrm{I}^{(h)}$ above.  

\subsubsection*{Convergence of $\mathrm{IV}^{(h)}$}
As for $\mathrm{IV}^{(h)}$, we apply Theorem~\ref{thm:auba} once more. 
 We choose
 \[g_h^\pm=m^{(h)}\chi_{\Omega_\pm^{(h)}} \text{ and }f_h^\pm=\fint_{t}^{t+h} u^{(h)}(s)\, ds-\ext{\eta^{(h)}}((\partial_t\eta)^{(h)})\chi_{\Omega_\pm^{(h)}}.\]
 The spaces now are $X=H^{-s}(\Omega)$, $X^*=H^{s}(\Omega)$ and $Z=L^2(\Omega)$, where as always choose $0<s<\frac14$.  Hence as before (1) and  (4) in Theorem~\ref{thm:auba} follow by weak compactness and the compactness relations between Sobolev spaces.
The main difference with respect to the convergence of $\mathrm{II}^{(h)}$ is again  the modification of the mollification approach. Indeed the right hand side term in $\mathrm{IV}^{(h)}$ is not zero on the curve $\eta^{(h)}$; for $(t,x)\in [0,T]\times [0,\ell]$ we have
 \[
 (\partial_t\eta)^{(h)}(t,x)=\fint_{-h}^0 u(t+s,\eta^{(h)}(s,x))\, ds =\fint_{-h}^0  u(t+s)\circ {\Phi}^{(h)}_{s}(t)(\eta^{(h)}(t,x))\, ds \, .
 \]
 In particular, we wish to use $\ext{\eta^{(h)}(t)}((\partial_t\eta)^{(h)})$ and hence we need to substract a mean value of $u$ that does cancel the boundary values. Due to the above identities the right choice is
 \[
 \tilde{u}^{(h)}(t,y):=\fint_{-h}^0 u(t+s,{\Phi}^{(h)}_{s}(t)(y))\, ds;
 \]
however, this function does satisfy the right boundary conditions. But its rather weak differentiability does not matter by the choices of $Z$ and $X$. Hence by defining $\tilde{f}_h^\pm:=(\tilde{u}^{(h)}-\ext{\eta^{(h)}}((\partial_t\eta)^{(h)})\chi_{\Omega_\pm^{(h)}}$ we find analogous to $\mathrm{II}^{(h)}$ a mollification of $\tilde{f}_h^\pm$ denoted by $f_{h,\kappa}^\pm$ that is smooth (with respective $\kappa$ bounds). Further we can require that it satisfies $\supp(f_{h,\kappa}^\pm)\subset \subset \Omega_\pm^{(h)}$ for all $h<h_\kappa$ and
\begin{align*}
\norm[H^{-s}(\Omega)]{\tilde{f}_h^\pm-f_{h,\kappa}^\pm}\leq c\epsilon \norm[L^2(\Omega)]{\tilde{f}_h},
\end{align*}
for all $\kappa<\sigma_\epsilon$ and $h<h_\kappa$ fixed. But now using Lemma~\ref{lem:closeness} we find using also the fact that $L^p(\Omega)$ continuously embeds into $H^{-s}(\Omega)$ for $-s-1\leq -\frac{2}{p}$ that
\[
\norm[H^{-s}(\Omega)]{\tilde{f}_h^\pm-f^\pm_h}\leq c\norm[L^p(\Omega)]{\tilde{f}_h^\pm-f^\pm_h}\leq 
C \sqrt{h}
 \norm[{L^{2}(0,T;L^\frac{2p}{2-p}(\Omega))}]{\nabla u}.
\]
But that right hand side can be quantified uniformly by some power of $h$ using that we choose $\delta_0$ not to small~\eqref{eq:delta0}. Indeed, 
\[
\norm[L^2(0,T;C^{0,1}(\Omega))]{\nabla u^{(h)}}^2\leq \frac{C}{\delta_0}=\frac{C}{h^\alpha_0},
\]
with $\alpha_0<1$.
Hence 
\begin{align*}
\norm[H^{-s}(\Omega)]{{f}_h^\pm-f_{h,\kappa}^\pm}\leq C\epsilon ,
\end{align*}
and (2) in Theorem~\ref{thm:auba} is satisfied. Further now $f_{h,\kappa}^\pm$ is a valid smooth and compactly supported test-function on which Lemma~\ref{lem:time1} can be applied and hence all conditions to apply Theorem~\ref{thm:auba} are valid.
\end{proof}

\subsection{Convergence of the equation}

 We now proceed in a similar fashion to the proof of Theorem \ref{maintheoremdelay} and derive the limit equation by considering two different kinds of test-functions.
 
 First we start with an Eulerian test function $q$ without support on the solid, i.e.\ $q\in C_c^\infty (I\times\Omega;\R^2)$ such that $\nabla\cdot q = 0$ and $\supp q(t) \cap \eta(t,[0,\ell]) = \emptyset$. Then for all $h$ small enough, by the uniform convergence of the solid, the support of $q$ also does not intersect $\eta^{(h)}(t,[0,\ell])$, so the pair $(0,q)$ is also an admissible test-function for the time-delayed equation \eqref{eq:longTimedelayed}.
 
 We then take all the terms in this equation and consider the limit $h \to 0$. All the terms related to the solid drop out, the force term is independent of $h$ and for the fluid viscosity and the $\delta_0$ regularizing term
 \begin{align*}
  \int_0^T \int_\Omega \delta_0 \nabla \Delta u^{(h)}\nabla \Delta q +  \mu^{(h)} \nablasym(u^{(h)}) \cdot \nabla q \, dz dt \to  \int_0^T \int_\Omega \mu \nablasym (u) \cdot \nabla q \, dz dt
 \end{align*}
 by the weak $H^1$-convergence of $u^{(h)}$, the observation that $\mu^{(h)} = \mu$ on $\supp q$ for $h$ small enough and by the choice of $\delta_0$ in \eqref{eq:delta0}.
 
 What is left is the inertial term. For this we calculate by shifting half of the term in time
 \begin{align*}
  &\phantom{{}={}}  \int_0^T \int_{\Omega} \left( \frac{(\rho^{(h)} u^{(h)})(t) - (\rho^{(h)} u^{(h)})(t-h) \circ \Phi^{(h)}_{-{h}}(t) }{h} \cdot q  \right) dz dt \\
  &= -\int_0^T \int_{\Omega} \left( (\rho^{(h)} u^{(h)})(t) \cdot \frac{q(t+h) \circ \Phi^{(h)}_{h} - q(t)}{h}  \right)  dz dt\\
  &= -\int_0^T \int_{\Omega} \left( (\rho^{(h)} u^{(h)})(t) \cdot \fint_0^{h} \partial_s \left(q(t+s) \circ \Phi^{(h)}_{s}\right) ds  \right)  dz dt\\
  &= -\int_0^T \int_{\Omega} \left( (\rho^{(h)} u^{(h)})(t) \cdot \fint_0^{h} \partial_t q(t+s) \circ \Phi^{(h)}_{s} + (u^{(h)}\cdot \nabla q)(t+s) \circ \Phi^{(h)}_s ds  \right)  dz dt.
 \end{align*}
 Splitting this into two we obtain
 \begin{align*}
  -\int_0^T \int_{\Omega} \left( (\rho^{(h)} u^{(h)})(t) \cdot \fint_0^{h} \partial_t q(t+s) \circ \Phi^{(h)}_{s} ds  \right)  dz dt \to -\int_0^T \int_{\Omega}  \rho u \cdot \partial_t q  \,dz dt
 \end{align*}
 as the limit of the first term and
 \begin{align*}
  &\phantom{{}={}} -\int_0^T \int_{\Omega} \left( (\rho^{(h)} u^{(h)})(t) \cdot \fint_0^{h}  (u^{(h)}\cdot \nabla q)(t+s) \circ \Phi^{(h)}_s ds  \right)  dz dt \\
  &= -\int_0^T \int_{\Omega} \left( \fint_{-h}^0 (\rho^{(h)} u^{(h)})(t+s) \circ \Phi^{(h)}_s(t) ds \cdot   (u^{(h)}\cdot \nabla q)(t)   \right)  dz dt \\
  &\to -\int_0^T \int_{\Omega} \left( (\rho u)(t)  \cdot   (u \cdot \nabla q)(t)   \right)  dz dt
 \end{align*}
 for the second. Here we use that $\fint_{-h}^0 (\rho^{(h)} u^{(h)})(t+s) ds \to \rho u$ strongly and that the error given by difference $(\rho^{(h)} u^{(h)})(t+s) \circ \Phi_s^{(h)}(t) - (\rho^{(h)} u^{(h)})(t+s)$ vanishes as $h \to 0$.
 
 In total we get by a density argument that
 \begin{align} \label{weakSolNoSolid}
  \int_0^T \int_\Omega \mu \nablasym (u) \cdot \nabla q \, dz - \int_\Omega \rho u \cdot \left(\partial_t q +  (u\cdot \nabla q) \right) dz - \int_\Omega (f\cdot q) dz dt = 0
 \end{align}
 for all $q\in H^1_0(I;H^1_0(\Omega;\R^2))$ with $\nabla \cdot q = 0$ in $\Omega$ and $q(t,\eta(t,x)) =0$ for all $t\in[0,T]$,$x\in [0,\ell]$.

Next consider a fixed $\xi \in C_0^\infty(I\times [0,\ell];\R^2)$ such that $\int_0^\ell \xi \wedge \partial_x \eta dx = 0$. Using Proposition \ref{prop:arbMeanExt} we can extend this to a $\ext{\eta^{(h)}}(\xi)$ on $\Omega$ with $\ext{\eta^{(h)}}(\xi)(t) \circ \eta^{(h)}(t) = \xi$ for all $t$ and $\nabla \cdot \ext{\eta^{(h)}}(\xi)(t) = \lambda^{(h)}(t) \psi^+ - \lambda(t) \psi^-$ with the fixed two bump functions supported in $\Omega_\pm^{(h)}$ and $\lambda^{(h)}(t) = \int_0^\ell \xi \wedge \partial_x\eta \, dx$. The latter in particular implies that $\lambda^{(h)} \to 0$ uniformly.

Now per construction the pair $(\xi,\ext{\eta^{(h)}}(\xi))$ is admissible for \eqref{eq:longTimedelayed-pressure}. We can thus consider the limit $h \to 0$ of that equation. On the right hand side, in the limit $\int_0^T\lambda^{(h)}(t)P^{(h)}(t) dt \to 0$, so we only need to deal with the other terms.

Here we have
\begin{align*}
  \int_0^T \int_{0}^{\ell} \mathcal{L}_{\Gamma}(\eta^{(h)})  \cdot \xi \, dx dt &\to \int_0^T \int_{0}^{\ell} \mathcal{L}_{\Gamma}(\eta)  \cdot \xi \, dxdt\\
  \sqrt{\delta_{0}} \int_0^T \int_{0}^{\ell} \partial_{x}^3  \eta^{(h)} \cdot \partial_{x}^3 \xi \, dxdt& \to  0 \\
  \int_0^T\int_{\Omega} \mu \nablasym(u^{(h)}) \cdot \nabla (\ext{\eta^{(h)}}(\xi)) \, dzdt&\to  \int_0^T \int_{\Omega} \mu \nablasym (u) \cdot \nabla \ext{\eta}(\xi) \, dzdt \\
  \int_0^T\int_{\Omega} f \cdot (q^{(h)})  \, dz dt &\to   \int_0^T\int_{\Omega} f \cdot q  \, dz dt, 
\end{align*}
as simple pairs of weak times strong convergence for some subsequence, which follows from the uniform bounds on $\eta^{(h)}$ and $u^{(h)}$ as well as $\ext{\eta^{(h)}}$. Additionally note that $\ext{\eta^{(h)}}(\xi) \to \ext{\eta}(\xi)$, so the pair $(\xi,\ext{\eta}(\xi))$ also fulfills the required coupling conditions.

Further we find that by the uniform bounds on $\eta^{(h)},u^{(h)}$ and $\ext{\eta^{(h)}}$ as well as by \eqref{eq:delta0} that
\begin{align*}
 \abs{\int_0^T \int_\Omega \delta_0 \nabla \Delta u^{(h)}\nabla \Delta \ext{\eta^{(h)}}\, dx\, dt} \leq c\delta_0^\frac{1}{4}\sqrt{\delta}\norm[H^3(\Omega)]{u^{(h)}}\delta_0^\frac{1}{4}\norm[H^3(0,\ell)]{\eta^{(h)}}\leq C\delta_0^\frac{1}{4}\to 0.
\end{align*}

Additionally we have
\begin{align*}
 \int_0^T \int_{0}^{\ell} \rho_s \frac{\partial_t \eta^{(h)}(t) - \partial_t \eta^{(h)}(t-h)}{h}  \cdot \xi^{(h)} \, dx dt &= - \int_0^T \int_{0}^{\ell} \rho_s \partial_t \eta^{(h)}(t) \cdot \frac{\xi^{(h)}(t+h)-\xi^{(h)}(t)}{h} \, dx dt \\
 &\to -\int_0^T \int_{0}^{\ell} \rho_s  \partial_t \eta^{(h)}(t) \cdot \partial_t \xi \, dx dt.
\end{align*}

Finally we treat the kinetic term for the fluid like before. As any admissible pair $(\xi,q)$ can be decomposed into an extension of $\xi$ and a divergence free $q$ that is $0$ on the interface this then proves that $(\eta,u)$ is a weak solution.

Furthermore, considerning the energy inequality in Corollary \ref{longTimedelayedEnergyEstimate}, we know by Lemma \ref{thm:auba} that the two time averages converge to $\int_0^{\ell} \rho_s \frac{\abs{\partial_t \eta}^2}{2} dx$ and $\int_\Omega \rho \frac{u^2}{2} dz$ respectively and that all other terms on the left hand side are lower-semicontinuous with respect to $h \to 0$. As the force term on the right hand side is continuous, this then implies the energy inequality for our solutions.

Assuming that $\sup_{x\in[0,\ell]} \abs{\eta(x)} \leq \frac{\ell}{2}$, this then allows us to take the final data as new initial data and continue the solution from there. So either there exists a solution for all times $T> 0$, or there is a final time $T_{col}$ for which $\eta(T_{col},(0,\ell))\cap \partial\Omega \neq \emptyset$, i.e.\ we have a collision.

Finally we note that a pressure can be constructed as a distribution by 
\begin{align}
\label{eq:final-pres}
\begin{aligned}
 P(a) &:= \int_0^T \int_{0}^{\ell} \mathcal{L}_{\Gamma}(\eta)  \cdot \xi- \rho_s \partial_t \eta \cdot \partial_t \xi \, dx +   \int_{\Omega} \mu \nablasym (u) \cdot \nabla q \, dz \\ 
&  - \int_{\Omega} \rho u \cdot \left(  \partial_t q + u\cdot \nabla q \right) dz  -  \int_{\Omega} \left(f \cdot q \right) \, dz dt 
\end{aligned}
\end{align}
for any $a \in C_0^\infty(I\times \Omega)$ with $\int_\Omega a(t) dx = 0$ for all $t\in I$. Here $\xi$ and $q$ are constructed in the following fashion:

Pick two bump-functions $\psi_\pm$ as in Proposition \ref{prop:arbMeanExt}. Use this proposition to extend any constant function $\xi_0 \in C_0^\infty(I\times [0,\ell])$ such that $\lambda(t):= \int_0^\ell \xi_0(t) \wedge \partial_x \eta(t) dx \neq 0$ for all $t \in I$ to a $q_0(t,x)$ with $\nabla \cdot q_0(t,x) = \lambda(t) \psi_\pm$. 

For each $t\in I$, apply the Bogovskij-operators defined in Proposition \ref{theobog}, $\mathcal{B}_\pm$ on $\Omega_\pm(t)$ to $a(t)\chi_{\Omega_+(t)} - \psi^+ \int_{\Omega_+(t)} a(t)dz$ and $a(t)\chi_{\Omega_-(t)} - \psi^-\int_{\Omega_-(t)} a(t)\,dz$ respectively to get $q_\pm$. Then set
\begin{align*}
 \xi := - \frac{ \xi_0 }{\lambda(t)} \int_{\Omega_+(t)} a(t)\,dz \text{ and } q := q_+ + q_- - \frac{q_0}{\lambda(t)} \int_{\Omega_+(t)} a(t)dz.
\end{align*}
A short calculation reveals that $q(t) \circ \eta(t) = \xi(t)$ and $\nabla \cdot q(t) = a(t)$. Additionally by the estimates on $\eta$, $u$ and $\mathcal{B}_\pm$, $P$ is a bounded operator.%

Furthermore, since $(\eta,u)$ is a weak solution, $P(a)$ has the same value for any $(\xi,q)$ with $q(t) \circ \eta(t) = \xi(t)$ and $\nabla \cdot q(t) = a(t)$. As a result we finally have
\begin{align*}
 0 &= \int_0^T \int_{0}^{\ell} \mathcal{L}_{\Gamma}(\eta)  \cdot \xi- \rho_s \partial_t \eta \cdot \partial_t \xi \, dx +   \int_{\Omega} \mu \nablasym (u) \cdot \nabla q \, dz \\ \nonumber
&  - \int_{\Omega} \rho u \cdot \left(  \partial_t q + u\cdot \nabla q \right) dz  -  \int_{\Omega} \left(f \cdot q \right) \, dz dt - P(\nabla \cdot q). 
\end{align*}

\begin{remark}
\label{rem:pres1}
 Note that as usual, the resulting pressure can only be determined up to a constant, as we are working in a closed container. The same would be true if we instead would have used two individual $a_\pm \in C_0^\infty(I\times \Omega_\pm)$ with $\int_{\Omega_\pm(t)} a_\pm dx = 0$ to recover two individual pressures $p^\pm$ for the two subdomains. However this would have given us two constants and thus an additional degree of freedom.
 
 The solution to this apparent paradox is of course that the solid allows for an exchange of pressure between the two fluids. Given such $p_\pm$, we can relate both of their constants by considering any single pair of test-functions $(\xi,q)$ such that $\int_{\Omega_+} \nabla \cdot q = - \int_{\Omega_-} \nabla \cdot q \neq 0$. Testing the weak formulation with this, we then obtain an equation in which all terms are known, except for the difference in the two constants. They are thus determined.
 
 This is in contrast to the case of a single fluid, see Remark \ref{rem:pres2}.
\end{remark}

\section{Proof of Theorem~\ref{maintheorem2}}

The aim of this section is to show the weak existence of solutions to
\begin{align} \label{weak00}
& \phantom{{}+{}} \int_0^T \rho_{-} \left( \dfrac{d}{dt} \int_{\Omega_{-}^{\eta}(t)} u_{-} \cdot q_{-} \, dz - \int_{\Omega_{-}^{\eta}(t)} \left[ u_{-} \cdot \partial_{t} q_{-} + \left(u_{-} \otimes u_{-} \right) \cdot \nabla q_{-} \right] dz \right) \nonumber \\
&+   \mu_{-} \int_{\Omega_{-}^{\eta}(t)} \nablasym(u_{-}) \cdot \nabla q_{-} \, dz - \varepsilon_{0} \int_{0}^{\ell} \partial_{t}\eta \cdot \partial_{xx} \xi \, dx \\
&+  \rho_{s} \left( \dfrac{d}{dt} \int_{0}^{\ell} \partial_{t} \eta \cdot \xi \, dx - \int_{0}^{\ell} \partial_{t} \eta \cdot \partial_{t} \xi \, dx \right) + \int_{0}^{\ell} \left( \lambda_{0} \partial_{xx} \eta_{1} \partial_{xx} \xi_{1} + c_{1} \partial_{x} \eta_{1} \partial_{x} \xi_{1} \right) dx \nonumber \\
&+  c_{2} \int_{0}^{\ell} \partial_{xx} \eta_{2} \partial_{xx} \xi_{2} \, dx + 2\alpha \int_{0}^{\ell} \dfrac{\partial_{x} \eta_{1} \partial_{x} \xi_{1}}{| \partial_{x} \eta_{1} |^{2(\alpha + 1)}} \, dx dt =   \rho_{-} \int_{\Omega_{-}^{\eta}(t)} \left( f \cdot q_{-} \right) \, dz dt \nonumber
\end{align}
for all test functions $q_- \in H^1([0,T]; H^1(\Omega_-^\eta,\R^2)$ and $\xi \in V_k$ such that $q_-(t,\eta(t,x)) = \xi(t,x)$ for almost all $(t,x) \in [0,T) \times [0,\ell]$.

Additionally we claim that the energy inequality 
\begin{align} \label{oneSidedIneq}
\begin{aligned}
& \phantom{{}={}}\dfrac{\rho_{-}}{4} \| u_{-}(t) \|^{2}_{L^{2}(\Omega_{-}^{\eta}(t))}   + \dfrac{\mu_{-}}{2} \int_{0}^{t} \| \nablasym(u_{-})(s) \|^{2}_{L^{2}(\Omega_{-}^{\eta}(s))} \, ds + \dfrac{\rho_{s}}{2} \| \partial_{t} \eta(t) \|^{2}_{L^{2}(\Gamma^{\eta}(t))} +  \mathcal{E}_{K}(\eta)(t) \\
&  \leq  \dfrac{\rho_{-}}{2} \| u_{-}^{0} \|^{2}_{L^{2}(\Omega_{-})} + \dfrac{\rho_{s}}{2} \| v_{0} \|^{2}_{L^{2}(\Gamma)} + \mathcal{E}_{K}(\eta)(0) + \rho_{-} \int_{0}^{t} \| f(s) \|^{2}_{L^{2}(\Omega_{-}^{\eta}(s))} \, ds \text{ for all } t \in [0,T).
\end{aligned}
\end{align}
holds along the solution. 

To show this we now want to send $\mu_+ \to 0$ and $\rho_+ \to 0$ in order to obtain the vacuum-limit. We do so in such a way that $\frac{\rho_+}{\mu_+} \to 0$, which is necessary for dealing with the non-linear terms. 

Given $f \in L^{2}([0,T) \times \Omega;\mathbb{R}^2)$, consider a weak solution $(u_{\pm}^{(k)}, \eta^{(k)}) \in V_{S}$ of problem \eqref{weak0} where $\rho_-$ and $\mu_-$ are fixed and $\rho_+ := \frac{1}{k^2}$, $\mu_+ := \frac{1}{k}$. By Theorem \ref{maintheorem}
such solutions exist for all $k \in \N$. Additionally from the energy inequality we get the same uniform bounds as before, but we note that in particular
\begin{align*}
\sqrt{\rho_{+}} \, u_{+}^{(k)} \in L^{\infty}(0,T;L^{2}(\Omega_{+}^{\eta}(t);\mathbb{R}^2)) \text{ and } \sqrt{\mu_{+}} \, \nablasym(u_{+}^{(k)}) \in L^{2}(0,T;L^{2}(\Omega_{+}^{\eta}(t);\mathbb{R}^2)),
\end{align*}
are uniformly bounded. 
If we now take the limit $k\to \infty$, we again have converging subsequences such that
\begin{align*}
 \eta^{(k)} &\rightharpoonup \eta &&\text{ in } C_w([0,T];H^2((0,\ell);\R^2)) \\
 \partial_t \eta^{(k)} &\rightharpoonup \partial_t \eta &&\text{ in } L^2([0,T];H^1((0,\ell);\R^2)) \\
 u^{(k)}_{-} &\rightharpoonup u_{-} &&\text{ in } L^2([0,T];W_{div}^{1,2}(\Omega_{-};\R^2)) \\
\end{align*}

With this in hand, we can now consider the limit-equation using the same approach as before. 
In particular the strong convergence of $(\partial_t \eta^{(k)},u_{-}^{(k)})$ does follow by the very same arguments as in Subsection \ref{sec:aubin}. Indeed the interior parts are unchanged (as they were done for $u^\pm$ separately. But also the solid part related to $(\partial_t \eta^{(k)},\ext{\eta^{(k)}}(\partial_t \eta^{(k)}))$ can be handled analogously, as the extension relies on the uniform estimates of $\eta^{(k)}$ only (which are unchanged, except for the $\delta_0$-part), all terms connected to $u_+^{(k)}$ will vanish in a quantified and uniform way.

Hence we are left to pass to the limit with the equation.

First consider $q \in C^\infty([0,T]\times\Omega;\R^2)$ compactly supported in $[0,T] \times \Omega_\pm(t)$. There we can ignore all the solid-terms and all the terms linear in $u$ converge. In particular we have
\begin{align*}
 \abs{\int_{\Omega_+^\eta(t)} \rho_+ u_+ \cdot \partial_t q } &\leq \sqrt{\rho^+} \norm[\Omega_+^\eta(t)]{\sqrt{\rho_+} u_+} \norm[\Omega_+^\eta(t)]{\partial_t q} \to 0 \\
 \abs{\int_{\Omega_+^\eta(t)} \mu_+ \nablasym {u_+} \cdot \nablasym{q} } &\leq \sqrt{\mu_+} \norm[\Omega_+^\eta(t)]{\sqrt{\mu_+} \nablasym (u_+) } \norm[\Omega_+^\eta(t)]{\nablasym(q)} \to 0
\end{align*}
Using the classic Aubin-Lions lemma, we also get $u_\pm^{(k)} \to u_\pm$ in $L^2([0,T]\times \Omega)$. This implies
\begin{align*}
 \int_{\Omega_-^\eta(t)} \rho_- u_-^{(k)} \otimes u_-^{(k)} \cdot \nabla q \to  \int_{\Omega_-^\eta(t)} \rho_- u_- \otimes u_- \cdot \nabla q
\end{align*}
and in particular also
\begin{align*}
 \abs{\int_{\Omega_+^\eta(t)} \rho_+ u_+^{(k)} \otimes u_+^{(k)} \cdot \nabla q} \leq \rho_+ \norm[\Omega_+^\eta(t)]{u_+^{(k)}}^2 \norm[\infty]{\nabla q} \leq c \frac{\rho_+}{\mu_+} \norm[\Omega_+^\eta(t)]{ \sqrt{\mu_+} \nablasym (u_+^{(k)})}^2 \norm[\infty]{\nabla q} \to 0
\end{align*}
using that $\Omega_+^\eta(t)$ has a Korn-inequality with constants only depending on the bounds on $\mathcal{E}(\eta)$~\cite{DieRS10}.

Next we can proceed as in the previous proofs of convergence, by first constructing a relative pressure differential $P^{(k)}(t)$ using the almost solenoidal extension and then considering test functions on the solid, as well as their extension. Since this construction is mostly identical to our previous proofs, we omit further details.

Combining all possible test functions we finally get that our limit pair $(u_-,\eta)$ satisfies \eqref{weak00} for all pairs of test functions $(q,\xi)$ in $V_T^\eta$. Note that technically this space includes values for $q$ on $\Omega_+^\eta$, but while these do not occur in the equation, given $\xi$, a matching $q_+$ can always be found using the extension theorem.

\begin{remark}[Pressure in the one-sided model]
\label{rem:pres2}
 Similar to before, it is possible to reconstruct a pressure for the weak equation by effectively testing with vector-fields that are non-divergence free and then defining the pressure as an operator on their divergence. However now, since the test-function $q$ is no longer constrained to be zero on the whole boundary of $\partial \Omega$ but only on that part in common with $\partial \Omega_-$, as a consequence, we are no longer restricted to defining the pressure only on functions of mean zero. In other words, the pressure is now known fully, instead of only up to a constant.
 
 At first glance that might seem like a contradiction. While in the two sided model we had some additional information about the pressure differential between the two fluids mediated by the solid, the total pressure was still only known up to a constant. Yet now, after taking a limit this uncertainty suddenly seems resolved.
 
 The resolution is of course that reconstructing the pressure after the limit is not the same as taking the limit of reconstructed pressures. If we try the latter, similar to the previous proof, but on a formal level, then it turns out that not all terms related to $\Omega_+$ vanish in the weak equation. Instead the pressure term, which would be of the form $\int_{\Omega_+} p \nabla \cdot q \,dx$ has no dependence on either $\rho_+$ or $\mu_+$ and does not vanish once we send these two to zero. As a result the full limit equation including the pressure in strong form reads as $\nabla p = 0$ on $\Omega_+$.
 
 Unsurprisingly again, this allows us to only determine the total pressure up to a constant. However, it also tells us that the pressure in all of $\Omega_+$ is constant. As we assume that side to effectively be a vacuum, it is thus reasonable to set that constant to zero. This then results in the same absolute pressure as the one recovered from the weak equation.
 
 Consequentially the convergence in the proof of Theorem \ref{maintheorem2} should be seen first and foremost as a mathematical convergence with limited physical significance. This does however not diminish the physical meaning of the actual limit equation obtained.
\end{remark}

\section{Outlook: Non-flat reference geometries}
\label{ssec:outlook}

We have proven the existence of solutions for a rather specific configuration of a single beam stretched across a quadratic domain. In particular we wished to discuss a physical motion of a beam. We believe that the set up is general enough to allow to speculate for more settings to be treated by the here introduced method. %

To emphasize that, we discuss here the case when the solid is a union of curved beams. At each end they can be attached to the boundary of $\Omega$, to each other in a graph-like network, or they can be circular and thus free floating. Instead of a reference configuration it is thus better to describe them locally as curves. These might also partition $\Omega$ into a fixed number of subdomains instead of the $\Omega^\pm$ used previously.

We can then use this local parametrization to talk about the elastic energy. To account for curved beams, we now have to use the full tangential parametrization velocity $\abs{\partial_x \eta}$ instead of approximating it by $\partial_x \eta_1$. Similarly we have to work with the full curvature $\tfrac{\partial_x \eta \wedge \partial_{xx} \eta }{\abs{\partial_x \eta}}$ instead of $\partial_{xx} \eta_2$.  Specifically if a part of the solid is given locally by $\eta: I \to \Omega$ for some interval $I$ one would analogously consider an energy density of the form
\begin{align*}
E(x,\eta,\partial_x \eta, \partial_{xx} \eta) := c_h\abs{ \abs{\partial_x \eta} - \gamma}^2 + \frac{1}{\abs{\partial_x \eta}^{2\alpha}} + \frac{c_c}{2} \abs{ \frac{\partial_x \eta \wedge \partial_{xx} \eta }{\abs{\partial_x \eta}} - H}^2%
\end{align*}
where $\gamma: I \to \R$ with $\min_I \gamma > 0$ is a sufficiently smooth ``reference length'' and similarly $H: I \to \R$ is a reference curvature. Of course again much more general energies are possible in particular while only adding terms of lower order, but this choice has terms corresponding to the ones in our simplified model.

With the current method, again some regularizer needs to be added (e.g.\ $\frac{\lambda_0}{2} \abs{\frac{\partial_x \eta \cdot \partial_{xx} \eta}{\abs{\partial_x \eta}} }^2$). The deciding property for it is that we need to be able to estimate $\norm[L^2(I)]{\partial_x \abs{\partial_x \eta}}^2$ as this will allow us to prove a lower bound on $\abs{\partial_x \eta}$ in conjunction with the second term by exactly the same argument as before.

Such an energy and a configuration then would allow to construct the time-discrete approximations as before. Once this is done, one could use the quasi-linear nature of the resulting Euler-Lagrange equations to conclude convergence as was done in the flat case in this paper. 

What is more involving is proving that for a fixed initial configuration there is a minimal time until contact. While different parts of the solid can only come close to each other slowly and thus can be handled in essentially the same way as the solid touching $\partial \Omega$,  we now no longer can immediately exclude the possibility of a part of the solid bending back on itself, as we could by showing that $\partial_x \eta_1 > 0$ before. Here a geometrical argument is needed: If there are two points $x$ and $x' \in I$ such that $\eta(x) = \eta(x')$, then $\partial_x \eta$ needs to have changed its direction completely in the intervening interval $[x,x']$. As $\abs{\partial_x \eta}$ is uniformly bounded from below, this means that there are points $\hat{x},\hat{x}' \in [x,x']$ for which $\abs{\partial_x \eta(\hat{x}) -\partial_x \eta(\hat{x}')}$ is larger then twice this lower bound. But since we also control the $L^2$ norm of the second derivative, this means that $\abs{\hat{x}-\hat{x}'}$ and thus $\abs{x-x'}$ is uniformly (but energy dependent) bounded from below. In other words, by localizing, we can treat all potential collisions as collisions of different parts of the solid.

With this preparations one can follow the here outlined proof strategy: First, the variational scheme is absolutely capable of handling non-linear energies (that is exactly its huge advantage to other strategies). Further, since the geometry is uniformly non-degenerate a respective extension can be built and hence the compactness argument can be performed (for instance via a localisation). That then allows to pass to the limit and produce a solution.

Similar results can also be envisioned for 3d-fluids interacting with 2d shells, though in that case it is more complicated to correctly deal with local injectivity.

\subsection*{Declarations}
This work has been supported by the Primus research programme PRIMUS/19/SCI/01, the University Centre UNCE/SCI/023 of Charles University as well as the grants GJ19-11707Y of the Czech national grant agency (GA\v{C}R) and the first and second author acknowledge ECR-CZ grant LL2105. The second author wishes to thank the University of Vienna for their kind hospitality. The third Author acknowledges the support of the Fédération Wallonie - Bruxelles (ARC Advanced Grant at the Université Libre de Bruxelles). There is no conflict of interest.

\appendix

 \bibliographystyle{plainnat}
\bibliography{bib}

\subsection*{Appendix: A more thorough derivation of the reduced model for the solid}

In \cite{BenKamSch20} existence for a similar problem has already been shown in the case of a general bulk solid. As we are dealing with a more specific geometry, we have access to a few additional assumptions, which allow us to linearize some aspects of the model while at the same time sending the thickness to zero.

Consider deformations $\eta^{(\epsilon)}$ from the reference domain $Q_\epsilon := (0,\ell) \times (-\epsilon,\epsilon)$ to $\Omega$, which keep the lateral boundary $\{0,\ell\} \times (-\epsilon,\epsilon)$ constant as well as a hyperelastic energy of the form
\[ \fint_{Q_\epsilon} W_\epsilon(\nabla \eta^{(\epsilon)}) + H_\epsilon(\nabla^2 \eta^{(\epsilon)}) dx.\]

For now we will ignore the higher order term and focus on the first one. The simplest approach to modeling a solid would be as linear elastic, i.e. using a quadratic form $C(I-\nabla \eta^{(\epsilon)})$. However the key ingredient needed for fluid structure interaction is the ability to switch between Lagrangian and Eulerian description using the deformation. As such we in particular require that the resulting volume and surface elements are not degenerate; or translated to the specific problem, we in particular require that
\[W_\epsilon(A) \to \infty \text{ whenever } \det A \to 0 \text{ or } \abs{A} \to \infty. \] 
To additionally penalize the the first limit we then choose
\[W_\epsilon(A) = \frac{1}{(\det A)^{2\alpha}} + C( I-A)\]
for some $\alpha > 0$ to be determined later.

Now to take the limit $\epsilon \to 0$ we assume that our deformation is given by a thickened curve $\eta:(0,\ell) \times (-\epsilon,\epsilon) \to \Omega$, i.e.
\[\eta^{(\epsilon)}(x,y) := \eta(x) + n_\eta(x) y \]
where $n_\eta(x)$ is the upwards pointing unit normal vector to the curve $\eta$ at the point $x$. In particular we have 
\[\det(\nabla \eta^{(\epsilon)}) = \det \left(\partial_x \eta \middle| \frac{\partial_x \eta^\bot}{\abs{\partial_x \eta}}\right) =  \abs{\partial_x \eta} \approx \abs{\partial_x \eta_1}\]
as we assume that $\partial_x \eta \approx (1,0)$.
Together with the well known theory for linear elastic plates (see e.g.~\cite{CiarletBook2}) we then can take the formal limit
\[\lim_{\epsilon \to 0} \int_{Q_\epsilon} W_\epsilon (\nabla \eta_\epsilon) dx \approx \int_0^\ell \frac{1}{\abs{\partial_x \eta_1 }^{2\alpha}} + c_1\abs{\partial_x \eta_1-1}^2 + c_2\frac{\abs{\partial_{xx} \eta_2}^2}{2}  dx\] 
for constants $c_1,c_2 >0$.

This leaves us to deal with the higher order term. In \cite{BenKamSch20} this term was of the form $|\nabla^2 \eta|^q$, with $q$ larger than the bulk dimension which in our case would be two. However the main purpose of this term is to obtain uniform Hölder regularity in order to obtain a more uniform control on the other terms. As we are primarily interested in the limit model which is one dimensional it is enough to choose $q=2$. Additionally we already control one of the components of $\eta$ via the curvature term $c_c\abs{\partial_{xx} \eta_2}^2$, so we only need to add another term for $\partial_{xx} \eta_1$.

In total we then have the energy
\[\mathcal{E}_{K}(\eta) := \int_0^l \left(c_h\abs{\partial_x \eta_1-1}^2 + \frac{1}{\abs{\partial_x \eta_1}^{2\alpha}} + c_c\frac{\abs{\partial_{xx} \eta_2}^2}{2} + \lambda_0 \frac{\abs{\partial_{xx} \eta_1}^2}{2} \right)dx \]
with $\lambda_0 >0$ needed for additional regularisation. 

\vspace{18mm}

\begin{minipage}{80mm}
\textbf{Malte Kampschulte}\\
Department of Mathematical Analysis\\
Faculty of Mathematics and Physics\\
Charles University in Prague\\
Sokolovská 83\\
186 75 Prague - Czech Republic\\
{\bf E-mail}: kampschulte@karlin.mff.cuni.cz\\

\textbf{Gianmarco Sperone}\\
Département de Mathématique\\
Université Libre de Bruxelles\\
Boulevard du Triomphe 155\\
1050 Brussels - Belgium\\
{\bf E-mail}: gianmarco.sperone@ulb.be
\end{minipage}
\hfill
\begin{minipage}{80mm}
\textbf{Sebastian Schwarzacher}\\
Department of Mathematical Analysis\\
Faculty of Mathematics and Physics\\
Charles University in Prague\\
Sokolovská 83\\
186 75 Prague - Czech Republic\\

Department of Mathematics\\
Analysis and Partial Differential Equations\\
Uppsala University\\
L\"agerhyddsv\"agen 1\\
752 37 Uppsala - Sweden

{\bf E-mail}: schwarz@karlin.mff.cuni.cz,\\
sebastian.schwarzacher@math.uu.se
\end{minipage}

\vspace{1cm}
\begin{minipage}{80mm}
\end{minipage}

\end{document}